\documentclass{article}
\usepackage[english]{babel}
\usepackage{latexsym,amsmath,amsthm,amssymb}
\usepackage{mathrsfs,wasysym,graphicx,lscape}
\numberwithin{equation}{section}
\usepackage{algorithm2e}
\usepackage{esint}
\usepackage{color}
\usepackage[UTF8]{ctex}
\usepackage[color]{xy}
\usepackage{xcolor}

\usepackage[letterpaper,top=2cm,bottom=2cm,left=3cm,right=3cm,marginparwidth=1.75cm]{geometry}
\usepackage{amsmath}
\usepackage{graphicx}
\usepackage[colorlinks,linkcolor=blue, anchorcolor=blue, urlcolor=blue, citecolor=blue]{hyperref}
\usepackage{cleveref} 

\usepackage{mathrsfs}
\usepackage{graphics}
\usepackage{enumerate}
\title{\textbf{Long-time asymptotics and invariant manifold for the fractional 2D Navier-Stokes equation}}
\author{
  Jingchi Huang$^{2}$\footnotemark[1] \and
  Dandan Li$^{1}$\footnotemark[2] \and
  Shuai Zhang$^{2,1}$\footnotemark[3]
}

\date{}

\begin{document}
\newtheorem{theorem}{\indent Theorem}[section]
\newtheorem{proposition}[theorem]{\indent Proposition}
\newtheorem{definition}[theorem]{\indent Definition}
\newtheorem{lemma}[theorem]{\indent Lemma}
\newtheorem{remark}[theorem]{\indent Remark}
\newtheorem{corollary}[theorem]{\indent Corollary}
\maketitle

\footnotetext[1]{E-mail: huangjch25@mail.sysu.edu.cn}
\footnotetext[2]{Corresponding author. E-mail: dandanli@gbu.edu.cn}
\footnotetext[3]{E-mail: zhangsh278@mail2.sysu.edu.cn; 253621003@stu.gbu.edu.cn}

\maketitle    

\begin{center} 
      {\small 1. School of Science, Great Bay University, Dongguan, Guangdong, 523000, P.R. China} \\
      {\small 2.  School of Mathematics, Sun Yat-sen University, Guangzhou, Guangdong, 510006, P.R. China}
\end{center}

\begin{abstract}
   We consider the two-dimensional incompressible Navier–Stokes equations with supercritical fractional dissipation in the vorticity formulation. In self-similar variables, we analyze the linearized operator in weighted spaces, prove a spectral gap, and construct a finite-dimensional local slow invariant manifold for small solutions. As a consequence, solutions are attracted to this manifold and admit an explicit long-time asymptotic expansion determined by the leading eigenmodes; additional moment conditions yield faster decay. We also show Lipschitz dependence of the manifold on the dissipation exponent, and recover the classical Navier–Stokes dynamics in the limit as the exponent approaches one.
\end{abstract}

\section{Introduction}

    Fractional Laplacians arise naturally in transport and diffusion models where the microscopic dynamics involves long jumps or memory effects. For $\alpha\in(0,1)$, the operator $(-\Delta)^\alpha$ is the generator of a symmetric $\alpha$-stable L\'evy process and provides a continuum description of anomalous diffusion and nonlocal transport \cite{NL 00, DA 09}. In incompressible fluid models, fractional dissipation offers a flexible way to encode nonlocal energy transfer and subgrid-scale effects, and it appears, for instance, in quasi-geostrophic and turbulence modeling frameworks \cite{LA 10, JW 04}. Mathematically, Navier--Stokes equations with fractional dissipation form a one-parameter family interpolating between Euler dynamics (formally $\alpha=0$) and the classical Navier--Stokes system ($\alpha=1$), with $\alpha$ controlling the strength and range of dissipation and hence the balance between nonlinearity and diffusion.

    In this paper, we consider the fractional Navier-Stokes equation in $\mathbb{R}^{2}$ has the form
    \begin{align}\label{v}
        \frac{\partial\textbf{u}}{\partial t}+(\textbf{u}\cdot \nabla)\textbf{u}=-(-\Delta)^{\alpha}\textbf{u}-\nabla p,\quad \nabla\cdot \textbf{u}=0,
    \end{align}
    where $\textbf{u}=\textbf{u}(x,y)\in \mathbb{R}^{2}$ is the velocity field, $p=p(x,t)\in \mathbb{R}$ is the pressure field, $\alpha\in (0,1)$ and $x\in \mathbb{R}^{2}$, $t\geq 0$. We set $\omega=\text{curl}\ \boldsymbol{u}=\partial_{x_{1}}u_{2}-\partial_{x_{2}}u_{1}$ to obtain the fractional vorticity equation 
    \begin{align}\label{ve}
        \partial_{t}\omega+(u\cdot \nabla) \omega =-(-\Delta)^{\alpha}\omega,
    \end{align}
    where $\omega=\omega(x,t)\in \mathbb{R}$, $x=(x_{1},x_{2})\in \mathbb{R}^{2}$, $t\geq 0$, $\alpha\in(0,1)$. In this paper, we primarily study the vorticity formulation \eqref{ve} rather than the primitive equation \eqref{v}, noting that the velocity field $\boldsymbol{u}$ can be recovered via the Biot-Savart law.

    Before turning to long-time asymptotics, we note that global well-posedness and regularity for fractional Navier-Stokes have been widely studied; see \cite{M.J 21, C.S 01, A.M 17, K.R 00, J.L 23,  H.W 22, Y.L 21, A.J 25, T.Y 15, L.E 20, T.T 09, X.J 22, X.Z 12, TC 021}. Within our framework we also prove global existence and uniqueness for sufficiently small initial data, using weighted a priori estimates adapted to fractional dissipation. Our main goal, however, is a precise description of the large-time dynamics. At the linear level, $(-\Delta)^\alpha$ generates $2\alpha$-stable semigroups whose kernels decay like $t^{-n/2\alpha}$, in contrast to the classical Gaussian heat kernel decay \cite{Z.T 03}, leading to algebraic decay bounds for weak solutions, including optimal $L^{2}$ estimates \cite{ZL 24}, and to forward self-similar solutions with universal profiles in hypodissipative and critical regimes \cite{B.C 19, B 24}. Yet recent sharp non-uniqueness results for Leray--Hopf solutions even in hyperdissipative settings \cite{C.S.M 23, L 24} suggest that the weak-solution class is too broad for a clean asymptotic theory. We therefore focus on global small smooth solutions and describe their long-time behaviour geometrically via invariant slow manifolds for the rescaled fractional flow.

    In the classical case $\alpha=1$, the long-time dynamics of 2D Navier--Stokes are well understood from a dynamical-systems viewpoint \cite{TC 02, TC 05, GR 02}. In self-similar vorticity variables the equation becomes an infinite-dimensional flow driven by an Ornstein--Uhlenbeck operator; its discrete spectrum yields finite-dimensional local invariant manifolds, and small trajectories are rapidly attracted to a slow manifold that governs the asymptotics. 

    The present work extends this program to the fractional Navier--Stokes equations with $\alpha\in(1/2,1)$. The extension is nontrivial because the linear part is a nonlocal operator whose spectrum depends on $\alpha$, requiring an $\alpha$-dependent spectral analysis. Moreover, Gaussian-kernel arguments for $L^{p}$ decay no longer apply; instead we combine a fractional Moser iteration (Theorem \ref{th:3.4}) with weighted a priori estimates based on Stroock--Varopoulos type inequalities (See Appendix C) to obtain global well-posedness for small data (Theorem \ref{global}) and sharp decay rates. A key qualitative difference concerns resonances. For $\alpha=1$, the evenly spaced Ornstein--Uhlenbeck spectrum produces resonances $\lambda_i+\lambda_j=\lambda_k$ and hence logarithmic corrections $t^{-\gamma}\log t$ in decay \cite{TC 02, TC 05}. In the fractional setting the eigenvalues depend nonlinearly on $\alpha$; thus $\alpha$ generically breaks the arithmetic progression and eliminates resonant logarithmic terms, yielding purely algebraic decay. This "de-resonating" mechanism may also help explain the smoother waveforms often observed in fractional models, e.g.\ in hemodynamics \cite{JMB 15, BJ 07}.

    In contrast to the classical case where the dimension of the slow subspace can be increased by selecting higher weights, in our regime the dimension is strictly limited and cannot be made arbitrarily large. Although this precludes capturing arbitrarily many algebraic modes within a single weighted space, it already reveals a robust structural difference from the classical setting and suffices to recover the leading and subleading algebraic decay rates, together with sharp threshold conditions in terms of vanishing moments.

    Geometrically, our approach is based on invariant slow manifolds for the rescaled fractional vorticity equation. In self-similar variables, the linearized dynamics are generated by the fractional operator $\mathcal L_\alpha$ on weighted spaces $L^2(m\alpha)$. For each $\alpha\in(1/2,1)$ one can choose $m$ and $k\ge0$ so that the spectrum splits into finitely many "slow" eigenvalues with $\Re\lambda\ge -k$ and a "fast" part strictly to the left of $-k$, yielding a decomposition $L^2(m\alpha)=E_{\mathrm{d}}\oplus E_{\mathrm{c}}$. We then construct a local slow manifold $W_{\mathrm{d}}^{\mathrm{loc}}=\text{graph}\ g(w)$ over $E_{\mathrm{d}}$ and show that all sufficiently small solutions are attracted to it at a rate at least $t^{-k}$. In this picture, algebraic decay exponents correspond to contraction rates in the fast directions, while the reduced dynamics on $W_{\mathrm{d}}^{\mathrm{loc}}$ determines the universal leading profile and permits higher-order asymptotic expansions, in the spirit of \cite{TC 02}.

    On the other hand, optimal algebraic decay for general Leray solutions is still unknown. In \cite{ZL 24}, one only obtains the bound $\limsup_{t\to\infty} t^{1/\alpha}\|u(t)\|_{L^2}\le C$, so even for small data the rescaled quantity $t^{1/\alpha}\|u(t)\|_{L^2}$ is merely bounded. Restricting to small smooth solutions, we prove in the regime $3/4<\alpha<1$ (with $m$ close to $4$) a sharp, moment-based criterion: suitable vanishing moments imply $\lim_{t\to\infty} t^{1/\alpha}\|u(t)\|_{L^2}=0$, and more generally yield the optimal decay rate associated with any prescribed set of vanishing moments (Theorem \ref{opt}). Geometrically, these conditions place the initial data in the fast spectral subspace, leading to decay strictly faster than the generic $t^{-1/\alpha}$ rate; the slow manifold is the threshold separating this optimal faster decay from the generic slow one.

    A further ingredient is stability with respect to $\alpha$. Using an $\alpha$-dependent spectral decomposition and invoking Kato's spectral perturbation theory (see \cite{TK 95}) to construct bounded linear isomorphisms that identify the varying subspaces $E_{\mathrm{d}}(\alpha)$ and $E_{\mathrm{c}}(\alpha)$ with fixed reference spaces, we show that the graph map $g_\alpha$ depends continuously on $\alpha$ in a Banach space of Lipschitz functions. In particular, $W_{\mathrm{d}}^{\mathrm{loc}}(\alpha)$ forms a robust family for $\alpha\in(1/2,1)$, and $W_{\mathrm{d}}^{\mathrm{loc}}(\alpha)\to W_{\mathrm{d}}^{\mathrm{loc}}(1)$ in graph norm as $\alpha\to1$, recovering the classical slow manifold of \cite{TC 02}. Indeed, our framework ensures that all derived objects—including the spectrum, invariant manifolds, and asymptotic expansions-converge to the corresponding results of the classical Navier-Stokes equations as the fractional Laplacian approaches the classical one. Thus the fractional manifolds can be viewed as nonlocal perturbations of the integer-order one, and the large-time scenario for small solutions persists continuously as $\alpha$ varies.

    The paper is organized as follows. Section 2 outlines the necessary preliminaries on function spaces and technical tools. In Section 3, we begin by surveying the (well-developed) existence and uniqueness theory for 2D the fractional vorticity equation. In Section 4, we introduce self-similar variables and the weighted spaces $L^{2}(m\alpha)$ to derive the rescaled equation, and establish global existence and uniqueness for the fractional vorticity equation within the present functional framework. We then construct the finite-dimensional invariant slow manifolds associated with prescribed decay rates (Theorem \ref{invarint}) and provide a geometric characterization of solutions decaying faster than the dominant algebraic rate (Theorem \ref{loc-strong}). We conclude this section by proving the structural stability of these manifolds with respect to $\alpha$. Section 5 applies this geometric framework to the long-time dynamics of small solutions, establishing the attraction to the slow manifold and deriving asymptotic expansions. Finally, technical details are relegated to the appendices: Appendix A covers the self-similar transform, Appendix B details the spectral analysis of $\mathcal L_{\alpha}$, and Appendix C contains auxiliary lemmas and estimates.


\section{Preliminaries}

    \textbf{Fraction Laplacian operator}:  Let \(\mathcal S(\mathbb R^N)\) denote the Schwartz space of rapidly decaying \(C^\infty\) functions on \(\mathbb R^N\). Throughout the paper we use the following convention for the Fourier transform:
    \begin{align*}
        \widehat{\phi}(\xi)=\mathcal F[\phi](\xi):=\int_{\mathbb R^N}e^{-i x\cdot \xi}\phi(x)\mathrm{d}x,
    \end{align*}
    and 
    \begin{align*}
        \phi(x)=\mathcal F^{-1}[\widehat{\phi}](x):=\frac{1}{(2\pi)^N}\int_{\mathbb R^N}e^{i x\cdot \xi} \widehat{\phi}(\xi)\mathrm{d}\xi .
    \end{align*}
    With this convention, Plancherel's identity takes the form
    \begin{align*}
        \|\phi\|_{L^2(\mathbb R^N)}^2=\frac{1}{(2\pi)^N} \|\widehat{\phi}\|_{L^2(\mathbb R^N)}^2 .
    \end{align*}
    In particular, all constants below may absorb harmless factors depending only on \(N\).

    Let $\Omega$ be a general, possibly nonsmooth and open set in $\mathbb{R}^{N}$. For any real $\alpha\in (0,1)$ and $p\in [1,\infty)$, we define the fractional Sobolev spaces $W^{\alpha,p}(\Omega)$ as follows 
    \begin{align}
        W^{\alpha,p}(\Omega)=\bigg\{u\in L^{p}(\Omega): \frac{|u(x)-u(y)|}{|x-y|^{\frac{N}{p}+\alpha}}\in L^{p}(\Omega\times \Omega)\bigg\};
    \end{align}
    endowed with the natural norm
    \begin{align}
        \|u\|_{W^{\alpha,p}(\Omega)}:=\bigg(\int_{\Omega}|u|^{p}dx+\int_{\Omega}\int_{\Omega}\frac{|u(x)-u(y)|^{p}}{|x-y|^{N+\alpha p}}\mathrm{d}x\mathrm{d}y\bigg)^{\frac{1}{p}},
    \end{align}
    where the term 
    \begin{align}\label{seminorm}
        [u]_{W^{\alpha,p}(\Omega)}:=\bigg(\int_{\Omega}\int_{\Omega}\frac{|u(x)-u(y)|^{p}}{|x-y|^{N+\alpha p}}\mathrm{d}x\mathrm{d}y\bigg)^{\frac{1}{p}}
    \end{align}
    is the so-called Gagliardo semi-norm of $u$, see \cite{EG 12}.

    In this paper, we focus on the case $p=2$. This is quite an important case since the fractional Sobolev spaces $W^{\alpha,2}(\mathbb{R}^{N})$ turn out to be Hilbert spaces. It is usually denoted by $H^{\alpha}(\mathbb{R}^{N})$. Moreover, it is strictly related to the fractional Laplacian operator $(-\Delta)^{\alpha}$, where, for any $u\in \mathcal{S}(\mathbb{R}^{N})$ and $\alpha\in (0,1)$, $(-\Delta)^{\alpha}$ it is defined as
    \begin{align}\label{F-f}
        (-\Delta)^{\alpha}u(x)&=C(N,\alpha)\text{P.V.} \int_{\mathbb{R}^{N}}\frac{u(x)-u(y)}{|x-y|^{N+2\alpha}}\mathrm{d}y\notag\\
        &=C(N,\alpha)\lim\limits_{\varepsilon\to 0^{+}}\int_{\mathbb{R}^{N}\backslash B_{\varepsilon}(x)}\frac{u(x)-u(y)}{|x-y|^{N+2\alpha}}\mathrm{d}y,
    \end{align}
    where $\text{P.V.}$ is a commonly used abbreviation for in the principal value sense and $C(N,\alpha)$ is dimensional constant that depends on $N$ and $\alpha$.
    \begin{remark}
        We also write \eqref{F-f} as 
        \begin{align*}
            (-\Delta)^{\alpha}u(x)=-\frac{1}{2}C(N,\alpha)\int_{\mathbb{R}^{N}}\frac{u(x+y)+u(x-y)-2u(x)}{|y|^{N+2\alpha}}\mathrm{d}y,\quad \forall x\in \mathbb{R}^{N},
        \end{align*}
        here we remove the $\mathrm{P.V.}$ Furthermore, we can infer Fourier from of the fractional Laplacian operator. 
        \begin{align}\label{fd}
            (-\Delta)^{\alpha}u=\mathcal{F}^{-1}(|\xi|^{2\alpha}(\mathcal{F}u)),\quad \forall \xi\in \mathbb{R}^{N}.
        \end{align}
    \end{remark}

    \begin{remark}
        Although \eqref{F-f}--\eqref{fd} are first defined for $u\in\mathcal {S}(\mathbb {R}^{N})$, they extend in a standard way to non-Schwartz functions by density or duality. This is  equivalent to defining the operator as a Fourier multiplier on $\mathcal {S}'(\mathbb {R}^{N})$; throughout this paper $(-\Delta)^\alpha u$ is understood in this extended distributional sense.
    \end{remark}

    \textbf{Notation}. Throughout the paper we adopt the following conventions. Vector–valued functions (for instance the velocity field) are written in boldface, such as $\mathbf{u}(x,t)$, while spatial points in $\mathbb{R}^2$ are denoted by standard italics, $x=(x_1,x_2)$. In either case $|\cdot|$ stands for the Euclidean norm on $\mathbb{R}^2$.  For $1\le p\le\infty$ we write $|f|_{p}$ for the norm of a scalar function $f$ in $L^p(\mathbb{R}^2)$, and if $\textbf{f}\in (L^{p}(\mathbb{R}^{2}))^{2}$ is vector–valued we set $|\textbf{f}|_{p}=|\ |\textbf{f}|\ |_{p}$. Weighted norms will play a central role. We introduce the weight $b(x) = \bigl(1+|x|^2\bigr)^{1/2}, x\in\mathbb{R}^2$, and for any $m\ge0$ define $\|f\|_{m\alpha} = |b^{m\alpha}f|_{2}$, so that $L^2(m\alpha)$ denotes the Hilbert space obtained as the completion of $C_c^\infty(\mathbb{R}^2)$ under the norm $\|\cdot\|_{m\alpha}$.  If $f\in C^0([0,T],L^p(\mathbb{R}^2))$ we often abbreviate $f(\cdot,t)$ simply by $f(t)$. Finally, the symbol $C$ denotes a positive constant whose value may change from line to line, even within the same chain of inequalities.

\section{The Cauchy problem of the vorticity equation}

    In two dimensions, the velocity field $\textbf{u}$ is defined in terms of the vorticity via the Biot-Savart law
    \begin{align}\label{B-S}
        \textbf{u}(x)=\frac{1}{2\pi}\int_{\mathbb{R}^{2}}\frac{(\textbf{x}-\textbf{y})^{\perp}}{|x-y|^{2}}\omega(y)\mathrm{d}y,\quad x\in \mathbb{R}^{2}.
    \end{align}
    Here and in what follows, if $x=(x_{1},x_{2})\in \mathbb{R}^{2}$, we define $\mathbf{x}=(x_{1},x_{2})^{\text{T}}$ and $\mathbf{x}^{\perp}=(-x_{2},x_{1})^{\text{T}}$.

    The following lemma collects useful estimates for the velocity $\textbf{u}$ in terms of $\omega$.
    \begin{lemma}\label{sol-e}
        Let $\textbf{u}$ be the velocity field obtained from $\omega$ via the equality \eqref{B-S}.
    
        \begin{itemize}
            \item [$(a)$] Assume that $1<p<2<q<\infty$ and $\frac{1}{q}=\frac{1}{p}-\frac{1}{2}$. If $\omega\in L^{p}(\mathbb{R}^{2})^{2}$, then $\textbf{u}\in L^{q}(\mathbb{R}^{2})^{2}$, and there exists $C>0$ such that 
        \begin{align}\label{uw}
            |u|_{q}\leq C|\omega|_{p}.
        \end{align}
            \item  [$(b)$] Assume that $1\leq p<2<q\leq \infty$, and define $\beta \in (0,1)$ by the relation $\frac{1}{2}=\frac{\beta}{p}+\frac{1-\beta}{q}$. If $\omega \in L^{p}(\mathbb{R}^{2})\ \cap\ L^{q}(\mathbb{R}^{2})$, and there exists $C>0$ such that 
        \begin{align}\label{uiw}
            |u|_{\infty} \leq C|\omega|^{\beta}_{p}|\omega|_{q}^{1-\beta}.
        \end{align}
            \item  [$(c)$] Assume that $1<p<\infty$. If $\omega\in L^{p}(\mathbb{R}^{2})$, then $\nabla\boldsymbol{u}\in L^{p}(\mathbb{R}^{2})^{4}$ and there exists $C>0$ such that
        \begin{align}\label{nu}
            |\nabla\boldsymbol{u}|_{p}\leq C|\omega|_{p}.
        \end{align}
        \end{itemize}
        In addition, $\mathrm{div}$ $\textbf{u}=0$ and $\mathrm{curl}$ $\textbf{u}\equiv \partial_{1}u_{2}-\partial_{2}u_{1}=\omega$.
        \end{lemma}
    
    \begin{proof}

        As the Biot-Savart law \eqref{B-S} is defined independently of the fractional dissipation, the estimate techniques for (a)-(c) extend naturally from the classical case (cf. \cite[Lemma 2.1]{TC 02}). The detailed proofs are therefore omitted.
    \end{proof}

   The following theorem addresses the global well-posedness of \eqref{ve} in $L^{1}(\mathbb{R}^{2})$.
    
    \begin{theorem}\label{gwp}
        For all initial data $\omega_{0}\in L^{1}(\mathbb{R}^{2})$, $1/2<\alpha<1$, \eqref{ve} has a unique global solution $\omega\in C^{0}([0,\infty),L^{1}(\mathbb{R}^{2}))\cap C^{0}((0,\infty),L^{\infty}(\mathbb{R}^{2}))$ with $\omega(0)=\omega_{0}$. Moreover, for all $p\in [1,+\infty)$, there exists $C_{p,\alpha}>0$ such that 
        \begin{align}\label{ww0}
            |\omega(t)|_{p}\leq \frac{C_{p,\alpha}|\omega_{0}|_{1}}{t^{\frac{1}{\alpha}(1-\frac{1}{p})}},\quad  t>0.
        \end{align}
        Finally, the total mass of $\omega$ is preserved under the evolution:
        \begin{align*}
            \int_{\mathbb{R}^{2}}\omega(x,t)\mathrm{d}x=\int_{\mathbb{R}^{2}}\omega_{0}\mathrm{d}x,\quad t\geq 0.
        \end{align*}
    \end{theorem}
    \begin{proof}
        Throughout the proof, $C$ denotes a positive constant depending only on $\alpha$ and the dimension.

        For any $\phi\in L^{1}(\mathbb{R}^{2})\cap \dot H^{\alpha}(\mathbb{R}^{2})$,
        \begin{equation}\label{eq:Nash}
            |\phi|_{2}^{2}\le C\,
            \\|(-\Delta)^{\alpha/2}\phi|_{2}^{\frac{2}{1+\alpha}}\, |\phi|_{1}^{\frac{2\alpha}{1+\alpha}} \qquad\Longleftrightarrow\qquad |(-\Delta)^{\alpha/2}\phi|_{2}^{2}\ge c\,\frac{|\phi|_{2}^{2(1+\alpha)}}{|\phi|_{1}^{2\alpha}} .
        \end{equation}
        Indeed, by Plancherel and a frequency split, for $R>0$,
        \begin{align*}
            |\phi|_2^2=\int_{|\xi|\le R}|\hat\phi|^2+\int_{|\xi|\ge R}|\hat\phi|^2 \le C\Big(R^2|\phi|_1^2+R^{-2\alpha}|(-\Delta)^{\alpha/2}\phi|_2^2\Big),
        \end{align*}
        and optimizing in $R$ gives \eqref{eq:Nash}.
        
        For $q\ge2$,
        \begin{equation}\label{eq:SV}
            \int_{\mathbb{R}^{2}}|\omega|^{q-2}\omega\,(-\Delta)^{\alpha}\omega\mathrm{d}x \ \ge\ \frac{4(q-1)}{q^{2}}\, \big|(-\Delta)^{\alpha/2}\big(|\omega|^{q/2}\big)\big|_{2}^{2}.
        \end{equation}
        This follows from the bilinear form of $(-\Delta)^\alpha$ and the scalar inequality
        \begin{align*}
            \big(|a|^{q-2}a-|b|^{q-2}b\big)(a-b)\ \ge\ \frac{4(q-1)}{q^{2}}\big(|a|^{q/2}-|b|^{q/2}\big)^2, \qquad a,b\in\mathbb R.
        \end{align*}

        Let $y_q(t):=|\omega(t)|_q^q$. Multiplying $\partial_t\omega+u\cdot\nabla\omega+(-\Delta)^\alpha\omega=0$ by $|\omega|^{q-2}\omega$ and integrating, the transport term vanishes since $\nabla\!\cdot u=0$, hence
        \begin{align*}
            \frac{\mathrm{d}}{\mathrm{d}t}y_q(t) =-\,q\int_{\mathbb{R}^2}|\omega|^{q-2}\omega\,(-\Delta)^\alpha\omega\mathrm{d}x.
        \end{align*}
        Using \eqref{eq:SV} and then applying \eqref{eq:Nash} to $\phi=|\omega|^{q/2}$ such that $|\phi|_2^2=y_q$ and $|\phi|_1=|\omega|_{q/2}^{q/2}$, we obtain
        \begin{equation}\label{eq:ODE}
            \frac{d}{dt}y_q(t)\le -c_\alpha\,\frac{y_q(t)^{1+\alpha}}{|\omega(t)|_{q/2}^{\alpha q}}, \qquad q\ge2.
        \end{equation}

        For $q=2$, using $|\omega(t)|_1\le |\omega_0|_1$ in \eqref{eq:ODE} gives
        \begin{align*}
            \frac{d}{dt}y_2(t)\le -c_\alpha\,|\omega_0|_1^{-2\alpha}\,y_2(t)^{1+\alpha}.
        \end{align*}
        Equivalently $(y_2^{-\alpha})'(t)\ge c'_\alpha|\omega_0|_1^{-2\alpha}$, hence
        \begin{equation}\label{eq:L2-decay}
            |\omega(t)|_2^2=y_2(t)\le C_\alpha |\omega_0|_1^2\,t^{-1/\alpha},\qquad t>0.
        \end{equation}

        Fix $q=2^n\ge2$. Assume that for some $A_\alpha$ independent of $q$,
        \begin{equation}\label{eq:IH-unif}
            |\omega(t)|_{q/2}\le A_\alpha|\omega_0|_1\,t^{-(1-2/q)/\alpha},\qquad t>0.
        \end{equation}
        Plugging \eqref{eq:IH-unif} into \eqref{eq:ODE} yields
        \begin{align*}
            y_q'(t)\le -\kappa_\alpha\,A_\alpha^{-\alpha q}|\omega_0|_1^{-\alpha q}\,t^{q-2}\,y_q(t)^{1+\alpha}.
        \end{align*}
        Let $Z_q(t):=y_q(t)^{-\alpha}$, then $Z_q'(t)\ge \kappa_\alpha\,A_\alpha^{-\alpha q}|\omega_0|_1^{-\alpha q}t^{q-2}$. Integrating over $[t/2,t]$ gives
        \begin{align*}
            y_q(t)^{-\alpha}\ge C\,A_\alpha^{-\alpha q}|\omega_0|_1^{-\alpha q}\,t^{q-1},
        \end{align*}
        hence
        \begin{align*}
            |\omega(t)|_q =y_q(t)^{1/q} \le C^{1/q}A_\alpha|\omega_0|_1\,t^{-(1-1/q)/\alpha}.
        \end{align*}
        Enlarging $A_\alpha$ to absorb $\sup_{q\ge2}C^{1/q}<\infty$, we obtain the uniform estimate
        \begin{equation}\label{eq:Lq-unif}
            |\omega(t)|_q \le A_\alpha|\omega_0|_1\,t^{-(1-1/q)/\alpha},\qquad q=2^n.
        \end{equation}
        Letting $q\to\infty$ in \eqref{eq:Lq-unif} yields $|\omega(t)|_\infty\le A_\alpha|\omega_0|_1\,t^{-1/\alpha}$. Finally, by interpolation with $|\omega(t)|_1\le|\omega_0|_1$, we get \eqref{ww0} for all $p\in[1,\infty]$.
    \end{proof}

    If $\omega(t)$ is the solution of \eqref{ve} given by Theorem \ref{gwp}, it follows form Lemma \ref{sol-e} that the velocity field $\textbf{u}(t)$ constructed from $\omega(t)$ satisfies $\textbf{u}(t)\in C^{0}((0,\infty),L^{q}(\mathbb{R}^{2})^{2})$ for all $q\in (2,\infty]$, and that there exist constants $C_{q}>0$ such that
    \begin{align}\label{u-sol-e}
        |\textbf{u}(t)|_{q}\leq \frac{C_{q,\alpha}|\omega_{0}|_{1}}{t^{\frac{1}{\alpha}(\frac{1}{2}-\frac{1}{q})}}, \quad t>0.
    \end{align}
    \begin{remark}
        Stronger integrability conditions on the initial data yield explicit decay estimates for the solution. More precisely, for the 2D fractional Navier-Stokes equations with exponent $\alpha \in (0,1)$, Theorem 1.2 of \cite{ZL 24} states that if $\boldsymbol{u}_{0} \in L^{1}(\mathbb{R}^{2}) \cap H^{1}(\mathbb{R}^{2})$, then the solution satisfies
        \begin{align*}
            \boldsymbol{u} \in L^{\infty}(0,\infty; H^{1}(\mathbb{R}^{2})) \cap L^{2}_{\text{loc}}(0,\infty; H^{1+\alpha}(\mathbb{R}^{2}))
        \end{align*}
        and admits the decay rate
        \begin{align*}
            |\boldsymbol{u}(t)|_{r} \leq C (t+1)^{-\frac{1}{\alpha}(1-\frac{1}{r})}, \qquad 1 \leq r < \infty.
        \end{align*}
    \end{remark}

    In the general setting of Theorem 3.2, we have the following result concerning the asymptotic limits.
    \begin{theorem}\label{th:3.4}
        Assume that $\textbf{u}_{0}\in L^{2}(\mathbb{R}^{2})^{2}\cap L^{1}(\mathbb{R}^{2})$ and that $\omega_{0} \in L^{1}(\mathbb{R}^{2})$. Let $\omega(t)$ be the solution of \eqref{ve} given by Theorem \ref{gwp}. Then 
        \begin{align}\label{w-d}
            \lim\limits_{t\to \infty}t^{\frac{1}{\alpha}(1-\frac{1}{p})}|\omega(t)|_{p}=0,\quad 1< p\leq \infty.
        \end{align}
        If $\textbf{u}(t)$ is the velocity field obtained from $\omega(t)$ via the \eqref{B-S}, then
        \begin{align}\label{u-d}
            \lim\limits_{t\to \infty}t^{\frac{1}{\alpha}(\frac{1}{2}-\frac{1}{q})}|\textbf{u}(t)|_{q}=0,\quad 2\leq q\leq \infty.
        \end{align}
    \end{theorem}
    \begin{proof}
        Let $E(t):=|\omega(t)|_{2}^{2},\ \widehat{\omega}(\xi,t)=\mathcal{F}[\omega(\cdot,t)](\xi).$
        Multiply both sides of the equation \eqref{ve} by $\omega(x,t)$ and integrate over the entire space $\mathbb{R}^{2}$, $\int_{\mathbb{R}^{2}}\omega\partial_{t}\omega \mathrm{d}x+\int_{\mathbb{R}^{2}}\omega(\textbf{u}\cdot \nabla)\omega \mathrm{d}x=-\int_{\mathbb{R}^{2}}\omega(-\Delta)^{\alpha}\omega \mathrm{d}x$. A direct calculation yields
        \begin{align}\label{eq:en}
            \frac{\mathrm{d}}{\mathrm{d}t}E(t)=-2|(-\Delta)^{\alpha/2}\omega|_{2}^{2}  \quad\Longleftrightarrow\quad \frac{\mathrm{d}}{\mathrm{d}t}|\widehat{\omega}|^{2}_{2}+2\int_{\mathbb{R}^{2}}|\xi|^{2\alpha}|\widehat{\omega}|^{2}\mathrm{d}\xi=0.
        \end{align}
        The function $t\mapsto |\omega(t)|_{2}$ is non-increasing. Take the time-dependent radius $R(t)=\kappa(1+t)^{-\frac{1}{2\alpha}}, \kappa>0$, and split the frequency space into $B(t)=\{|\xi|\leq R(t)\}$, $B(t)^{c}=\{|\xi|>R(t)\}$. Then 
        \begin{align*}
            \int_{\mathbb{R}^{2}}|\xi|^{2\alpha}|\widehat{\omega}|^{2}\mathrm{d}\xi\geq R(t)^{2\alpha}\int_{|\xi|>R(t)}|\widehat{\omega}|^{2}\mathrm{d}\xi=R^{2\alpha}(t)|\widehat{\omega}|_{2}^{2}-R^{2\alpha}(t)\int_{|\xi|\leq R(t)}|\widehat{\omega}|^{2}\mathrm{d}\xi.
        \end{align*}
        By Fourier transform, we obtain $\widehat{\omega}=\widehat{\partial _{1}u_{2}}-\widehat{\partial_{2}u_{1}} =i\xi^{\perp}\cdot\widehat{\boldsymbol{u}}(\xi)$. Thus, we have
        \begin{align*}
            \int_{|\xi|\leq R(t)}|\widehat{\omega}|^{2}\mathrm{d}\xi\leq R^{2}(t)\int_{|\xi|\leq R(t)}|\widehat{\boldsymbol{u}}|^{2}\mathrm{d}\xi\leq R^{2}(t)|\widehat{\boldsymbol{u}}|_{2}^{2}=R^{2}(t)|\boldsymbol{u}(t)|_{2}^{2}.
        \end{align*}
        Plugging these bounds into \eqref{eq:en} gives
        \begin{align}\label{eq:in-en}
            \frac{\mathrm{d}}{\mathrm{d}t}E(t)\leq -2R^{2\alpha}(t)E(t)+2R^{2\alpha+2}(t)|\boldsymbol{u}(t)|_{2}^{2}.
        \end{align}
        Define $F(t):=(1+t)^{\frac{1}{\alpha}}E(t)$. Then $F'(t)=\frac{1}{\alpha}(1+t)^{\frac{1}{\alpha}-1}E(t)+(1+t)^{\frac{1}{\alpha}}E'(t)$. Using \eqref{eq:in-en} we obtain
        \begin{align*}
            F'(t)&\leq \frac{1}{\alpha}(1+t)^{\frac{1}{\alpha}-1}E(t)+(1+t)^{\frac{1}{\alpha}}\big[-2R^{2\alpha}(t)E(t)+2R^{2\alpha+2}(t)|\boldsymbol{u}(t)|_{2}^{2}\big]\\
            &=\bigg(\frac{1}{\alpha}-2R^{2\alpha}(t)(1+t)\bigg)\frac{F(t)}{1+t}+2(1+t)^{\frac{1}{\alpha}}R^{2\alpha+2}(t)|\boldsymbol{u}(t)|_{2}^{2}.
        \end{align*}
        Since $R(t)=\kappa(1+t)^{-\frac{1}{2\alpha}}$, we have $R^{2\alpha}(t)(1+t)=\kappa^{2\alpha}$. Hence
        \begin{align}\label{eq:in-en1}
            F'(t)\leq -\bigg(2\kappa^{2\alpha}-\frac{1}{\alpha}\bigg)\frac{F(t)}{1+t}+C\frac{\kappa^{2\alpha+2}}{1+t}|\boldsymbol{u}(t)|_{2}^{2},
        \end{align}
        where $C=2$. Choose $\kappa$ large so that $a:=2\kappa^{2\alpha}-\frac{1}{\alpha}>0$. Then \eqref{eq:in-en1} becomes
        \begin{align}\label{eq:in-en-f}
            F'(t)+a\frac{F(t)}{1+t}\leq \frac{\varepsilon(t)}{1+t},\quad \varepsilon(t)=C\kappa^{2\alpha+2}|\boldsymbol{u}(t)|_{2}^{2}.
        \end{align}
        It is established in \cite{ZL 24} that $|\boldsymbol{u}(t)|_{2} \to 0$ as $t \to \infty$, hence $\varepsilon(t) \to 0$. Applying standard ODE comparison arguments to the differential inequality for $F(t)$ with $a>0$, we obtain $\lim_{t\to\infty} F(t) = 0$. 
        Hence 
        \begin{align*}
            F(t)=(1+t)^{\frac{1}{\alpha}}|\omega(t)|_{2}^{2}\to 0,\  \text{equivalently}\quad t^{\frac{1}{2\alpha}}|\omega(t)|_{2}\to 0.
        \end{align*}
        This proves the case $p=2$. 

        For any $p\geq 2$, from \eqref{eq:ODE}, define $y_p(t):=|\omega(t)|_p^p$. Then
        \begin{equation}\label{diff-ineq}
            y_p'(t)\le -K_{\alpha,p}\,y_p^{1+\alpha}(t)\,|\omega(t)|_{p/2}^{-\alpha p}, \qquad K_{\alpha,p}>0.
        \end{equation}
        Assume the induction hypothesis
        \begin{align*}
            t^{\frac1\alpha(1-\frac2p)}|\omega(t)|_{p/2}\to 0, \qquad t\to\infty.
        \end{align*}
        Let
        \begin{align*}
            M_p(t):=\Bigl(t^{\frac1\alpha(1-\frac2p)}|\omega(t)|_{p/2}\Bigr)^{-\alpha p}\to\infty,
        \end{align*}
        so that $|\omega(t)|_{p/2}^{-\alpha p}=t^{p-2}M_p(t)$. Hence \eqref{diff-ineq} yields
        \begin{align*}
            (y_p^{-\alpha})'(t)=-\alpha y_p^{-\alpha-1}(t)y_p'(t) \ge \alpha K_{\alpha,p}\,t^{p-2}M_p(t).
        \end{align*}
        Fix $L>1$ and choose $T_L$ such that $M_p(t)\ge L$ for all $t\ge T_L$. Integrating over $(T_L,t)$,
        \begin{align*}
            y_p^{-\alpha}(t)\ge \alpha K_{\alpha,p}L\int_{T_L}^t s^{p-2}\mathrm{d}s=\frac{\alpha K_{\alpha,p}}{p-1}L\,(t^{p-1}-T_L^{p-1}).
        \end{align*}
        Therefore,
        \begin{align*}
            t^{\frac{p-1}{\alpha}}y_p(t) \le C_{\alpha,p}\,L^{-1/\alpha}\Bigl(1-(T_L/t)^{p-1}\Bigr)^{-1/\alpha},\qquad C_{\alpha,p}:=\Bigl(\frac{\alpha K_{\alpha,p}}{p-1}\Bigr)^{-1/\alpha}.
        \end{align*}
        Taking $\limsup_{t\to\infty}$ and then $L\to\infty$ gives $t^{\frac{p-1}{\alpha}}y_p(t)\to 0$, i.e.
        \begin{align*}
            t^{\frac1\alpha(1-\frac1p)}|\omega(t)|_p\to 0, \qquad t\to\infty.
        \end{align*}
        In particular, taking $p=2^n\to\infty$ and using $|\omega|_p\to |\omega|_\infty$, we obtain $t^{1/\alpha}|\omega(t)|_\infty\to 0$. Finally, \eqref{w-d} follows by interpolation, and \eqref{u-d} follows from Lemma \ref{sol-e}; moreover $|\mathbf{u}(t)|_2\to 0$ is proved in \cite{ZL 24}.       
    \end{proof}

    \section{Scaling variables and invariant manifolds}
    
    Our analysis of the long-time asymptotics of \eqref{ve} depend on rewriting the equation in terms of scaling variables:
    \begin{align*}
        \xi=\frac{x}{(1+t)^{\frac{1}{2\alpha}}},\quad \tau=\text{log}(1+t).
    \end{align*}
    If $\omega(x,t)$ is the solution of \eqref{ve} and $\boldsymbol{u}(t)$ is the corresponding velocity field, we define new function $w(\xi,\tau)$ and $\textbf{v}(\xi,\tau)$ by  
    \begin{align}
        \omega(x,t)=\frac{1}{(1+t)}w\bigg(\frac{x}{(1+t)^{1/2\alpha}},\text{log}(1+t)\bigg),\label{w-omega}\\
        \textbf{u}(x,t)=\frac{1}{(1+t)^{1-1/2\alpha}}\textbf{v}\bigg(\frac{x}{(1+t)^{1/2\alpha}},\text{log}(1+t)\bigg)\label{v-u}.
    \end{align}
    Then $w(\xi,\tau)$ satisfies 
    \begin{align}\label{nve}
        \partial_{\tau}w=\mathcal{L}_{\alpha}w-(\textbf{v}\cdot \nabla_{\xi})w,
    \end{align}
    where 
    \begin{align}\label{linear}
        \mathcal{L}_{\alpha}w=-(-\Delta)^{\alpha}w+\frac{1}{2\alpha}(\boldsymbol{\xi}\cdot \nabla_{\xi})w+w,
    \end{align}
    and 
    \begin{align}\label{nbs}
        \textbf{v}(\xi,\tau)=\frac{1}{2\pi}\int_{\mathbb{R}^{2}}\frac{(\boldsymbol{\xi}-\boldsymbol{\eta})^{\perp}}{|\xi-\eta|^{2}}w(\eta,\tau)\mathrm{d}\eta.
    \end{align}
    For the detailed derivation of the equation \eqref{nve}, refer to Appendix A.

    The bounds obtained in the previous section yield quantitative information for solutions to \eqref{nve}; for any $w_0\in L^1(\mathbb{R}^2)$, there exists a unique global solution $w\in C^0([0,\infty);L^1(\mathbb{R}^2))\cap C^0((0,\infty);L^\infty(\mathbb{R}^2))$ with $w(0)=w_0$. 
    Translating \eqref{ww0} and \eqref{u-sol-e} into rescaled variables, we obtain that, for all $\tau>0$
    \begin{equation}\label{eq:Lp-w}
        |w(\tau)|_{p}\leq \frac{C_{p,\alpha}|w_{0}|_{1}}{a(\tau)^{\frac{1}{\alpha}(1-\frac{1}{p})}}e^{(1-\frac{1}{\alpha})\tau},\qquad 1\leq p\leq \infty,
    \end{equation}
    and
    \begin{equation}\label{eq:Lq-v}
        |\textbf{v}(\tau)|_{q}\leq \frac{C_{q,\alpha}|w_{0}|_{1}}{a(\tau)^{\frac{1}{\alpha}(\frac{1}{2}-\frac{1}{q})}}e^{(1-\frac{1}{\alpha})\tau},\qquad 2<q\leq \infty,
    \end{equation}
    where $a(\tau)=1-e^{-\tau}$ and the constants $C_{p,\alpha},C_{q,\alpha}>0$ are independent of $\tau$ and of $w_{0}$. It is worth noting that, unlike $\alpha=1$, the total mass of the rescaled vorticity $w(\tau)$ is not conserved but decays exponentially due to the factor $e^{(1-1/\alpha)\tau}$. In addition, if $\textbf{v}_{0}\in L^{2}(\mathbb{R}^{2})^{2}$, then $\displaystyle\int_{\mathbb{R}^{2}} w(\xi,\tau)\mathrm{d}\xi=0$ for every $\tau\ge0$. Using this and estimate \eqref{w-d}, one concludes that $|w(\tau)|_{p}\to0$ as $\tau\to\infty$ for all $p\in(1,\infty]$.

    In this paper, we will use weighted $L^{2}$ space to obtain better the long-time asymptotics of this solution. For any $m\geq 0$, we define the Hilbert space $L^{2}(m\alpha)$ by
    \begin{gather*}
        L^{2}(m\alpha)=\{f\in L^{2}(\mathbb{R}^{2})\ |\  \|f\|_{m\alpha}<\infty\},\\
        \|f\|_{m\alpha}=\bigg(\int_{\mathbb{R}^{2}}(1+|\xi|^{2})^{m\alpha}|f(\xi)|^{2}\mathrm{d}\xi\bigg)^{\frac{1}{2}}=|b^{m\alpha}f|_{2},
    \end{gather*}
    where $b(\xi)=(1+|\xi|^{2})^{1/2}$. If $m\alpha>1$, then $L^{2}(m\alpha)\hookrightarrow L^{1}(\mathbb{R}^{2})$ (see Proposition \ref{prop:L2mq_to_Lq}). Analogously, the weighted Sobolev spaces are defined by
    \begin{align*}
        H^{1}(m\alpha)=\{w\in L^{2}(m\alpha)\ |\ \partial_{i}f\in L^{2}(m\alpha)\ \text{for}\ i=1,2\}.
    \end{align*}

    The explicit kernel formula \eqref{group-P} in Appendix B defines bounded operators $S^\alpha(\tau)$ on $L^2(m\alpha)$ forming a semigroup and satisfying $\lim_{\tau\to 0}\|S^\alpha(\tau)f-f\|_{m\alpha}=0$; therefore $\mathcal L_\alpha$ generates the $C_0$-semigroup $S^\alpha(\tau)=e^{\tau\mathcal L_\alpha}$. It should be stressed that the semigroup $e^{\tau\mathcal{L}_{\alpha}}$ fails to commute with space derivatives. Indeed, due to $\partial_{x_{i}}\mathcal{L}_{\alpha}=(\mathcal{L}_{\alpha}+\frac{1}{2\alpha})\partial_{x_{i}}$ for $i=1,2$, we get $\partial_{x_{i}}e^{\tau\mathcal{L}_{\alpha}}=e^{\frac{\tau}{2\alpha}}e^{\tau\mathcal{L}_{\alpha}}\partial_{x_{i}}$ for all $\tau\geq 0$. Hence, according to this and the fact $\nabla\cdot\textbf{v}=0$, we can rewrite equation \eqref{nve} in integral form:
    \begin{align}\label{sol-inte}
        w(\tau)=e^{\tau\mathcal{L}_{\alpha}}w_{0}-\int_{0}^{\tau}e^{-\frac{1}{2\alpha}(\tau-s)}\nabla\cdot e^{(\tau-s)\mathcal{L}_{\alpha}}(\textbf{v}(s)w(s))\mathrm{d}s,
    \end{align}
    where $w_{0}=w(0)\in L^{2}(m\alpha)$. The following lemma shows that the quadratic nonlinear term in \eqref{sol-inte} is bounded (hence smooth) in $L^{2}(m\alpha)$ if $2\leq m<4$ and $\alpha>1/2$.
    \begin{lemma}\label{nonlinear-es}
        Fix $2\leq m<4$, $\alpha>1/2$ and $T>0$. Given $w_{1}$, $w_{2}\in C^{0}([0,T],L^{2}(m\alpha))$, define 
        \begin{align*}
            R(\tau)=\int_{0}^{\tau}\nabla\cdot e^{(\tau-s)\mathcal{L}_{\alpha}}(\mathbf{v}_{1}(s)w_{2}(s))\mathrm{d}s,\qquad 0\leq \tau\leq T,
        \end{align*}
        where $\mathbf{v}_{1}$ is obtained from $w_{1}$ via the Biot-Savart law \eqref{nbs}. Then $R\in C^{0}([0,T],L^{2}(m\alpha))$, and there exists $C_{0}=C_{0}(m,T,\alpha)>0$ such that 
        \begin{align*}
            \sup\limits_{0\leq \tau\leq T}\|R(\tau)\|_{m\alpha}\leq C_{0}\bigg(\sup\limits_{0\leq \tau\leq T}\|w_{1}(\tau)\|_{m\alpha}\bigg)\bigg(\sup\limits_{0\leq \tau\leq T}\|w_{2}(\tau)\|_{m\alpha}\bigg).
        \end{align*}
        Moreover, $C_{0}(m,T,\alpha)\to 0$ as $T\to 0$.
    \end{lemma}
    \begin{proof}
        Choose $q\in (1,2)$ such that $q>\frac{2}{m\alpha+1}$. Using proposition \eqref{group-est}, for $|\beta|=1$, $p=N=2$ and $q<\frac{2}{m\alpha-2\alpha}$, we have  
        \begin{align*}
            |b^{m\alpha}\nabla\cdot S^{\alpha}(\tau)f|_{2}\leq C a(\tau)^{-\frac{1}{\alpha q}}|b^{m\alpha}f|_{q},\qquad 0\leq \tau\leq T.
        \end{align*}
        Moreover
        \begin{align*}
            \|R(\tau)\|_{m\alpha}\leq C\int_{0}^{\tau}a(\tau-s)^{-\frac{1}{\alpha q}}|b^{m\alpha}\textbf{v}_{1}(s)w_{2}(s)|_{q}\mathrm{d}s,
        \end{align*}
        where $a(\tau)=1-e^{-\tau}$ and $b(\xi)=(1+|\xi|^{2})^{\frac{1}{2}}$. Moreover, using H\"older's inequality and Lemma \ref{sol-e}, we have 
        \begin{align*}
            |b^{m\alpha}\textbf{v}_{1}w_{2}|_{q}\leq |b^{m\alpha}w_{2}|_{2}|\textbf{v}_{1}|_{\frac{2q}{2-q}}\leq \|w_{2}\|_{m\alpha}\|w_{1}\|_{m\alpha},
        \end{align*}
        since $L^{2}(m\alpha)\hookrightarrow L^{q}(\mathbb{R}^{2})$. Therefore, for all $\tau\in [0,T]$, we find
        \begin{align*}
            \|R(\tau)\|_{m\alpha}\leq C\int_{0}^{\tau}a(\tau-s)^{-\frac{1}{\alpha q}}\mathrm{d}s\bigg(\sup\limits_{0\leq \tau\leq T}\|w_{1}(\tau)\|_{m\alpha}\bigg)\bigg(\sup\limits_{0\leq \tau\leq T}\|w_{2}(\tau)\|_{m\alpha}\bigg),
        \end{align*}
        where $a(s)=1-e^{-s}$. Since $a(s)\sim s$ as $s\to 0$, the time integral is finite if and only if $\frac{1}{\alpha q}<1$, or equivalently $\alpha>1/2$.
    \end{proof}

We now establish that \eqref{nve} admits a global solution in $L^{2}(m\alpha)$ under the conditions $2 \leq m < 4$ and $\alpha > 1/2$.

    \begin{theorem}\label{global}
        Suppose that $w_{0}\in L^{2}(m\alpha)$ for some $2\leq m<4$ and $\alpha>1/2$. Then \eqref{nve} has a unique global solution $w\in C^{0}([0,T],L^{2}(m\alpha))$ with $w(0)=w_{0}$, and there exists $C_{1}=C_{1}(\|w_{0}\|_{m\alpha})>0$ such that
        \begin{align}\label{sol-bounded}
            \|w(\tau)\|_{m\alpha}\leq C_{1},\qquad \tau\geq 0.
        \end{align}
        Moreover, $C_{1}(\|w_{0}\|)\to 0$ as $\|w_{0}\|_{m\alpha}\to 0$. Finally, if $w_{0}\in L^{2}_{0}(m\alpha)$, then $\int_{\mathbb{R}^{2}}w(\xi,\tau)\mathrm{d}\xi=0$ for all $\tau\geq 0$, and $\lim\limits_{\tau\to\infty}\|w(\tau)\|_{m\alpha}=0$.
    \end{theorem}
    \begin{proof}
    Using Lemma \ref{nonlinear-es} and a fixed-point argument. Choose Banach space $X_{T}=C([0,T],L^{2}(m\alpha))$ with norm $\|w\|_{X_{T}}=\sup\limits_{0\leq \tau\leq T} \|w(\tau)\|_{m\alpha}$. Define the mapping $\Psi: X_{T}\mapsto X_{T}$
    \begin{align*}
        \Psi[w](\tau):=S^{\alpha}(\tau)w_{0}-\int_{0}^{\tau}e^{-(\tau-s)/2\alpha}\nabla\cdot S^{\alpha}(\tau-s)(\textbf{v}(s)w(s))\mathrm{d}s. 
    \end{align*}
    To prove that $\Psi$ is a contraction mapping, we estimate the nonlinear term. Let $q \in (1, 2)$ be fixed. According to Proposition \ref{group-est}, we have
    \begin{align*}
        \|e^{-(\tau-s)/2\alpha}\nabla\cdot S^{\alpha}(\tau-s)f\|_{m\alpha} \leq C(\tau-s)^{-\gamma} |b^{m\alpha}f|q,
    \end{align*} 
    where $\gamma = \frac{1}{2\alpha}(1 + 2(\frac{1}{q}-\frac{1}{2})) < 1$, since $\alpha > 1/2$, we can choose $q$ close to 2 such that $\gamma < 1$. Now, for $w_{1}, w_{2} \in B_{k}$, consider the difference $\textbf{v}_{1}w_{1}-\textbf{v}_{2}w_{2} = (\textbf{v}_{1}-\textbf{v}_{2})w_{1} + \textbf{v}_{2}(w_{1}-w_{2})$. Using H\"older's inequality with exponent pair $(2, \frac{2q}{2-q})$ (note that $\frac{2q}{2-q} > 2$), we obtain:
    \begin{align*}
        |b^{m\alpha}(\textbf{v}_{1}-\textbf{v}_{2})w_{1}|_{q} &\leq \|w_{1}\|_{m\alpha} |\textbf{v}_{1}-\textbf{v}_{2}|_{\frac{2q}{2-q}} \\
        &\leq C \|w_{1}\|_{m\alpha} |w_{1}-w_{2}|_{q} \quad \text{(by Biot-Savart law in $L^p$)}\\
        &\leq C \|w_{1}\|_{m\alpha} \|w_{1}-w_{2}\|_{m\alpha}.
    \end{align*}
    Similarly, $|b^{m\alpha}\textbf{v}_{2}(w_{1}-w_{2})|_{q} \leq C \|w_{2}\|_{m\alpha} \|w_{1}-w_{2}\|_{m\alpha}$. Therefore, substituting these estimates into the integral equation:
    \begin{align*}
        \|\Psi[w_{1}](\tau)-\Psi[w_{2}](\tau)\|_{m\alpha} 
        &\leq \int_{0}^{\tau} C(\tau-s)^{-\gamma} |b^{m\alpha}(\textbf{v}_{1}w_{1}-\textbf{v}_{2}w_{2})|_{q} \mathrm{d}s \\
        &\leq C \left(\int_{0}^{\tau} (\tau-s)^{-\gamma} \mathrm{d}s \right) (\|w_{1}\|_{X_T} + \|w_{2}\|_{X_T}) \|w_{1}-w_{2}\|_{X_T} \\
        &\leq C T^{1-\gamma} (\|w_{1}\|_{X_T} + \|w_{2}\|_{X_T}) \|w_{1}-w_{2}\|_{X_T}.
    \end{align*}
    Since $1-\gamma > 0$, we can choose $T$ sufficiently small such that $C T^{1-\gamma} (4\|w_{0}\|_{m\alpha}) < 1$. Thus, $\Psi$ is a contraction mapping on the ball $B_k$. Hence, for any $K>0$, there exists $\tilde{T}=\tilde{T}(K)>0$ such that equation \eqref{nve} has a unique local solution $w\in C^{0}([0,\tilde{T}],L^{2}(m\alpha))$. This solution $w(\tau)$ depends continuously on the initial data $w_{0}$, uniformly in $\tau\in [0,\tilde{T}]$. Moreover, $\tilde{T}$ can be chosen so that $\|w(\tau)\|_{m\alpha}\leq 2\|w_{0}\|_{m\alpha}$ for all $\tau\in [0,\tilde{T}]$. Thus, in order to prove global existence, it is sufficient to show that any solution $w\in C^{0}([0,T],L^{2}(m\alpha))$ of \eqref{sol-inte} satisfies the bound \eqref{sol-bounded} for some $C_{1}>0$. 

    Let $w_{0}\in L^{2}(m\alpha)$, $T>0$, and assume that $w\in C^{0}([0,T],L^{2}(m\alpha))$ is a solution of \eqref{sol-inte}. Without loss of generality, we suppose that $T\geq \tilde{T}\equiv \tilde{T}(\|w_{0}\|_{m\alpha})$. Since $L^{2}(m\alpha)\hookrightarrow L^{p}(\mathbb{R}^{2})$ for all $p\in [1,2]$, there exists $c>0$ such that $|w(\tau)|_{p}\leq C\|w(\tau)\|_{m\alpha}$ for all $\tau\in [0,\tilde{T}]$. By \eqref{eq:Lp-w}, we also have $|w(\tau)|_{p}\leq C_{p}a(\tilde{T})^{-\frac{1}{\alpha}(1-1/p)}|w_{0}|_{1}\leq C\|w_{0}\|_{m\alpha}$ such that $|w(\tau)|_{p}\leq C_{2}$ for all $p\in [1,2]$ and all $\tau\in [0,T]$. Moreover, $C_{2}(\|w_{0}\|_{m\alpha})\to 0$ as $\|w_{0}\|_{m\alpha}\to 0$. To bound $||\xi|^{m\alpha}w(\tau)|_{2}$, we compute
        \begin{align}\label{xi-ex}
            \frac{1}{2}\frac{\mathrm{d}}{\mathrm{d}\tau}&\int_{\mathbb{R}^{2}} |\xi|^{2m\alpha}w(\xi,\tau)^{2}\mathrm{d}\xi=\int_{\mathbb{R}^{2}}|\xi|^{2m\alpha}\bigg\{w(-(-\Delta)^{\alpha}w+\frac{1}{2\alpha}(\xi\cdot\nabla)w+w-(\textbf{v}\cdot\nabla)w)\bigg\}\mathrm{d}\xi.
        \end{align}
        Integrating by parts and using the fact that div $\textbf{v}=0$, we can rewrite
        \begin{gather}
            \int_{\mathbb{R}^{2}}|\xi|^{2m\alpha}(w(-(-\Delta)^{\alpha}w))\mathrm{d}\xi\leq -C_{*}[|\xi|^{m\alpha}w]_{\dot{H}^{\alpha}}^{2}+C_{*}\bigg(\int_{\mathbb{R}^{2}}|\xi|^{2m\alpha}w^{2}\mathrm{d}\xi+\int_{\mathbb{R}^{2}}w^{2}\mathrm{d}\xi\bigg),\label{f-esti}\\
            \frac{1}{2\alpha}\int_{\mathbb{R}^{2}}|\xi|^{2m\alpha}w(\boldsymbol{\xi}\cdot\nabla)w\mathrm{d}\xi=-\frac{m\alpha+1}{2\alpha}\int_{\mathbb{R}^{2}}|\xi|^{2m\alpha}w^{2}\mathrm{d}\xi,\notag\\
            \int_{\mathbb{R}^{2}}|\xi|^{2m\alpha}w(\mathbf{v}\cdot\xi)\nabla w\,\mathrm{d}\xi= \frac{1}{2}\int_{\mathbb{R}^{2}}|\xi|^{2m\alpha}\mathbf{v}\cdot\nabla(w^{2})\,\mathrm{d}\xi= \frac{1}{2}\int_{\mathbb{R}^{2}}|\xi|^{2m\alpha}\nabla\cdot(\mathbf{v}w^{2})\mathrm{d}\xi\notag\\
            = -m\alpha\int_{\mathbb{R}^{2}}|\xi|^{2m\alpha-2}(\boldsymbol{\xi}\cdot\mathbf{v})w^{2}\,\mathrm{d}\xi\notag,
        \end{gather}
        where the estimate \eqref{f-esti} is proved in Appendix C. Inserting these expressions into \eqref{xi-ex}, we get
        \begin{align*}
            \frac{1}{2}\frac{\mathrm{d}}{\mathrm{d}\tau}\int_{\mathbb{R}^{2}}|\xi|^{2m\alpha}w^{2}\mathrm{d}\xi\leq& -C_{*}[|\xi|^{m\alpha}w]_{\dot{H}^{\alpha}}^{2}+\bigg(C_{*}-\frac{1-m\alpha}{\alpha})\int_{\mathbb{R}^{2}}|\xi|^{2m\alpha}w^{2}\mathrm{d}\xi\\
            &+m\alpha\int_{\mathbb{R}^{2}}|\xi|^{2m\alpha-2}(\boldsymbol{\xi}\cdot\textbf{v})w^{2}\mathrm{d}\xi+C_{*}\int_{\mathbb{R}^{2}}w^{2}\mathrm{d}\xi.
        \end{align*}
        Note that for all $\varepsilon >0$, there exists $C_{\varepsilon}>0$ such that
        \begin{align*}
            \bigg|\int_{\mathbb{R}^{2}}|\xi|^{2m\alpha-2}(\boldsymbol{\xi}\cdot\textbf{v})w^{2}\mathrm{d}\xi\bigg|\leq \varepsilon \int_{\mathbb{R}^{2}}|\xi|^{2m\alpha}w^{2}\mathrm{d}\xi+C_{\varepsilon}|\textbf{v}|_{\infty}^{2m\alpha}\int_{\mathbb{R}^{2}}w^{2}\mathrm{d}\xi.
        \end{align*}
        According to \eqref{uiw}, $|\textbf{v}(\tau)|_{\infty}\leq C|w|_{p}^{\beta}|w|_{q}^{1-\beta}$ with $1<p<2<q<\infty$ and $\frac{\beta}{p}+\frac{1-\beta}{q}=\frac{1}{2}$. We also know that $|w(\tau)|_{p}\leq C_{2}$ for all $\tau\in [0,T]$. Choosing $\beta=1-\frac{1}{8m\alpha}$, $q=4+\frac{1}{2m\alpha}$, and using \eqref{eq:Lp-w} to bound $|w(\tau)|_{q}$, we obtain
        \begin{align*}
            |\textbf{v}(\tau)|_{\infty}^{2m\alpha}\leq CC_{2}^{2m\alpha}\bigg(\frac{C_{q}|w_{0}|_{1}}{a(\tau)^{1-(\frac{1}{q})}}\bigg)^{2m\alpha(1-\beta)}\leq C_{3}(1+\tau^{-\frac{1}{4\alpha}}),\qquad 0<\tau\leq T.
        \end{align*}
        where $C_{3}=C_{3}(m,\alpha,\|w_{0}\|_{m\alpha})$. Thus, there exists $C_{4}=C_{4}(m,\alpha,\|w_{0}\|_{m\alpha})$ such that
        \begin{align}\label{gronwall}
            \frac{\mathrm{d}}{\mathrm{d}\tau}\int_{\mathbb{R}^{2}}|\xi|^{2m\alpha}w^{2}\mathrm{d}\xi\leq -C_{4} \bigg(\int_{\mathbb{R}^{2}}|\xi|^{2m\alpha}w^{2}\mathrm{d}\xi-(1+\tau^{-\frac{1}{4\alpha}})\int_{\mathbb{R}^{2}}w^{2}\mathrm{d}\xi\bigg).
        \end{align}
        By Gronwall's inequality, we conclude that the bound \eqref{sol-bounded} holds. Finally, if $w_{0}\in L^{2}_{0}(m\alpha)$ and if $\textbf{v}_{0}$ is the velocity field obtained from $w_{0}$ via the Biot-Savart law \eqref{nbs}, then $\textbf{v}_{0}\in L^{2}(\mathbb{R}^{2})^{2}$. In this case, we already observed that $|w(\tau)|_{2}$ converges to zero as $\tau\to +\infty$, and so does $||\xi|^{m\alpha}w(\tau)|_{2}$ by \eqref{gronwall}. This concludes the proof of Theorem \ref{global}.
    \end{proof}
    
    To analyze the proposed fractional-order equation, we adopt a strategy rooted in dynamical systems theory, effectively used in \cite{TC 02} for the integer-order counterpart. The analysis proceeds in three main steps. First, we investigate the spectrum of the linear operator $\mathcal{L}_{\alpha}$. We explicitly characterize its structure, which consists of a discrete spectrum $\sigma_{d}(\mathcal{L}_{\alpha})=\{1-\frac{1}{\alpha}-\frac{k}{2\alpha}\mid k\in \mathbb{N}_0\}$, with eigenfunctions corresponding to the partial derivatives of the fractional Oseen-type vortex $G_\alpha$, and a continuous spectrum $\sigma_{c}(\mathcal{L}_{\alpha})$ located in a left half-plane. A crucial step is to work in a weighted space $L^{2}(m\alpha)$, which shifts the edge of the continuous spectrum to $\Re \lambda\leq 1-\frac{1}{2\alpha}-\frac{m}{2}$, see Appendix B; choosing a suitable $m$ then places the continuous spectrum well to the left of the imaginary axis and prevents it from affecting the leading-order dynamics. Second, using the eigenfunctions of the discrete spectrum as a basis, we project the PDE onto an finite-dimensional system of ODEs. This system features a diagonal linear part and a quadratic nonlinearity, a structure amenable to invariant manifold and normal-form techniques \cite{PC 89}. Finally, we prove the existence of a finite-dimensional invariant manifold tangent to the "slow" eigenspace at the origin. The system's long-term behavior is then entirely governed by the flow on this manifold. The primary technical challenges are posed by the need to control the shifted continuous spectrum and to handle the lack of smoothness exhibited by the nonlinearity in the weighted space.

    Rather than tackling the differential equation directly, we analyze the associated semiflow. By deriving appropriate estimates for this semiflow, we can directly apply the invariant manifold theorem of Chen, Hale, and Tan \cite{XC 97}, thereby completing the analysis. To apply the invariant manifold theorem, we localize the dynamics of \eqref{nve} near the equilibrium $w = 0$ by truncating the nonlinear term outside a sufficiently small neighborhood of the origin.
   
    Let $\chi: L^{2}(m\alpha) \to \mathbb{R}$ be a bounded $C^{\infty}$ function satisfying $0 \leq \chi \leq 1$,
    \begin{align*}
        \chi(w) = 1 \ \text{for} \ \|w\|_{m\alpha} \leq 1, \quad \text{and} \quad \chi(w) = 0 \ \text{for} \ \|w\|_{m\alpha} \geq 2.
    \end{align*}

    For any $r_0 > 0$, define its scaled version $\chi_{r_0}(w) := \chi(w / r_0)$. Then $\chi_{r_0}(w) = 1$ if $\|w\|_{m\alpha} \leq r_0$, and $\chi_{r_0}(w) = 0$ if $\|w\|_{m\alpha} \geq 2r_0$.

    \begin{remark}
        We remark that such a smooth cutoff function $\chi$ can be constructed in a standard way within the Hilbert space setting, in particular for the weighted space $L^{2}(m\alpha)$.
    \end{remark}

    We therefore study the regularized equation
    \begin{align}\label{truncation}
        \partial_{\tau} w = \mathcal{L}_{\alpha} w - (\mathbf{v}\!\cdot\!\nabla)\big(\chi_{r_{0}}(w)\,w\big),
    \end{align}
    which removes the potential ill-posedness by truncation. Here $\mathcal{L}_{\alpha}$ denotes the linear fractional part. By construction, \eqref{truncation} matches \eqref{nve} inside the $r_0$-ball and degenerates to the linear equation outside the $2r_0$-ball. Applying an estimate similar to Lemma \ref{nonlinear-es} guarantees globally defined, uniformly bounded solutions for parameters $2 \le m < 4$. This yields a global semiflow $\Phi_{\tau}^{r_{0},m\alpha}$ for the regularized system, in contrast to the local semiflow $\Phi_{\tau}^{m\alpha}$ of the original system \eqref{nve}. The following proposition establishes the decomposition of the time-one map $\Phi_1^{r_0, m\alpha}$, which is crucial for the subsequent analysis.

    \begin{proposition}\label{decomposition}
        Fix $2\leq m<4$, $r_{0}>0$ and let $\Phi_{\tau}^{r_{0,m\alpha}}$ be the semiflow on $L^{2}(m\alpha)$ defined by \eqref{truncation}. If $r_{0}>0$ is sufficiently small,  then the time-one map \(\Phi^{r_0,m\alpha}_1\) admits the decomposition
        \begin{align}\label{semi-decompose}
            \Phi_{1}^{r_{0},m\alpha}=\Lambda+\mathscr{R},
        \end{align}
        where $\Lambda\equiv e^{\mathcal{L}_{\alpha}}$ is a bounded linear operator on $L^{2}(m\alpha)$, and $\mathscr{R}: L^{2}(m\alpha)\to L^{2}(m\alpha)$ is a $C^{\infty}$ map satisfying $\mathscr{R}(0)=0$, $D\mathscr{R}(0)=0$. Moreover, $\mathscr{R}$ is globally Lipschitz, and there exists $L>0$ (independent of $r_{0}$) such that $Lip(\mathscr{R})\leq Lr_{0}$.
    \end{proposition}
    \begin{proof}
        Let $w_{1}, w_{2}\in L^{2}(m\alpha)$, and define $w_{i}(\tau)=\Phi_{\tau}^{r_{0},m\alpha}w_{i}$ for $i=1,2$, $0\leq \tau \leq 1$. Then, for $i=1,2$,
        \begin{align*}
            w_{i}(\tau)=e^{\tau\mathcal{L}_{\alpha}}w_{i}-\int_{0}^{\tau}e^{-\frac{\tau-s}{2\alpha}}\nabla\cdot e^{(\tau-s)\mathcal{L}_{\alpha}}(\textbf{v}_{i}\chi_{r_{0}}(w_{i}(s))w_{i}(s))\mathrm{d}s,
        \end{align*}
        where $\textbf{v}_{i}(\tau)$ is the velocity field obtained from $w_{i}(\tau)$ via the Biot-Savart law \eqref{nbs}. Proceeding as in the proof Lemma \ref{nonlinear-es}, we find that there exists $C\geq 1$ and $K>0$ such that 
        \begin{align}\label{3.14}
            \sup\limits_{0\leq \tau\leq 1}\|w_{1}(\tau)-w_{2}(\tau)\|_{m\alpha} \leq C\|w_{1}-w_{2}\|_{m\alpha}+Kr_{0}\sup\limits_{0\leq \tau\leq 1} \|w_{1}(\tau)-w_{2}(\tau)\|_{m\alpha}.
        \end{align}
        Thus, assuming $Kr_{0}\leq \frac{1}{2}$, we obtain $\|\Phi_{\tau}^{r_{0},m\alpha}w_{1}-\Phi_{\tau}^{r_{0},m\alpha}w_{2}\|_{m\alpha}\leq 2C\|w_{1}-w_{2}\|_{m\alpha}$ for all $\tau\in [0,1]$. We now define $\mathscr{R}=\Phi_{1}^{r_{0},m\alpha}-\Lambda$, so that 
        \begin{align}\label{3.15}
            \mathscr{R}(w_{i})=-\int_{0}^{1}e^{-\frac{1}{2}(1-\tau)}\nabla\cdot e^{(1-\tau)\mathcal{L}_{\alpha}}(\textbf{v}_{i}(\tau)\chi_{r_{0}}(w_{i}(\tau))w_{i}(\tau))\mathrm{d}\tau, \qquad i=1,2.
        \end{align}
        In view of \eqref{3.14}, we have $\|\mathscr{R}(w_{1})-\mathscr{R}(w_{2})\|_{m\alpha}\leq Lr_{0}\|w_{1}-w_{2}\|_{m\alpha}$, where $L=2CK$. Moreover, using \eqref{3.15}, it is not difficult to show that $\mathscr{R}:L^{2}(m\alpha)\to L^{2}(m\alpha)$ is smooth and satisfied $D\mathscr{R}(0)=0$.
    \end{proof}

    Fix $k\in\mathbb{N}$, we assume that the weight parameter $m$ satisfies both the global existence condition $2\le m<4$ (as established in Theorem \ref{global}) and the spectral gap condition $m\alpha > k+1$. As recalled in Appendix B, the spectrum of the linear evolution operator $\Lambda=e^{\mathcal{L}_\alpha}$ on $L^{2}(m\alpha)$ splits as
    \begin{align*}
        \sigma(\Lambda)=\Sigma_c\cup\Sigma_d,\qquad \Sigma_d=\Bigl\{e^{1-\frac{1}{\alpha}-\frac{n}{2\alpha}}:\ n\in\mathbb{N}_0\Bigr\},\qquad \Sigma_c\subset\Bigl\{\lambda\in\mathbb{C}:\ |\lambda|\le e^{\frac{2\alpha-1-m\alpha}{2\alpha}}\Bigr\}.
    \end{align*}
    If $m\alpha-1>k$, then $\Lambda$ has a spectral gap and hence admits $k+1$ isolated eigenvalues $\lambda_0,\dots,\lambda_k$ separated from $\Sigma_c$. Let $P_k$ be the spectral projection onto the span of the corresponding eigenfunctions in $L^{2}(m\alpha)$, and set $Q_k:=I-P_k$. Applying the invariant manifold theory of \cite{XC 97} to the semiflow $\Phi_{\tau}^{r_{0},m\alpha}$ yields the following result.

    \begin{theorem}\label{invarint}
        Fix $k\in \mathbb{N}$, $2\leq m<4$, $m\alpha>k+1$ and choose $\mu_{1}, \mu_{2}\in \mathbb{R}$ such that $\frac{k}{2\alpha}+\frac1\alpha-1<\mu_1<\mu_2<\min\left\{\frac{k+1}{2\alpha}+\frac1\alpha-1,\,\frac{m}{2}+\frac{1}{2\alpha}-1\right\}$. Let $P_{k}$, $Q_{k}$ be the spectral projections defined above, and let $E_{\text{d}}=P_{k}L^{2}(m\alpha)$, $E_{\text{c}}=Q_{k}L^{2}(m\alpha)$. Then, for $r_{0}>0$ sufficiently small, there exists a $C^{1}$ and globally Lipschitz map $g: E_{\text{d}}\to E_{\text{c}}$ with $g(0)=0$, $Dg(0)=0$, such that the submanifold
        \begin{align*}
            W_{\text{d}}=\{w_{s}+g(w_{s})\ |\ w_{s}\in E_{\text{d}}\}
        \end{align*}
        has the following properties:
        \begin{enumerate}
            \item (Invariance) The restriction $W_{\text{d}}$ of the semiflow $\Phi_{\tau}^{r_{0},m\alpha}$ can be extended to a Lipschitz flow on $W_{\text{d}}$, there exists a unique negative semi-orbit contained in $W_{\text{d}}$ with $w(0)=w_{0}$. If $\{w(\tau)\}_{\tau\leq 0}$ is a negative semi-orbit contained in $W_{\text{d}}$, then
            \begin{align}\label{wc-estimate}
                \underset{\tau\to -\infty}{\limsup}\frac{1}{|\tau|}\ln \|w(\tau)\|_{m\alpha}\leq \mu_{1}.
            \end{align}
            Conversely, if a negative semi-orbit of $\Phi_{\tau}^{r_{0},m\alpha}$ satisfies 
            \begin{align}
                \underset{\tau\to -\infty}{\limsup}\frac{1}{|\tau|}\ln \|w(\tau)\|_{m\alpha}\leq \mu_{2},
            \end{align}
            then it lies in $W_{\text{d}}$.
            \item (Invariant Foliation) There is a continuous map $h: L^{2}(m\alpha)\times E_{\text{c}}\to E_{\text{d}}$ such that, for each $w\in W_{\text{d}}$, $h(w,Q_{k}(w))=P_{k}(w)$ and the manifold $M_{w}=\{h(w,w_{f})+w_{f}\ |\ w_{f}\in E_{\text{c}}\}$ passing through $w$ satisfies $\Phi_{\tau}(M_{w})\subset M_{\Phi_{\tau}^{r_{0},m\alpha}(w)}$ and 
            \begin{align}\label{3.18}
                M_{w}=\bigg\{\tilde{w}\in L^{2}(m\alpha)\ \bigg|\ \underset{\tau\to -\infty}{\limsup}\frac{1}{|\tau|}\ln \|\Phi_{\tau}^{r_{0},m\alpha}(\tilde{w})-\Phi_{\tau}^{r_{0},m\alpha}(w)\|_{m\alpha}\leq \mu_{2}\bigg\}.
            \end{align}
            Moreover, $h:\ L^{2}(m\alpha)\times E_{\text{c}}\to E_{\text{d}}$ is $C^{1}$ in the $E_{\text{c}}$ direction.
            \item (Completeness) For every $w\in L^{2}(m\alpha)$, $M_{w}\cap W_{\text{d}}$ is exactly a single point. In particular, $\{M_{w}\}_{w\in W_{\text{d}}}$ is a foliation of $L^{2}(m\alpha)$ over $W_{\text{d}}$.
        \end{enumerate}
    \end{theorem}
    \begin{remark}
        It should be noted that $W_{\text{d}}$ is not intrinsically unique because its construction relies on the specific choice of the cut-off function $\chi$. This dependence constitutes the sole source of non-uniqueness; once the cut-off function is fixed, the procedure outlined in \cite{XC 97}.
    \end{remark}
    \begin{proof}
        We only need to prove that hypotheses (H.1)-(H.4) of Theorem 1.1 in \cite{XC 97} hold for the semiflow $\Phi_{\tau}^{r_{0},m\alpha}$. First choose $r_{0}>0$ small, then (H.1) holds because the map $w\mapsto \Phi_{\tau}^{r_{0},m\alpha}w$ is globally Lipschitz on $L^{2}(m\alpha)$, uniformly for $\tau\in [0,1]$, see \eqref{3.14}. Assumption (H.2) corresponds to the decomposition
        \begin{align*}
            \Phi_{\tau}^{r_{0},m\alpha}=\Lambda +\mathscr{R},
        \end{align*}
        which is exactly the representation obtained in Proposition \ref{decomposition}. Condition (H.4) requires the Lipschitz constant $\mathscr{R}$ to be sufficiently small; this is achieved again by taking $r_{0}$ small enough.

        To verify (H.3), recall that $L^{2}(m\alpha)=E_{\text{d}}\oplus E_{\text{c}}$ and define 
        \begin{align*}
            \Lambda_{\text{d}}=P_{k}\Lambda P_{k},\qquad \Lambda_{\text{c}}=Q_{k}\Lambda Q_{k}.
        \end{align*}
        The spectrum of $\Lambda_{\text{d}}$ is $\sigma(\Lambda_{\text{d}})=\{e^{1-1/\alpha}, e^{1-3/2\alpha},\cdots,e^{1-\frac{2+k}{2\alpha}}\}$, and all the eigenvalues are semisimple. Hence $\Lambda_{\text{d}}$ is invertible and there exists a constant $C\geq 1$ such that 
        \begin{align*}
            \|\Lambda^{-j}_{\text{d}}w\|_{m\alpha}\leq Ce^{j\frac{2+k-2\alpha}{2\alpha}}\|w\|_{m\alpha},\qquad j\in \mathbb{N},w\in E_{\text{d}}.
        \end{align*}
        On the complementary subspace \(E_c\), the spectrum consists of the remaining discrete eigenvalues \(\lambda_n=1-\frac1\alpha-\frac{n}{2\alpha}\), \(n\ge k+1\), together with the continuous spectrum satisfying
        \begin{align*}
            \operatorname{Re}\lambda\le 1-\frac{1}{2\alpha}-\frac{m}{2}.
        \end{align*}
        Hence, for any $0<\varepsilon<\min\left\{\frac{k+1}{2\alpha}+\frac1\alpha-1,\,\frac{m}{2}+\frac{1}{2\alpha}-1 \right\}-\mu_2$, there exists \(C_\varepsilon>0\) such that
        \begin{align*}
            \|\Lambda_c^j w\|_{m\alpha}\le C_\varepsilon\exp\left[-j\left(\min\left\{\frac{k+1}{2\alpha}+\frac1\alpha-1,\,\frac{m}{2}+\frac{1}{2\alpha}-1\right\}-\varepsilon\right)\right]\|w\|_{m\alpha}.
        \end{align*}
        In particular, since \(\mu_2\) is chosen below this minimum, we obtain 
        \begin{align*}
            \|\Lambda_c^j w\|_{m\alpha}\le C e^{-\mu_2 j}\|w\|_{m\alpha}, \qquad j\in\mathbb N,\quad w\in E_c.
        \end{align*}
        Together with
        \begin{align*}
            \|\Lambda_d^{-j}w\|_{m\alpha}\le C e^{\left(\frac{k}{2\alpha}+\frac1\alpha-1\right)j}\|w\|_{m\alpha} \le C e^{\mu_1 j}\|w\|_{m\alpha},
        \end{align*}
        this verifies hypothesis (H.3).
    \end{proof}

    We must first satisfy the foundational condition by selecting a sufficiently small radius $r_{0}>0$ such that the conclusions of Theorem \ref{invarint} hold valid. A direct consequence of Theorem \ref{global} is the existence of a value $r_{1}>0$ which guarantees that any solution $w(\tau)$ starting from the region $\|w(0)\|_{m\alpha}\leq r_{1}$ remains within the ball $\|w(\tau)\|_{m\alpha}\leq r_{0}$ for all $\tau\geq 0$. For these trajectories, the truncated semiflow $\Phi_{\tau}^{r_{0},m\alpha}$ coincides with the original semiflow $\Phi_{\tau}^{m\alpha}$. This coincidence confirms that the invariant manifolds established in Theorem \ref{invarint} are also locally invariant for the original nonlinear system. Consequently, the phase space of the Navier-Stokes system possesses a rich family of finite-dimensional invariant manifolds. Moreover, parts 2 and 3 of Theorem \ref{invarint} are particularly significant as they quantify the convergence rate of solutions toward these manifolds. Since this result is indispensable for our subsequent analysis, we formally state it as the following Corollary.
    \begin{corollary}\label{local invariant}
        Fix $k\in \mathbb{N}$, $2\leq m<4$, and let $W_{\text{d}}$ be the submanifold of $L^{2}(m\alpha)$ constructed in Theorem \ref{invarint}. Define 
        \begin{align}
            W_{\text{d}}^{\text{loc}}=W_{\text{d}}\cap\{w\in L^{2}(m\alpha)|\|w\|_{m\alpha}<r_{0}\}.
        \end{align}
        Then $W_{\text{d}}^{\text{loc}}$ is locally invariant under the semiflow $\Phi_{\tau}^{m\alpha}$ defined by \eqref{nve}. If $\{w(\tau)\}_{\tau\leq 0}$ is a negative-orbit of \eqref{nve} such that $\|w(\tau)\|_{m\alpha}<r_{0}$ for all $\tau\leq 0$, then $w(\tau)\in W_{\text{d}}^{\text{loc}}$ for all $\tau\leq 0$. Moreover, for any $\mu<\mu_{2}$ (where $\mu_{2}$ is as in Theorem \ref{invarint}), there exist $r_{2}>0$ and $C>0$ with the following property: for all $\tilde{w}_{0}\in L^{2}(m\alpha)$ with $\|\tilde{w}_{0}\|_{m\alpha}\leq r_{2}$, there exists a unique $w_{0}\in W_{\text{d}}$ such that $\Phi_{\tau}^{m\alpha}\in W_{\text{d}}^{\text{loc}}$ for all $\tau\geq 0$ and
        \begin{align}\label{loc-invar-estimate}
            \|\Phi_{\tau}^{m\alpha}(\tilde{w}_{0})-\Phi_{\tau}^{m\alpha}(w_{0})\|_{m\alpha}\leq Ce^{-\mu\tau},\qquad \tau\geq 0.
        \end{align}
    \end{corollary}
    \begin{proof}
        By point 1 of Theorem \ref{invarint}, we can get that $W_{\text{d}}^{\text{loc}}$ is locally invariant under the $\Phi_{\tau}^{m\alpha}$ and contains the negative semi-orbits that stay in a neighborhood of the origin. To guarantee stability, we select a value $r_{1}>0$ such that imposing the restriction $\|w_{0}\|_{m\alpha}\leq r_{1}$ ensures the subsequent trajectory satisfies $\|\Phi_{\tau}^{m\alpha}(w_{0})\|_{m\alpha}<r_{0}$ for all $\tau\geq 0$. Next, we establish the core decomposition. Based on the point 2 and 3 of Theorem \ref{invarint}, for a given initial component $\tilde{w}_{0}\in L^{2}(m\alpha)$, there exists a unique corresponding unstable component $w_{0}$ situated on the intersection $M_{\tilde{w}_{0}}\cap W_{\text{d}}$. By defining this component as $w_{0}=w_{c}+g(w_{c})$, we find through the relevant definitions that $w_{c}$ must be the solution to the fixed-point relation:
        \begin{align}\label{3.21}
            w_{c}=h(\tilde{w}_{0},g(w_{c})).
        \end{align}
        Since the mapping $h$ is continuous and the operator $w_{c}\mapsto h(\tilde{w}_{0},g(w_{c}))$ satisfies a contraction (Lipschitz with small constant), \eqref{3.21} is guaranteed to possess a unique solution $w_{c}$. Importantly, this solution exhibits continuous dependence on $\tilde{w}_{0}$. Moreover, it follows from \eqref{3.18} that $w_{c}=0$ if $\tilde{w}_{0}=0$. Therefore, by continuity, there exists $r_{2}\in (0,r_{1}]$ such that, if $\|\tilde{w}_{0}\|_{m\alpha}\leq r_{2}$, then $\|w_{0}\|_{m\alpha}=\|w_{c}+g(w_{c})\|_{m\alpha}\leq r_{1}$. In this case, $\max(\|\Phi_{\tau}^{m\alpha}(\tilde{w}_{0})\|_{m\alpha},\|\Phi_{\tau}^{m\alpha}(w_{0})\|_{m\alpha})<r_{0}$ for all $\tau\geq 0$, and \eqref{loc-invar-estimate} follows from \eqref{3.18}.
    \end{proof}

    In the fractional case $1/2<\alpha<1$, let $\mu_{1}<\mu_{2}$. By Corollary \ref{local invariant}, any solution $w(\tau)$ of \eqref{nve} with initial data in $W^{\text{loc}}_{\text{d}}$ that is not identically zero cannot decay to the equilibrium faster than $\mathrm{e}^{-\mu_{1}\tau}$ as $\tau\to+\infty$. Following \cite{XC 97}, we therefore refer to $W^{\text{loc}}_{\text{d}}$ as the weak stable manifold of the origin.

    In many applications we are interested in trajectories that converge to the origin at a faster rate, namely with $\mathcal{O}\left(\mathrm{e}^{-\mu_{2}\tau}\right)$ or better. Theorem \ref{invarint} shows that, for sufficiently large $\tau$, every such solution belongs to a distinguished leaf $M_{w}$ of the center-stable foliation $\{M_{w}\}_{w\in W_{\text{d}}}$ passing through the origin. We call this leaf the strong stable manifold and denote by $W^{\text{loc}}_{\text{c}}$ its restriction to a neighbourhood of zero. In contrast with $W^{\text{loc}}_{\text{d}}$, the local strong stable manifold $W^{\text{loc}}_{\text{c}}$ is smooth and uniquely determined; in particular, it is independent of the choice of the cut-off function $\chi$.

    \begin{theorem}\label{loc-strong}
        Fix $k\in \mathbb{N}$, $2\leq m<4$, and let $E_{\text{d}}$, $E_{\text{c}}$ be as in Theorem \ref{invarint}. Then there exists $r_{3}>0$ and a unique $C^{\infty}$ function $f:\{w_{f}\in E_{\text{c}}\ |\ \|w\|_{m\alpha}<r_{3}\}\to E_{\text{d}}$ with $f(0)=0$, $Df(0)=0$, such that the submanifold
        \begin{align*}
            W_{\text{c}}^{\text{loc}}=\{w_{f}+f(w_{f})\ | \ w_{f}\in E_{\text{c}},\ \|w_{f}\|_{m\alpha}< r_{3}\}
        \end{align*}
        satisfies, for any $\mu\in\left(\frac{k+2-2\alpha}{2\alpha},\,\min\left\{\frac{k+3-2\alpha}{2\alpha},\frac{m}{2}+\frac{1}{2\alpha}-1\right\}\right)$.
        \begin{align}\label{local-fast-estimate}
            W_{\text{c}}^{\text{loc}}=\bigg\{w\in L^{2}(m\alpha)\bigg||\|w\||_{m\alpha}<r_{3}, \underset{\tau\to \infty}{\limsup}\tau^{-1}\ln\|\Phi_{\tau}^{m\alpha}w\|_{m\alpha}\leq -\mu\bigg\},
        \end{align}
        where $|\|w\||_{m\alpha}=\max(\|P_{k}w\|_{m\alpha},\|Q_{k}w\|_{m\alpha})$. In particular, if $w_{0}\in W_{\text{c}}^{\text{loc}}$, there exists $T\geq 0$ such that $\Phi_{\tau}^{m\alpha}\in W_{\text{c}}^{\text{loc}}$ for all $\tau\geq T$.
    \end{theorem}
    \begin{proof}
        Choose $r_{3}>0$ sufficiently small so that $\|\Phi_{\tau}^{m\alpha}w_{0}\|_{m\alpha}\leq r_{0}$ for all $\tau\geq 0$ whenever $\|w_{0}\|_{m\alpha}\leq 2r_{3}$. Take the function $h(\cdot,\cdot)$ of the point 2 in Theorem \ref{invarint}, and define 
        \begin{align*}
            f(w_{f})=h(0,w_{f})\quad \text{for all}\ w_{f}\in E_{\text{c}}\ \text{with}\ \|w_{f}\|_{m\alpha}<r_{3}.
        \end{align*}
        Then $f$ is of class $C^{1}$, $f(0)=0$, and the characterization \eqref{local-fast-estimate} with $\mu=\mu_{2}$ follows immediately from \eqref{3.18} with $w=0$. In particular, $f$ is unique. Moreover, since any solution $w(\tau)$ on $W_{\text{c}}^{\text{loc}}$ converges to zero as $\tau\to +\infty$, and set $\mu_*:=\min\left\{\frac{k+3-2\alpha}{2\alpha},\frac{m}{2}+\frac{1}{2\alpha}-1\right\}$. Since \(\mu_2\in\left(\frac{k+2-2\alpha}{2\alpha},\mu_*\right)\) was arbitrary, the characterization \eqref{local-fast-estimate} holds for any $\mu\in\left(\frac{k+2-2\alpha}{2\alpha},\mu_*\right)$. Finally, the smoothness of $f$ and the fact that $Df(0)=0$ can be proved using the integral equation satisfied by $f$. (see \cite[Section 5.2]{DH 81}).
    \end{proof}

    For later use, we introduce the following decay set associated with the strong stable directions. Let $\mu_*:=\min\left\{\frac{k+3-2\alpha}{2\alpha},\frac{m}{2}+\frac{1}{2\alpha}-1\right\}$. For any $\mu\in\left(\frac{k+2-2\alpha}{2\alpha},\,\mu_*\right)$, we define
    \begin{align*}
        W_{\mathrm c}:=\left\{w_0\in L^2(m\alpha)\,\Bigg|\,\limsup_{\tau\to+\infty}\frac1\tau\ln\|w(\tau)\|_{m\alpha}\le -\mu\right\},
    \end{align*}
    where \(w(\tau)\) denotes the solution of \eqref{nve} with initial data \(w_0\).  Thus \(W_{\mathrm c}^{\mu}\) consists of those initial data whose corresponding solutions converge to the origin at least at the exponential rate \(e^{-\mu\tau}\). We emphasize that \(W_{\mathrm c}^{\mu}\) is used here as a decay set; no global manifold structure is asserted in the fractional case.
    
    \begin{remark}\label{rem:global-extension}
        It is instructive to compare the present result with the classical Navier--Stokes equations ($\alpha=1$). In the integer-order case, the local strong stable manifold $W_{\mathrm{c}}^{\mathrm{loc}}$ can be extended to a global invariant manifold $W_{\text{c}}\subset L^2(m)$ by exploiting the backward uniqueness property of the vorticity flow (see, e.g., \cite{DH 81}). However, for the fractional case ($0<\alpha<1$), the backward uniqueness for the dissipative equation involving $(-\Delta)^\alpha$ remains an open problem. Consequently, the standard globalization techniques (such as Henry's geometric theory) are not applicable here. Therefore, in this work, we restrict our analysis to the local invariant manifold $W_{\mathrm{c}}^{\mathrm{loc}}$, which is sufficient to characterize the fine asymptotic behavior of small solutions.
    \end{remark}
    Finally, we will show that invariant manifold is stable for $\alpha$.
    
    \begin{proposition}\label{alpha-stability}
        Fix $k\in\mathbb N$ and choose $2\le m<4$. Let \(I=[\alpha_-,\alpha_+]\subset(1/2,1)\) be a compact interval such that $m\alpha_->k+1$. We consider \eqref{nve} on $L^{2}(m\alpha)$. Let $L^{2}(m\alpha)=E_{\mathrm{d}}(\alpha)\oplus E_{\mathrm{c}}(\alpha)$ be the spectral splitting defined above. Then there exist constants $r_0>0$ and $M>0$ such that, for every $\alpha\in I$, there is a map
        \begin{align*}
            g_\alpha: B_{E_{\mathrm{d}}(\alpha)}(0,r_0)\to E_{\mathrm{c}}(\alpha),
        \end{align*}
        with $g_\alpha(0)=0$ and $Dg_\alpha(0)=0$, which is $C^{1}$ and globally Lipschitz on $B_{E_{\mathrm{d}}(\alpha)}(0,r_0)$ and satisfies $\mathrm{Lip}(g_\alpha)\le M$ (with $M$ independent of $\alpha$). Moreover, the graph
        \begin{align*}
            W_{\mathrm{d}}^{\mathrm{loc}}(\alpha):=\{x+g_\alpha(x)\,:\,x\in B_{E_{\mathrm{d}}(\alpha)}(0,r_0)\}
        \end{align*}
        is a local invariant slow manifold for the semiflow generated by \eqref{nve}. Finally, the map $\alpha\mapsto g_\alpha$ is Lipschitz continuous.
    \end{proposition}
    \begin{proof}
        First, we recall the spectral information from Appendix B. The spectrum of the linear evolution operator $\Lambda_\alpha=e^{\mathcal L_\alpha}$ on $L^2(m\alpha)$ splits as a disjoint union $\sigma(\Lambda_\alpha)=\Sigma_d(\alpha)\cup\Sigma_{c}(\alpha)$, where
        \begin{align*}
            \Sigma_d(\alpha)=\Bigl\{e^{-\frac{n}{2\alpha}+1-\frac1\alpha}\ \Big|\ n=0,1,2,\dots\Bigr\},\qquad \Sigma_{c}(\alpha)\subset\Bigl\{\lambda\in\mathbb C\ \Big|\ |\lambda|\le e^{\frac{2\alpha-1-m\alpha}{2\alpha}}\Bigr\}.
        \end{align*}
        In particular, $\Lambda_\alpha$ admits at least $k+1$ simple eigenvalues $\lambda_0(\alpha),\dots,\lambda_k(\alpha)\in\Sigma_d(\alpha)$, which are isolated and separated from the continuous spectrum. Hence there exist constants $\mu_1<\mu_2<0$ with $\mu_2-\mu_1>0$ and $C\ge1$, independent of $\alpha\in (1/2,1)$, such that
        \begin{align*}
            |\lambda_j(\alpha)|\ge e^{\mu_2}\quad(j=0,\dots,k),\qquad |\lambda|\le e^{\mu_1}\quad(\lambda\in\Sigma_{c}(\alpha)).
        \end{align*}
        Let $P_k(\alpha)$ be the Riesz projection onto the finite-dimensional space $E_{\mathrm{d}}(\alpha)$ spanned by the eigenfunctions corresponding to $\lambda_0(\alpha),\dots,\lambda_k(\alpha)$, and set $Q_k(\alpha)=I-P_k(\alpha)$ with $E_{\mathrm{c}}(\alpha)=Q_k(\alpha)L^2(m\alpha)$. By the Riesz projection formula and spectral mapping, the semigroups $S_\alpha(\tau)=e^{\tau\mathcal L_\alpha}$ satisfy the uniform bounds, see Proposition \ref{spectral es},
        \begin{align}\label{eq:semigroup-fast}
            \|S_\alpha(\tau)Q_k(\alpha)\|_{\mathcal L(L^2(m\alpha))} \le C e^{\mu_1\tau},\qquad \tau\ge0, \\ \label{eq:semigroup-slow}
            \|S_\alpha(-\tau)P_k(\alpha)\|_{\mathcal L(L^2(m\alpha))} \le C e^{-\mu_2\tau},\qquad \tau\ge0.
        \end{align}
        Moreover, the eigenvalues $\lambda_0(\alpha),\dots,\lambda_k(\alpha)$ are simple and uniformly separated from the rest of $\sigma(\Lambda_\alpha)$ for $\alpha\in I$. Hence the corresponding spectral (Riesz) projections $P_k(\alpha)$ and $Q_k(\alpha)=I-P_k(\alpha)$ are well-defined and depend Lipschitz continuously on $\alpha\in I$ in the operator norm by standard perturbation theory.

        Next, by the structure of the Navier--Stokes nonlinearity $\mathcal N(w,\alpha):=(\mathbf v\cdot\nabla)w$ and the local well-posedness estimates for \eqref{nve} in Proposition \ref{decomposition}, we have: for each $\alpha\in(1/2,1)$, $\mathcal N(\cdot,\alpha):L^2(m\alpha)\to L^2(m\alpha)$ is $C^1$ near $0$ with $\mathcal N(0,\alpha)=0$ and $D_w\mathcal N(0,\alpha)=0$, and there exist $r_0>0$ and $L>0$, independent of $\alpha\in(1/2,1)$, such that
        \begin{equation}\label{eq:F-Lip-w-final}
            \|\mathcal N(z_1,\alpha)-\mathcal N(z_2,\alpha)\|_{m\alpha} \le L\|z_1-z_2\|_{m\alpha},
        \end{equation}
        whenever $\|z_i\|_{m\alpha}\le r_0$, $i=1,2$. In addition, since the nonlinearity does not depend on $\alpha$, $\mathcal N(w,\alpha_1)=\mathcal N(w,\alpha_2)$ for all $w$ and $\alpha_1,\alpha_2\in(1/2,1)$.

        To begin with, we will identify of a fixed phase space. By the Lipschitz dependence of $P_k(\alpha)$ and $Q_k(\alpha)$, we use the standard identification operators (see, e.g., \cite[Ch. 2, \S4.2]{TK 95}) to identify the parameter-dependent space $L^2(m\alpha)$ with a fixed Banach space $X$, and $E_{\mathrm d}(\alpha),E_{\mathrm c}(\alpha)$ with fixed subspaces $E_{\mathrm d},E_{\mathrm c}$ so that $X=E_{\mathrm d}\oplus E_{\mathrm c}$. Under this identification, \eqref{nve} becomes
        \begin{align*}
            \dot z = \mathcal{L}_{\alpha}z + \mathcal{N}(z,\alpha), \qquad z=(x,y)\in E_{\mathrm{d}}\oplus E_{\mathrm{c}},
        \end{align*}
        and \eqref{eq:semigroup-fast}--\eqref{eq:semigroup-slow} as well as \eqref{eq:F-Lip-w-final} remain valid (possibly with different constants), uniformly for $\alpha\in I$.

        In addition, we focus on graph space and graph transform. Choose $r_{0}>0$ sufficiently small so that
        \begin{align*}
            LC\int_0^{+\infty}  e^{(\mu_1-\mu_2)s}\mathrm{d}s \le \frac{1}{2},
        \end{align*}
        which is possible since $\mu_2-\mu_1>0$. Fix $\ell>0$ and define $B_{\mathrm{d}}(r_0)=\{x\in E_{\mathrm{d}}:\|x\|_{X}\le r_0\}$ and
        \begin{align*}
            \mathcal X =\{g:B_{\mathrm{d}}(r_0)\to E_{\mathrm{c}}:\ g(0)=0,\ \mathrm{Lip}(g)\le\ell\},\qquad \|g\|_{\mathcal X}:=\mathrm{Lip}(g).
        \end{align*}
        For $g\in\mathcal X$, $\alpha\in I$ and $\xi\in B_{\mathrm{d}}(r_0)$, let $z_g(\cdot;\xi,\alpha)=(x_g(\cdot;\xi,\alpha),y_g(\cdot;\xi,\alpha))$ be the unique mild solution on $(-\infty,0]$ satisfying the Lyapunov--Perron relations
        \begin{align*}
            x_g(t;\xi,\alpha) &= S_d(\alpha,t)\xi + \int_t^0 S_d(\alpha,t-s)\,\mathcal{N}_d(z_g(s;\xi,\alpha),\alpha)\mathrm{d}s,\\
            y_g(t;\xi,\alpha) &= -\int_{-\infty}^t S_c(\alpha,t-s)\, \mathcal{N}_c(z_g(s;\xi,\alpha),\alpha)\mathrm{d}s,\qquad t\le0,
        \end{align*}
        where $S_d$ is the restriction to $E_d$ and $S_c$ is the semigroup on $E_c$ for $t\ge0$. Standard estimates using \eqref{eq:semigroup-fast}--\eqref{eq:semigroup-slow} and \eqref{eq:F-Lip-w-final} give $\|z_g(t;\xi,\alpha)\|_{X}\le r_0$ for all $t\le0$ when $\|\xi\|_{X}\le r_0$, uniformly in $\alpha\in I$. Define the graph transform $\mathcal T_\alpha:\mathcal X\to\mathcal X$ by
        \begin{align*}
            (\mathcal T_\alpha g)(\xi) := y_g(0;\xi,\alpha)= -\int_{-\infty}^0 S_c(\alpha,-s)\,\mathcal{N}_c(z_g(s;\xi,\alpha),\alpha)\mathrm{d}s.
        \end{align*}

        Furthermore, we will show that uniform contraction of $\mathcal{T}_{\alpha}$. Fix $g\in\mathcal X$ and $\xi_1,\xi_2\in B_{\mathrm{d}}(r_0)$. For $t\le0$, $u(t):=\|z_g(t;\xi_1,\alpha)-z_g(t;\xi_2,\alpha)\|_{X}$ satisfies
        \begin{align*}
            u(t)\le C e^{\mu_2 t}\|\xi_1-\xi_2\|_{X} + CL\int_t^0 e^{\mu_2(t-s)} u(s)\mathrm{d}s.
        \end{align*}
        By Gronwall's lemma,
        \begin{align*}
            \|z_g(t;\xi_1,\alpha)-z_g(t;\xi_2,\alpha)\|_{X} \le C e^{\mu_2 t}\|\xi_1-\xi_2\|_{X},\qquad t\le0.
        \end{align*}
        Substituting into the definition of $\mathcal T_\alpha$ and using \eqref{eq:semigroup-fast},
        \begin{align*}
            \|\mathcal T_\alpha g(\xi_1)-\mathcal T_\alpha g(\xi_2)\|_{X} \le C^2 L\int_0^{+\infty} e^{(\mu_1-\mu_2)s}\mathrm{d}s \ \|\xi_1-\xi_2\|_{X} \le \ell\|\xi_1-\xi_2\|_{X}.
        \end{align*}
        Thus $\mathcal T_\alpha g\in\mathcal X$ and $\|\mathcal T_\alpha g\|_{\mathcal X}\le \ell$. Similarly, for $g_1,g_2\in\mathcal X$ one obtains
        \begin{align*}
            \|\mathcal T_\alpha g_1-\mathcal T_\alpha g_2\|_{\mathcal X}\le \kappa\|g_1-g_2\|_{\mathcal X}
        \end{align*}
        for some $\kappa\in(0,1)$ independent of $\alpha\in I$. Hence each $\mathcal T_\alpha$ is a uniform contraction on $\mathcal X$.

        Moreover, we prove that dependence of $\mathcal{T}_{\alpha}$ on $\alpha$. Fix $g\in\mathcal X$ and $\alpha_1,\alpha_2\in I$, and set $z_i(t)=z_g(t;\xi,\alpha_i)$. Since $\mathcal N$ is independent of $\alpha$, only the linear parts vary with $\alpha$. Using the Lipschitz dependence of $\mathcal L_\alpha$ (hence of $S_d,S_c$) and the uniform bounds above, one obtains for $t\le0$ an inequality of the form
        \begin{equation}\label{z1-z2}
            \|z_1(t)-z_2(t)\|_{X}\le C|\alpha_1-\alpha_2|\int_{-\infty}^0 e^{\mu_1(t-s)}\|z_2(s)\|_{X}\mathrm{d}s + CL\int_t^0 e^{\mu_1(t-s)}\|z_1(s)-z_2(s)\|_{X}\mathrm{d}s.
        \end{equation}
        Since $\|z_2(s)\|_{X}\le r_0$ for $s\le0$, another application of Gronwall yields
        \begin{align*}
            \|z_1(t)-z_2(t)\|_{X}\le C'|\alpha_1-\alpha_2|e^{\mu_1 t},\qquad t\le0,
        \end{align*}
        with $C'$ independent of $\alpha$. Substituting into the definition of $\mathcal T_{\alpha_i}g$ and using again \eqref{eq:semigroup-fast}, we obtain
        \begin{align*}
            \|\mathcal T_{\alpha_1}g(\xi)-\mathcal T_{\alpha_2}g(\xi)\|_{X} \le C_0|\alpha_1-\alpha_2|\,\|\xi\|_{X},
        \end{align*}
        hence $\|\mathcal T_{\alpha_1}g-\mathcal T_{\alpha_2}g\|_{\mathcal X}\le C_0|\alpha_1-\alpha_2|$.

        Finally, we will study the fixed points and Lipschitz dependence of $g_{\alpha}$. For each $\alpha\in I$ there exists a unique fixed point $g_\alpha\in\mathcal X$ such that $\mathcal T_\alpha g_\alpha=g_\alpha$. The Lyapunov--Perron construction implies that the graph of $g_\alpha$ is a local invariant slow manifold of \eqref{nve}, and that $g_\alpha$ is $C^1$ with $g_\alpha(0)=0$ and $Dg_\alpha(0)=0$ (see, e.g., \cite{DH 81,PC 89}). Finally, for $\alpha_1,\alpha_2\in I$,
        \begin{align*}
            g_{\alpha_1}-g_{\alpha_2} = (\mathcal T_{\alpha_1}g_{\alpha_1}-\mathcal T_{\alpha_1}g_{\alpha_2}) +(\mathcal T_{\alpha_1}g_{\alpha_2}-\mathcal T_{\alpha_2}g_{\alpha_2}),
        \end{align*}
        so taking $\|\cdot\|_{\mathcal X}$ and using the contraction and the Lipschitz estimate in $\alpha$ gives
        \begin{align*}
            \|g_{\alpha_1}-g_{\alpha_2}\|_{\mathcal X} \le \kappa\|g_{\alpha_1}-g_{\alpha_2}\|_{\mathcal X}+ C_0|\alpha_1-\alpha_2|.
        \end{align*}
        Rearranging yields
        \begin{align*}
            \|g_{\alpha_1}-g_{\alpha_2}\|_{\mathcal X}\le \frac{C_0}{1-\kappa}\,|\alpha_1-\alpha_2|,
        \end{align*}
        which completes the proof.
    \end{proof}

    \begin{remark}
        To rigorously analyze the Lipschitz dependence of $g_\alpha$ on $\alpha$, we identify the parameter-dependent spaces $L^2(m\alpha)$ with a fixed Banach space $X$, exploiting the uniform equivalence of the weighted norms for $\alpha\in I$. Under this identification, the spectral projections $P_k(\alpha)$ vary Lipschitz continuously, allowing us to treat $E_{\mathrm{d}}(\alpha)$ and $E_{\mathrm{c}}(\alpha)$ as subspaces of $X$. The graph map $g_\alpha$ is then constructed via the backward Lyapunov--Perron integral, which enforces the boundedness condition needed to slave the fast variables to the slow ones ($y=g_\alpha(x)$).
    \end{remark}

    \begin{remark}
        A crucial consequence of the Lyapunov--Perron construction is the continuous dependence of the invariant manifold on the parameter $\alpha$. Since the spectral gap ensures that the dimension of the slow subspace remains constant ($k+1$), the manifold $W_{\mathrm{d}}^{\mathrm{loc}}(\alpha)$ undergoes a continuous deformation without bifurcation as $\alpha$ varies. In particular, in the limit $\alpha\to 1^{-}$, it converges to the classical Navier--Stokes slow manifold (see, e.g., \cite{TC 02}). This implies that the fractional slow manifold is a robust perturbation of the classical one: its geometry adjusts continuously with $\alpha$, while its dynamical role as the center of attraction remains unchanged.
    \end{remark}

    \section{The long-time asymptotics of solution}

    \subsection{The eigenfunction of operator $\mathcal{L}_{\alpha}$}
    In this subsection we refine the relation between the vorticity $w(\xi)$ and the associated velocity field $\mathbf v(\xi)$ in self-similar variables. We recover $\mathbf v$ from $w$ via the Biot--Savart law \eqref{nbs} and derive quantitative decay estimates for $\mathbf v(\xi)$ as $|\xi|\to\infty$, which will be used in the next subsection to transfer vorticity information to solutions of the fractional Navier--Stokes equation.

    We begin with the velocity field generated by the principal eigenfunction of the linearized operator $\mathcal L_\alpha$. The first eigenvalue is $\lambda_{0}=1-\frac1\alpha$, with the associated eigenfunction given by the fractional Oseen-type vortex $G_\alpha=\mathcal F^{-1}(e^{-|p|^{2\alpha}})$. Using the Biot--Savart law $\mathbf v=\nabla^\perp\Delta^{-1}w$ and exploiting the radial symmetry of $G_\alpha$, the corresponding velocity field $\mathbf{v}^{G_{\alpha}}$ can be explicitly computed (either via Fourier transform involving Bessel functions or directly via the circulation distribution). The result is given by
    \begin{align}\label{G-expre}
    \mathbf{v}^{G_{\alpha}}(x)=\frac{x^{\perp}}{|x|^{2}}\Phi(|x|),\qquad \text{with}\quad \Phi(r)=\int_{0}^{r}sG_{\alpha}(s)\mathrm{d}s.
    \end{align}
    Here, $\Phi(r)$ represents the partial circulation of the vortex. Finally, using the asymptotic decay of $G_\alpha$, we obtain the following far-field behavior:
    \begin{align*}
    \mathbf{v}^{G_{\alpha}}(x)=\frac{1}{2\pi}\frac{x^{\perp}}{|x|^{2}}+O(|x|^{-1-2\alpha}) \quad \text{as} \quad |x|\to \infty.
    \end{align*}

    The second eigenvalue $\lambda_{1}=\frac{2\alpha-3}{2\alpha}$ has multiplicity two, with eigenfunctions $F_{i,\alpha}=\partial_{i}G_{\alpha}$ for $i=1,2$. Due to the linearity of the Biot--Savart law, the associated velocity fields are given by $\mathbf{v}^{F_{i,\alpha}}=\partial_{i}\mathbf{v}^{G_{\alpha}}$. A direct computation using \eqref{G-expre} yields:
    \begin{align*}
        \mathbf{v}^{F_{i,\alpha}}(x) = \bigg(\frac{\Phi}{r^{2}}\bigg)'\frac{x_{i}}{r}x^{\perp} + \frac{\Phi}{r^{2}}\partial_{i}(x^{\perp}), \quad \text{where } \partial_1 x^\perp = e_2, \ \partial_2 x^\perp = -e_1.
    \end{align*}
    Using the asymptotic expansion of $\Phi(r)$, we derive the far-field behavior for $i=1,2$:
    \begin{align*}
        \begin{cases}
            \textbf{v}^{F_{1,\alpha}}=\frac{1}{2\pi |x|^{4}}(2x_{1}x_{2}, x^{2}_{2}-x^{2}_{1})+O(|x|^{-2-2\alpha}),\\
            \textbf{v}^{F_{2,\alpha}}=\frac{1}{2\pi|x|^{4}}(x^{2}_{2}-x^{2}_{1}, -2x_{1}x_{2})+O(|x|^{-2-2\alpha}).
        \end{cases}
    \end{align*}

    The third eigenvalue $\lambda_{2}=\frac{\alpha-2}{\alpha}$ has multiplicity three. A convenient basis of eigenfunction
    \begin{align*}
        H_{1,\alpha}(\xi)&=\Delta G_{\alpha}(\xi),\\
        H_{2,\alpha}(\xi)&=(\partial_{1}^{2}-\partial_{2}^{2})G_{\alpha}(\xi),\\
        H_{3,\alpha}(\xi)&=\partial_{1}\partial_{2}G_{\alpha}(xi).
    \end{align*}
    and the corresponding velocity fields are given by
    \begin{align*}
        \textbf{v}^{H_{1,\alpha}}(\xi)&=\Delta \textbf{v}^{G_{\alpha}}(\xi),\\
        \textbf{v}^{H_{2,\alpha}}(\xi)&=(\partial_{1}^{2}-\partial_{2}^{2})\textbf{v}^{G_{\alpha}}(\xi),\\
        \textbf{v}^{H_{3,\alpha}}(\xi)&=\partial_{1}\partial_{2}\textbf{v}^{G_{\alpha}}(\xi).
    \end{align*}

    \subsection{Stability of the self-similar fractional vortex $G_\alpha$}
    In the classical case $\alpha=1$, the Oseen vortex is an explicit self-similar solution describing the diffusion of a point vortex. For $0<\alpha<1$, the fractional heat kernel $G_\alpha$ provides a natural analogue (a "fractional Oseen-type vortex") for radially symmetric diffusive evolution. To the best of our knowledge, however, a complete dynamical theory for such fractional vortices (e.g. stability or invariant manifolds) is not yet available. In this work, $G_\alpha$ is used only as an eigenfunction of the linearized operator $\mathcal L_\alpha$.

    We study small solutions of \eqref{nve} in the weighted space $L^{2}(m\alpha)$ with $m=2$ and $1/2<\alpha<1$. On $L^{2}(2\alpha)$, $\mathcal L_\alpha$ has a simple eigenvalue $\lambda_{0}=1-\frac{1}{\alpha}$ with associated eigenfunction $G_\alpha$; recall that $G_{\alpha}=\mathcal{F}^{-1}(e^{-|p|^{2\alpha}})$. Let $\mathbf{v}^{G_{\alpha}}(\xi)$ be the velocity field obtained from $G_{\alpha}$ by \eqref{nbs}. Observing that $G_{\alpha}$ is a radial function and Biot-Savart law is independent of $\alpha$, we thus conclude that the resulting velocity field $\textbf{v}^{G_{\alpha}}$ must be a purely azimuthal term, which is orthogonal to the radial gradient.
    \begin{align}\label{orthogonal}
        \textbf{v}^{G_{\alpha}}\cdot\nabla G_{\alpha}=0,\qquad \xi\in \mathbb{R}^{2}.
    \end{align}
    Indeed, $\textbf{v}(\xi)=K\ast w(\xi)$ with $K(z)=\frac{1}{2\pi}\cdot\frac{z^{\perp}}{|z|^{2}}$, let $\tilde{\psi}$ be stream function. $\Delta\tilde{\psi}=w$, $\textbf{v}=\nabla^{\perp}\tilde{\psi}:=(-\partial_{2}\tilde{\psi},\partial_{1}\tilde{\psi})$. If $w$ be a radial function, then $\tilde{\psi}$ is also radial function. Under the polar coordinate, we have $\Delta \tilde{\psi}=\tilde{\psi}_{rr}+\frac{1}{r}\tilde{\psi}_{r}=w(r)$. Multiplying  both sides by $r$ and integrating, we obtain 
    \begin{align*}
        (r\tilde{\psi}_{r})'=rw(r)\Rightarrow r\tilde{\psi}_{r}(r)=\int_{0}^{r}sw(s)\mathrm{d}s.
    \end{align*}
    Hence, $\tilde{\psi}_{r}(r)=\frac{1}{r}\int_{0}^{r}sw(s)\mathrm{d}s$. Noting the polar coordinate basis vectors $e_{r}=\xi/|\xi|$ and $e_{\theta}=\xi^{\perp}/|\xi|$, we thus get
    \begin{align*}
        \textbf{v}(\xi)=\nabla^{\perp} \tilde{\psi}(\xi)=\tilde{\psi}_{r}(r)e_{\theta}=\frac{e_{\theta}}{r}\int_{0}^{r}sw(s)\mathrm{d}s.
    \end{align*}
    On the other hand, we have $\nabla w(\xi)=w'(r)e_{r}=w'(r)\frac{\xi}{|\xi|}$. Due to $e_{\theta}\perp e_{r}$, we have $\textbf{v}(\xi)\cdot\nabla w(\xi)=\tilde{\psi}_{r}(r)e_{\theta}w'(r)e_{r}=0$. Hence, we obtain \eqref{orthogonal}. However, it should be emphasized that, unlike the integer-order case, $w(\xi)=\beta G_{\alpha}$ is not a steady-state solution, since the corresponding eigenvalue is not zero. 

    Let $E_{\text{d}}=\text{span}\{G_{\alpha}\}$ and denote by $E_c$ the spectral subspace of $\mathcal L_{\alpha}$ associated with the continuous spectrum
    \begin{align*}
        \sigma_{c} = \{\lambda \in \mathbb C \mid \Re \lambda \le -\tfrac{1}{2}\}.
    \end{align*}
    Let $W_d^{\mathrm{loc}}$ be the local slow manifold given by Corollary \ref{local invariant} (with $k=0$ and $m=2$). We claim that
    \begin{align}\label{2-Wc-loc}
        W_{\text{d}}^{\mathrm{loc}}
        = \bigl\{\beta G_{\alpha} \;\big|\; \beta \in \mathbb R,\; |\beta|\,\|G_{\alpha}\|_{2\alpha} < r_0\bigr\}.
    \end{align}
    Indeed, if $|\beta_0|\,\|G_\alpha\|_{2\alpha}<r_0$ and we define $\beta(\tau):=\beta_0 e^{\lambda_0\tau}$, then $w(\xi,\tau):=\beta(\tau)\,G_\alpha(\xi)$ solves \eqref{nve}. Moreover, $\|w(\tau)\|_{2\alpha}=|\beta_0|e^{\lambda_0\tau}\|G_\alpha\|_{2\alpha}<r_0$ for all $\tau\ge 0$.  Hence, by Corollary \ref{local invariant} we have $w(\tau) \in W_{\text{d}}^{\mathrm{loc}}$. On the other hand, the description of the local slow manifold implies that $ W_{\text{d}}^{\mathrm{loc}} \subset \{\beta G_{\alpha} + g(\beta G_{\alpha}) \mid \beta \in \mathbb R\}$ for some mapping $g : E_{\text{d}} \to E_{\text{c}}$. Combining these two facts forces $g \equiv 0$, and therefore \eqref{2-Wc-loc} holds. In this particular situation the local slow manifold is thus unique. Applying Corollary \ref{local invariant}, we obtain:
    \begin{proposition}\label{2alpha}
        Fix $1/\alpha-1<\mu<\frac{1}{2\alpha}$. There exist positive constants $r_{2}$ and $C$ such that, for any initial data $w_{0}$ with $\|w_{0}\|_{2\alpha}\leq r_{2}$, the solution $w(\cdot,\tau)$ of \eqref{nve} satisfies:
        \begin{align}\label{2-loc}
            \|w(\cdot,\tau)-Ae^{\lambda_{0}\tau}G_{\alpha}\|_{2\alpha}\leq Ce^{-\mu\tau},\qquad \tau\geq 0,
        \end{align}
        where $A=\int_{\mathbb{R}^{2}}w_{0}(\xi)\mathrm{d}\xi$.
    \end{proposition}
    \begin{proof}
        If $r_{2}>0$ is sufficiently small, the invariant manifold estimates \eqref{loc-invar-estimate} and \eqref{2-Wc-loc} imply that \eqref{2-loc} holds for some constant $A\in \mathbb{R}$. It remains to identify $A$ with the total mass of the initial data. Let $\beta(\tau)=\int_{\mathbb{R}^{2}}w(\xi,\tau)\mathrm{d}\xi$ and integrate the vorticity equation \eqref{nve} over $\mathbb{R}^2$, we obtain
        \begin{align*}
            \beta'(\tau) = -\int_{\mathbb{R}^{2}}(-\Delta)^{\alpha}w\,\mathrm{d}\xi + \frac{1}{2\alpha}\int_{\mathbb{R}^{2}}\xi\cdot\nabla w\,\mathrm{d}\xi + \int_{\mathbb{R}^{2}}w\,\mathrm{d}\xi - \int_{\mathbb{R}^{2}}\nabla\cdot(\mathbf{v}w)\,\mathrm{d}\xi.
        \end{align*}
        For the first term vanishes (since $\widehat{(-\Delta)^{\alpha}w}(0)=0$). For the nonlinear term, the divergence theorem and the decay of solutions imply $\int_{\mathbb{R}^{2}}\nabla\cdot(\mathbf{v}w)\,\mathrm{d}\xi = 0$. For the second term, integration by parts yields
        \begin{align*}
            \frac{1}{2\alpha}\int_{\mathbb{R}^{2}}\xi\cdot\nabla w\,\mathrm{d}\xi = -\frac{1}{2\alpha}\int_{\mathbb{R}^{2}}(\nabla\cdot\xi) w\,\mathrm{d}\xi = -\frac{1}{\alpha}\int_{\mathbb{R}^{2}}w\,\mathrm{d}\xi,
        \end{align*}
        where we used $\nabla\cdot\xi = 2$ in $\mathbb{R}^2$. Combining these results, we arrive at the ODE:
        \begin{align}\label{beta}
            \beta'(\tau) = \left(1-\frac{1}{\alpha}\right)\beta(\tau) = \lambda_{0}\beta(\tau).
        \end{align}
        Solving this gives $\beta(\tau)=e^{\lambda_{0}\tau}\beta(0) = e^{\lambda_{0}\tau}\int_{\mathbb{R}^{2}}w_{0}\,\mathrm{d}\xi$. On the other hand, integrating the asymptotic profile in \eqref{2-loc} and noting that $\int_{\mathbb{R}^{2}}G_{\alpha}\,\mathrm{d}\xi=1$, we conclude that the coefficient $A$ must satisfy $A=\int_{\mathbb{R}^{2}}w_{0}\,\mathrm{d}\xi$.
    \end{proof}
    Next, we revert to unscaled variables $(x,t)$. Let 
    \begin{align*}
        \Omega(x,t)=\frac{1}{1+t}G_{\alpha}\bigg(\frac{x}{(1+t)^{1/2\alpha}}\bigg),\quad \textbf{u}^{\Omega}(x,t)=\frac{1}{(1+t)^{1-1/2\alpha}}\textbf{v}^{G_{\alpha}}\bigg(\frac{x}{(1+t)^{1/2\alpha}}\bigg).
    \end{align*}
    Thus $\Omega$ is the solution of \eqref{ve} corresponding, via the change of variable \eqref{w-omega}, to the solution $G_{\alpha}$ of \eqref{nve}, and $\textbf{u}^{\Omega}$ is the associated velocity field. From Proposition \ref{2alpha}, we obtain:
    \begin{corollary}\label{2alpha-coro}
        Fix $1/\alpha-1<\mu<\frac{1}{2\alpha}$. There exists $r_{2}>0$ such that, for all initial  data $w_{0}\in L^{2}(2\alpha)$ with $\|w_{0}\|_{2\alpha}\leq r_{2}$, the solution of \eqref{ve} satisfies
        \begin{align}\label{4.5}
            |\omega-A(1+t)^{\lambda_{0}}\Omega(\cdot,t)|_{p}\leq \frac{C_{p}}{(1+t)^{1+\mu-\frac{1}{p\alpha}}}, \quad 1\leq p\leq 2,\quad t\geq 0,
        \end{align}
        where $A=\int_{\mathbb{R}^{2}}\omega_{0}(x)\mathrm{d}x$. If $\textbf{u}(x,t)$ is the velocity field obtained from $\omega(x,t)$ via the Biot-Savart law \eqref{B-S}, then
        \begin{align}\label{4.6}
            |\textbf{u}(\cdot,t)-A(1+t)^{\lambda_{0}}\mathbf{v}^{\Omega}(\cdot,t)|_{q}\leq \frac{C_{q}}{(1+t)^{1-\frac{1}{2\alpha}+\mu-\frac{1}{q\alpha}}}, \quad 1<q<\infty,\quad t\geq 0.
        \end{align}
    \end{corollary}
    \begin{proof}
        Let $\omega(x,t)$ be the solution of \eqref{ve} with $\omega(\cdot,0)=\omega_{0}$, and let $w(\xi,\tau)$ be the solution of \eqref{nve} with the same initial data. If $1\leq p\leq 2$, then $L^{2}(2\alpha)\hookrightarrow L^{p}(\mathbb{R}^{2})$. Using \eqref{w-omega} and \eqref{2-loc}, we get
        \begin{align*}
            |\omega(\cdot,t)-A(1+t)^{\lambda_{0}}\Omega(\cdot,t)|&=(1+t)^{-1+\frac{1}{p\alpha}}|w(\cdot,\log(1+t))-A(1+t)^{\lambda_{0}}G_{\alpha}(\cdot)|_{p}\\
            &\leq C(1+t)^{\lambda_{0}}\|w(\cdot,\log(1+t))-A(1+t)^{\lambda_{0}}G_{\alpha}(\cdot)\|_{2\alpha}\\
            &\leq C(1+t)^{-1-\mu+\frac{1}{p\alpha}}.
        \end{align*}
        For velocity field, according to Lemma \ref{sol-e}, $q\in (2,\infty)$, we get
        \begin{align*}
            |\textbf{u}(\cdot,t)-A(1+t)^{\lambda_{0}}\textbf{v}^{\Omega}(\cdot,t)|_{q}\leq C|\omega(\cdot,t)-A(1+t)^{\lambda_{0}}\Omega(\cdot,t)|_{p}\leq C(1+t)^{-1+\frac{1}{2\alpha}-\mu+\frac{1}{q\alpha}},\qquad t\geq 0.
        \end{align*}
        Finally, assume that $1<q\leq 2$, and fix $m\alpha \in (2/q,2)$. If $\tilde{w}(\tau)=w(\tau)-A(1+t)^{\lambda_{0}} G_{\alpha}$ and if $\tilde{\textbf{v}}(\tau)$ denotes the corresponding velocity field, it follows from \cite[Proposition B.1]{TC 02} and H\"older's inequality that 
        \begin{align*}
            |\tilde{\textbf{v}}(\tau)|_{q}\leq C|b^{m\alpha-\tfrac{1}{2}}\tilde{\textbf{v}}(\tau)|_{4}\leq C|b^{m\alpha}\tilde{w}(\tau)|_{2}\leq C\|\tilde{w}(\tau)\|_{2\alpha}\leq Ce^{-\mu\tau},\quad \tau\geq 0,
        \end{align*}
        where $b(\xi)=(1+|\xi|^{2})^{1/2}$. Using the change of variables \eqref{v-u}, we thus obtain \eqref{4.6} for $1<q\leq 2$.
    \end{proof}

    \begin{remark}
        Corollary \ref{2alpha-coro} can be viewed as the fractional counterpart of the $L^p$--asymptotic estimates for the classical two-dimensional Navier--Stokes equations obtained in earlier works such as \cite{TC 02, YT 88}. In the integer-order setting the Gaussian Oseen vortex plays the role of the dominant asymptotic profile, whereas in our fractional framework the leading mode is the fractional Oseen-type vortex $(1+t)^{\lambda_0}\Omega^\alpha$. Moreover, the decay is measured in the weighted space $L^2(2\alpha)$, which is well adapted to the nonlocal diffusion operator $(-\Delta)^\alpha$.
    \end{remark}

    \begin{remark}
        Corollary \ref{2alpha-coro} asserts that for sufficiently small $w_0\in L^2(2\alpha)$ the solution $w(\cdot,t)$ converges to the one-parameter family of fractional Oseen-type vortices $A(1+t)^{\lambda_0}\Omega^\alpha(\cdot,t)$, $A=\int_{\mathbb R^2}w_0(x)\mathrm{d}x$, with an explicit algebraic rate in $L^p$, $1\le p\le2$. More precisely, \eqref{4.5} yields
        \begin{align*}
            |w(\cdot,t)-A(1+t)^{\lambda_0}\Omega^\alpha(\cdot,t)|_{p} \lesssim (1+t)^{-1-\mu+\frac{1}{p\alpha}},\quad 1\le p\le2,\ t\ge0,
        \end{align*}
        so the remainder decays faster than the leading profile. By the Biot--Savart law, the velocity $u(\cdot,t)$ satisfies the corresponding bound \eqref{4.6} in $L^q$, $1<q<\infty$. Consequently, the large-time behaviour of small data in $L^2(2\alpha)$ is governed by the fractional Oseen-type vortex, which serves as a nonlinear attractor for the fractional Navier--Stokes flow.
    \end{remark}
    
    So far we have worked in the weighted space $L^{2}(2\alpha)$ (i.e.\ $m=2$). In this setting the relevant part of the spectrum of $\mathcal L_\alpha$ consists only of the simple eigenvalue $\lambda_0$ with eigenfunction $G_\alpha$, and Corollary \eqref{2alpha-coro} shows that small solutions are governed to leading order by the fractional vortex profile $A(1+t)^{\lambda_0}\Omega(\cdot,t)$.

    If we strengthen the topology to $L^{2}(3\alpha)$ ($m=3$), the spectral picture depends on $\alpha$. For $1/2<\alpha\le 2/3$, $\lambda_0$ remains the only eigenvalue in the right half-plane, so the slow subspace is one-dimensional and Corollary \eqref{2alpha-coro} still captures the large-time behaviour in $L^{2}(3\alpha)$. For $2/3<\alpha<1$, however, an additional eigenvalue $\lambda_1=\frac{2\alpha-3}{2\alpha}$ enters the discrete spectrum with multiplicity two, with eigenfunctions $F_{1,\alpha},F_{2,\alpha}$; thus the slow subspace becomes three-dimensional, $\mathrm{span}\{G_\alpha,F_{1,\alpha},F_{2,\alpha}\}$. In this regime the invariant-manifold approach yields a refined asymptotic description by capturing the $\lambda_1$-modes: solutions with small data approach the local slow manifold rapidly, after which the dynamics is effectively finite-dimensional and provides the next-order correction in the $F_{1,\alpha},F_{2,\alpha}$ directions. This leads to the expansion in Theorem \ref{3-alpha}.

    \begin{theorem}\label{3-alpha}
        Fix $\frac{3}{2\alpha}-1<\mu<\min\Bigl\{\frac{\alpha+1}{2\alpha},\;\frac{5}{2\alpha}-2\Bigr\}$ and $2/3<\alpha<1$. There exists $r_{2}>0$ and $C>0$ such that, for initial data $w_{0}\in L^{2}(3\alpha)$ with $\|w_{0}\|_{3\alpha}\leq r_{2}$, the solution $w(\cdot,\tau)$ of \eqref{nve} satisfies
        \begin{align}
            \|w(\xi,\tau)-Ae^{\lambda_{0}\tau}G_{\alpha}(\xi)+(B_{1}F_{1,\alpha}+B_{2}F_{2,\alpha})e^{\lambda_{1}\tau}\|_{3\alpha}\leq Ce^{-\mu\tau},\qquad \tau\geq 0,
        \end{align}
        where $A=\int_{\mathbb{R}^{2}}w_{0}(\xi)\mathrm{d}\xi$, $B_{1}=\int_{\mathbb{R}^{2}}\xi_{1}w_{0}(\xi)\mathrm{d}\xi$ and $B_{2}=\int_{\mathbb{R}^{2}}\xi_{2}w_{0}(\xi)\mathrm{d}\xi$.
    \end{theorem}    
    \begin{proof}
        Let $E_{\mathrm d}=\mathrm{span}\{G_\alpha,\partial_1G_\alpha,\partial_2G_\alpha\}$ and $E_{\mathrm c}$ be the spectral subspace associated with $\sigma_c=\{\lambda\in\mathbb C:\Re\lambda\le-(\alpha+1)/(2\alpha)\}$. By Corollary \ref{local invariant} (with $k=1$, $m=3$), any $w\in W_{\mathrm d}^{\mathrm{loc}}$ admits the unique decomposition
        \begin{equation}\label{5.9}
            w=\beta G_\alpha+\zeta_1\partial_1G_\alpha+\zeta_2\partial_2G_\alpha+g(\beta,\boldsymbol\zeta),\qquad g(\beta,\boldsymbol\zeta)\in E_{\mathrm c}.
        \end{equation}
        If $w$ solves \eqref{nve}, then $\dot\beta=(1-\frac1\alpha)\beta$ (Proposition \ref{2alpha}).

        Multiplying \eqref{5.9} by $\xi_j$ and integrating over $\mathbb R^2$, radial symmetry and integration by parts give
        \begin{align*}
            \int_{\mathbb R^2}\xi_j G_\alpha\mathrm{d}\xi=0,\qquad \int_{\mathbb R^2}\xi_j\partial_i G_\alpha\mathrm{d}\xi=-\delta_{ij},
        \end{align*}
        hence
        \begin{align*}
            \zeta_i(\tau)=-\int_{\mathbb R^2}\xi_i w(\xi,\tau)\mathrm{d}\xi, \qquad\text{and}\qquad \int_{\mathbb R^2}\xi_i g(\xi,\tau)\mathrm{d}\xi=0.
        \end{align*}
        Differentiating and using \eqref{nve} yields
        \begin{align*}
            \dot\zeta_i=-\int_{\mathbb R^2}\xi_i\mathcal L_\alpha w\mathrm{d}\xi+\int_{\mathbb R^2}\xi_i(\mathbf v\cdot\nabla)w\mathrm{d}\xi .
        \end{align*}
        Moreover,
        \begin{align*}
            \int_{\mathbb R^2}\xi_i(\xi\cdot\nabla w)\mathrm{d}\xi =-\int_{\mathbb R^2}\nabla\cdot(\xi_i\xi)\,w\mathrm{d}\xi =3\int_{\mathbb R^2}\xi_i w\mathrm{d}\xi=-3\zeta_i, \qquad\int_{\mathbb R^2}\xi_i(-\Delta)^\alpha w\mathrm{d}\xi=0 \ \bigl((-\Delta)^\alpha\xi_i=0\bigr),
        \end{align*}
        and
        \begin{align*}
            \int_{\mathbb R^2}\xi_i(\mathbf v\cdot\nabla)w\mathrm{d}\xi =-\int_{\mathbb R^2}v_i w\mathrm{d}\xi=0,
        \end{align*}
        since $\widehat v_i(k)=-i\frac{k_i^\perp}{|k|^2}\widehat w(k)$ makes the integrand odd. Therefore,
        \begin{align*}
            \dot\zeta_i=\Bigl(1-\frac{3}{2\alpha}\Bigr)\zeta_i,\qquad i=1,2.
        \end{align*}

        Next, we will estimate the nonlinear term $g(\beta,\boldsymbol{\zeta})$. Perform the $Q$-projection on \eqref{5.9} and let $g(\tau)=Qw(\tau)$ satisfy $\partial_{\tau}g(\tau)=\mathcal{L}_{\alpha}g(\tau)-Q\nabla\cdot(\textbf{v}w)$. To estimate the nonlinear term, we use the smoothing properties of the semigroup $e^{\tau\mathcal{L}_{\alpha}}$ combined with $L^q$ estimates. Fix $q \in (1, 2)$ sufficiently close to 2. According to Proposition \ref{group-est}, we have
        \begin{align*}
            \|\nabla\cdot e^{\tau \mathcal{L}_{\alpha}}f\|_{3\alpha}\leq Ct^{-\gamma}e^{-\sigma t}\|f\|_{3\alpha},
        \end{align*}
        where $\gamma=\frac{1}{2\alpha}$ and $\sigma=\frac{\alpha+1}{2\alpha}$. Following the proof of Theorem \eqref{global}, we have
        \begin{align*}
            \|\textbf{v}w\|_{L^q(3\alpha)} \leq C \|w\|_{3\alpha}^2.
        \end{align*}
        Moreover 
        \begin{align*}
            g(\tau)=e^{\tau\mathcal{L}_{\alpha}}h_{0}-\int_{0}^{\tau}e^{-\frac{1}{2\alpha}(\tau-s)}\nabla\cdot e^{(\tau-s)\mathcal{L}_{\alpha}}Q(\textbf{v}w)\mathrm{d}s,
        \end{align*}
        where $g(0)=g_{0}$. Then 
        \begin{align}\label{5.10}
            \|g(\tau)\|_{3\alpha}\leq Ce^{-\sigma\tau}\|g_{0}\|_{3\alpha}+C\int_{0}^{\tau}(\tau-s)^{-\tfrac{1}{2\alpha}}e^{-\sigma(\tau-s)}\|w(s)\|^{2}_{3\alpha}\mathrm{d}s.
        \end{align}
        Since 
        \begin{align*}
            \|w(s)\|_{3\alpha}\leq C(|\beta(s)|+|\boldsymbol{\zeta}|+\|g(s)\|_{3\alpha})=C(|A|e^{\lambda_{0}\tau}+|B|e^{\lambda_{1}s}+\|g(s)\|_{3\alpha}).
        \end{align*}
        Squaring both sides, we obtain
        \begin{align}\label{5.11}
            \|w(s)\|_{3\alpha}^{2}\leq &C(|A|^{2}e^{2\lambda_{0}s}+|B|^{2}e^{2\lambda_{1}s}+|A||B|e^{(\lambda_{0}+\lambda_{1})s}\notag\\
            &+(|A|e^{\lambda_{0}s}+|B|e^{\lambda_{1}s})\|g(s)\|_{3\alpha}+\|g(s)\|_{3\alpha}^{2}).
        \end{align}

        \begin{remark}\label{rem:AA-vanish}
            Let $P$ be the spectral projection onto the slow space $E_{\mathrm{d}}=\mathrm{span}\{G_\alpha,F_{1,\alpha},F_{2,\alpha}\}$ and set $Q:=I-P$. Decompose $w=\beta(\tau)G_\alpha+\boldsymbol{\zeta}(\tau)\!\cdot\!\boldsymbol F+g(\tau)$, $Pw=\beta G_\alpha+\boldsymbol{\zeta}\!\cdot\!\boldsymbol F$, $Qw=g$. Since $E_{\mathrm{d}}$ is invariant under $\mathcal L_\alpha$, we have $P\mathcal L_\alpha w=\lambda_0\,\beta\,G_\alpha+\lambda_1\,\boldsymbol{\zeta}\!\cdot\!\boldsymbol F$, $Q\mathcal L_\alpha w=\mathcal L_\alpha g$. Thus the linear part is purely first order in $(\beta,\boldsymbol{\zeta},g)$ and produces no quadratic interactions; all product terms come from the bilinear nonlinearity $\mathcal N(w)$, hence any quadratic forcing in the $g$--equation arises from $Q\mathcal N(w)$. Moreover, the $|A|^{2}$ self-interaction cancels: $\mathcal N(\beta G_\alpha)=-\nabla_\xi\!\cdot\bigl(\mathbf v[\beta G_\alpha]\ \beta G_\alpha\bigr)=-\beta^2\,\nabla_\xi\!\cdot\bigl(\mathbf v^{G_\alpha}G_\alpha\bigr)\equiv0$, since $\nabla\!\cdot(\mathbf v^{G_\alpha}G_\alpha)=0$. Consequently, $\partial_\tau g=\mathcal L_\alpha g+Q\mathcal N(w)$ contains no $|A|^2$ forcing; $Q\mathcal N(w)$ consists only of mixed terms (e.g.\ $AB$, $B^2$, $Ah$, $Bh$, $g^2$). This cancellation relies only on the 2D Biot--Savart structure and the radial symmetry of $G_\alpha$, and holds for both integer and fractional diffusion.
        \end{remark}

        Let $k(t):=t^{-\frac{1}{2\alpha}}e^{-\sigma t}$, $t>0$. We recall the following integral estimate.
        \begin{lemma}
            For any local bounded and non-negative function $X$ and any $\tilde{\sigma}>0$ satisfy $\tfrac{3}{2\alpha}-1<\mu<\tilde{\sigma}$, define 
            \begin{align*}
                Y(\tau):=e^{\mu\tau}\int_{0}^{\tau}k(\tau-s)X(s)\mathrm{d}s.
            \end{align*}
            then
            \begin{align}\label{15}
                \sup\limits_{\tau\geq 0}Y(\tau)\leq C(\tilde{\sigma}-\mu)^{\frac{1}{2\alpha}-1}\sup\limits_{s\geq 0}e^{\mu s}X(s).
            \end{align}
        \end{lemma}
        \begin{proof}   
            Assume that $\tilde{k}_{\mu}(t):=e^{\mu t}k(t)=t^{-\frac{1}{2\alpha}}e^{-(\tilde{\sigma}-\mu)t}$. Hence 
            \begin{align*}
                Y(\tau)=\int_{0}^{\tau}\tilde{k}_{\mu}(\tau-s)e^{\mu s}X(s)\mathrm{d}s\leq |\tilde{k}_{\mu}|_{1}\sup\limits_{s\leq \tau}e^{\mu s}X(s).
            \end{align*}
            A direct calculation yields  
            \begin{align*}
                |\tilde{k}_{\mu}|_{1}=\int_{0}^{\infty}t^{-\frac{1}{2\alpha}}e^{-(\tilde{\sigma}-\mu)t}\mathrm{d}t=\Gamma(1-\tfrac{1}{2\alpha})(\tilde{\sigma}-\mu)^{\frac{1}{2\alpha}-1}.
            \end{align*}
            Thus, the inequality holds with $C=\Gamma(1-\tfrac{1}{2\alpha})$.
        \end{proof}
        Plugging \eqref{5.11} into \eqref{5.10}, assume $M(\tau)=\sup\limits_{0\leq s\leq \tau}e^{\mu s}\|g(s)\|_{3\alpha}$, $\tfrac{3}{2\alpha}-1<\mu<\tilde{\sigma}$. Multiplying \eqref{5.10} by $e^{\mu \tau}$, we obtain
        \begin{align}\label{}
            M(\tau)\leq &Ce^{(-\tilde{\sigma}-\mu)\tau}\|g_{0}\|_{3\alpha}\notag\\
            &+\frac{C}{(\tilde{\sigma}-\mu)^{1-\gamma}}\bigg[|A||B|\sup\limits_{s\leq \tau}e^{(\mu+\lambda_{0}+\lambda_{1})s}+|B|^{2}\sup\limits_{s\leq \tau}e^{(\mu+2\lambda_{1})s}\bigg]\notag\\
            &+\frac{C}{(\tilde{\sigma}-\mu)^{1-\gamma}}\bigg[|A|\sup\limits_{s\leq \tau}e^{(\mu+\lambda_{0})s}\|g(s)\|^{*}_{3\alpha}+|B|\sup\limits_{s\leq \tau}e^{(\mu+\lambda_{1})s}\|g(s)\|^{*}_{3\alpha}\bigg]\notag\\
            &+\frac{C}{(\tilde{\sigma}-\mu)^{1-\gamma}}\sup\limits_{s\leq \tau}e^{\mu s}\|g(s)\|_{3\alpha}^{2},
        \end{align}
        where $\|g(s)\|_{3\alpha}^{*}:=\sup\limits_{s\leq \tau}\|g(s)\|_{3\alpha}\leq e^{-\mu s}M(\tau)$. Moreover 
        \begin{align*}
            M(\tau)\leq &Ce^{-(\tilde{\sigma}-\mu)\tau}\|g_{0}\|_{3\alpha}+\frac{C}{(\tilde{\sigma}-\mu)^{1-\gamma}}(|A||B|\underbrace{\sup\limits_{s\leq \tau}e^{(\mu+\lambda_{0}+\lambda_{1})s}}_{I_{1}}+|B|^{2}\underbrace{\sup\limits_{s\leq \tau}e^{(\mu+2\lambda_{1})s}}_{I_{2}})\\
            &+\frac{C}{(\tilde{\sigma}-\mu)^{1-\gamma}}[(|A|+|B|)M(\tau)+M^{2}(\tau)].
        \end{align*}
        To ensure that $I_{1}$ and $I_{2}$ are bounded and do not diverge with $\tau$, we need $\mu+\lambda_{0}+\lambda_{1}\leq 0, \mu+2\lambda_{1}\leq 0$. That is,
        \begin{align*}
            \mu\leq -(\lambda_{0}+\lambda_{1})=\frac{5}{2\alpha}-2,\qquad \mu\leq -2\lambda_{1}=\frac{3}{\alpha}-2.
        \end{align*}
        For $2/3<\alpha<1$, so the second upper bound is looser than $\nu<1$. The one that actually takes effect is 
        \begin{align}
            \mu\leq \frac{5}{2\alpha}-2.
        \end{align}
        Then, we have
        \begin{align*}
            M(\tau)\leq C\|g_{0}\|_{3\alpha}+\frac{C}{(\sigma-\mu)^{1-\gamma}}(|A||B|+|B|^{2})+\frac{C}{(\sigma-\mu)^{1-\gamma}}[(|A|+|B|)M(\tau)+M^{2}(\tau)].
        \end{align*}
        Choose small initial value, $\tfrac{3}{2\alpha}-1\leq\mu\leq \tfrac{5}{2\alpha}-2$ such that
        \begin{align*}
            M(\tau)\leq C(\|g_{0}\|_{3\alpha}+|A||B|+|B|^{2}).
        \end{align*}
        According to $M(\tau)=\sup\limits_{s\leq \tau}e^{\mu s}\|g(s)\|_{3\alpha}$, we infer 
        \begin{align*}
            \|g(s)\|_{3\alpha}\leq Ce^{-\mu\tau},\qquad \tau\geq 0.
        \end{align*}
        In particular, if $A=0$, then the constraint on $AB$ and $A$ vanishes. At this point, $\mu<\frac{\alpha+1}{2\alpha}$.

        Assume now that $w(\cdot,\tau)$ is a solution of \eqref{nve} on $W_{\text{d}}^{\text{loc}}$, and let $\beta(\tau)=Ae^{\lambda_{0}\tau}$, $\zeta_{1}(\tau)=B_{1}e^{\lambda_{1}\tau}$ and $\zeta_{2}(\tau)=B_{2}e^{\lambda_{1}\tau}$, where $A$, $\zeta_{1}$ and $\zeta_{2}$ are as in Theorem \ref{3-alpha}. Then
        \begin{align*}
            w(\cdot,\tau)=\beta(\tau)G_{\alpha}+\zeta_{1}F_{1,\alpha}+\zeta_{2}F_{2,\alpha}+g(\beta(\tau),\boldsymbol{\zeta}(\tau)).
        \end{align*}

        On the other hand, if $\|w_{0}\|_{3\alpha}\leq r_{2}$ and if $w(\tau)$ is the solution of \eqref{nve} with initial data $w_{0}$, Corollary \ref{local invariant} shows that there exists a solution $\tilde{w}(\tau)$ on $W_{\text{d}}^{\text{loc}}$ such that $\|w(\tau)-\tilde{w}(\tau)\|_{3\alpha}\leq Ce^{-\tilde{\mu} \tau}$ with $\tilde{\mu}\in (\tfrac{3}{2\alpha}-1,\tfrac{\alpha+1}{2\alpha})$. Thus, we have
        \begin{align*}
            \|w(\tau)-\beta(\tau)G_{\alpha}-\boldsymbol{\zeta}(\tau)\cdot \boldsymbol{F}\|_{3\alpha}&\leq \|w(\tau)-\tilde{w}(\tau)\|_{3\alpha}+\|\tilde{w}(\tau)-\beta(\tau)G_{\alpha}-\boldsymbol{\zeta}\cdot\boldsymbol{F}\|_{3\alpha}\\
            &\leq C(e^{-\nu\tau}+e^{-\tilde{\mu}\tau})\leq Ce^{-\mu\tau}.
        \end{align*}
        Here $\mu$ can be chosen so that
        \begin{align*}
            \text{if } A = 0 \ ,\quad \frac{3}{2\alpha}-1 < \mu < \frac{\alpha+1}{2\alpha};\qquad \text{if } A \neq 0,\quad \frac{3}{2\alpha}-1 < \mu < \min\Bigl\{\frac{\alpha+1}{2\alpha},\;\frac{5}{2\alpha}-2\Bigr\}.
        \end{align*}
        The conclusion has been proved.
    \end{proof}
    \begin{remark}
        Although the above asymptotic expansion near the fractional Oseen-type vortex is derived under the restriction  $\frac{3}{2\alpha}-1<\mu<\min\Bigl\{\frac{\alpha+1}{2\alpha},\;\frac{5}{2\alpha}-2\Bigr\}$ on the weight parameter, the spectral facts we used are in fact independent of \(\mu\). More precisely, for any admissible value of \(\mu\) such that \(G_\alpha, F_{1,\alpha}, F_{2,\alpha}\in L^2(2\alpha)\), the operator \(\mathcal L_\alpha\) has the simple eigenvalue\(\lambda_0 = 1-\frac1\alpha\) with eigenfunction \(G_\alpha\), and the double eigenvalue \(\lambda_1 = \frac{2\alpha-3}{2\alpha}\) with eigenfunctions \(F_{1,\alpha},F_{2,\alpha}\).
    \end{remark}

    We can also rewrite the result of Theorem \ref{3-alpha} in terms of the unscaled variables as we did in Corollary \ref{2alpha-coro}. Define 
    \begin{align*}
        \omega_{\mathrm{app}}(x,t)&=\frac{A}{(1+t)^{1/\alpha}}G_{\alpha}\!\left(\frac{x}{(1+t)^{1/(2\alpha)}}\right)+\sum_{i=1}^{2}\frac{B_i}{(1+t)^{3/(2\alpha)}}F_{i,\alpha}\!\left(\frac{x}{(1+t)^{1/(2\alpha)}}\right),\\
        \mathbf u_{\mathrm{app}}(x,t)&=\frac{A}{(1+t)^{1/(2\alpha)}}\mathbf v^{G_{\alpha}}\!\left(\frac{x}{(1+t)^{1/(2\alpha)}}\right)+\sum_{i=1}^{2}\frac{B_i}{(1+t)^{1/\alpha}}\mathbf v^{F_{i,\alpha}}\!\left(\frac{x}{(1+t)^{1/(2\alpha)}}\right).
    \end{align*}

    \begin{corollary}
        Fix $\frac{3}{2\alpha}-1<\mu<\min\Bigl\{\frac{\alpha+1}{2\alpha},\;\frac{5}{2\alpha}-2\Bigr\}$ and $2/3<\alpha<1$. There exists $r_{2}>0$ such that, for all initial data $\omega_{0}\in L^{2}(3\alpha)$ with $\|\omega_{0}\|_{3\alpha}\leq r_{2}$, the solution $\omega(\cdot,t)$ of \eqref{ve} satisfies 
        \begin{align*}
            |\omega(\cdot,t)-\omega_{\text{app}}(\cdot,t)|_{p}\leq \frac{C_{p}}{(1+t)^{1+\mu-\tfrac{1}{p\alpha}}},\quad 1\leq p\leq 2,\quad t\geq 0,
        \end{align*}
        where $A=\int_{\mathbb{R}^{2}}\omega_{0}(x)\mathrm{d}x$ and $B_{i}=-\int_{\mathbb{R}^{2}}x_{i}\omega_{0}(x)\mathrm{d}x$, $i=1,2$. If $\textbf{u}(x,t)$ is the velocity field obtained from $\omega(x,t)$ via the Biot-Savart law \eqref{B-S}, then
        \begin{align*}
            |\textbf{u}(\cdot,t)-\textbf{u}_{\text{app}}(\cdot,t)|_{q}\leq \frac{C_{q}}{(1+t)^{1-\tfrac{1}{2\alpha}+\mu-\tfrac{1}{q\alpha}}},\quad 1\leq q<\infty,\quad t\geq 0.
        \end{align*}
    \end{corollary}

    \subsection{Attenuation in the strongly stable direction}
    In this subsection we study the detailed large-time asymptotics of small solutions near the zero state and obtain \emph{spectrally optimal} decay rates. Recent work on the 2D fractional Navier--Stokes equations with dissipation $(-\Delta)^\alpha$, $0<\alpha<1$, shows that for Leray solutions with $u_0\in L^1(\mathbb R^2)\cap L^2_\sigma(\mathbb R^2)$ one has the optimal energy upper bound $|u(t)|_{2}\le C(1+t)^{-1/\alpha}$, equivalently $(1+t)^{1/\alpha}|u(t)|_2$ is uniformly bounded. This estimate does not improve for small data at the level of energy methods. Here we show that, for sufficiently small data in suitable weighted spaces, the $L^2$ decay can be strictly faster; see Theorem \ref{opt}.

    Our approach adapts the dynamical-systems framework of Gallay--Wayne (e.g.\ \cite{TC 02}) to the fractional setting. In similarity variables and on weighted spaces $L^2(m\alpha)$, the spectrum of the linearized operator $\mathcal L_\alpha$ splits into a finite-dimensional slow subspace (spanned by $G_\alpha,F_{1,\alpha},F_{2,\alpha}$) and a strongly stable complement where the semigroup decays exponentially in similarity time. This yields local invariant manifolds near the origin and sharper decay information: the slow projection determines the leading algebraic rate in physical time, while the strong-stable component decays faster, at a rate dictated by the spectral gap of $\mathcal L_\alpha$.

    Finally, unlike the classical case $\alpha=1$ where sharp $L^2$ decay exponents are known (e.g.\ \cite{TM 01,TC 02}), the fractional theory has so far provided mainly heat-kernel-type upper bounds (e.g.\ \cite{ZL 24}). Our invariant-manifold analysis gives decay rates sensitive to the spectral data (and thus to moment conditions on the initial vorticity), providing a first step toward a sharp asymptotic theory in the fractional regime.

    Throughout this subsection we work in $L^{2}(m\alpha)$ with $m=4-\varepsilon$ for $\varepsilon>0$ sufficiently small. If $1/2<\alpha\le 3/4$, the discrete spectrum relevant to our analysis is the same as in the previous subsection (only the $\lambda_1$-modes with eigenfunctions $F_{1,\alpha},F_{2,\alpha}$). If $3/4<\alpha<1$, the condition $2\alpha-1<k<4\alpha-1$ isolates one additional eigenvalue $\lambda_2$ of multiplicity three, with eigenfunctions $H_{1,\alpha},H_{2,\alpha},H_{3,\alpha}$; we exploit this extra mode to refine the asymptotic description of small solutions.
    \begin{theorem}\label{opt}
        Assume that $3/4<\alpha<1$, $\textbf{u}_{0}\in L^{1}(\mathbb{R}^{2})^{2}$, $\nabla\cdot\textbf{u}_{0}=0$, and $(1+|x|)\textbf{u}_{0}\in L^{1}(\mathbb{R}^{2})^{2}$. Let $\textbf{u}(t)$ be a corresponding global small solution of the Navier-Stokes equation with initial data $\textbf{u}_{0}$. For all $k,l=1\  \text{or}\ 2$, define 
        \begin{align}\label{ckl}
            b_{kl}=\int_{\mathbb{R}^{2}}x_{l}(\textbf{u}_{0})_{k}(x)\mathrm{d}x,\qquad c_{kl}=\int_{0}^{\infty}\int_{\mathbb{R}^{2}}u_{k}(x,t)u_{l}(x,t)\mathrm{d}x\mathrm{d}t.
        \end{align}
        Then
        \begin{align}
            \lim\limits_{t\to\infty} t^{\frac{1}{\alpha}}|\textbf{u}(t)|_{2}=0
        \end{align}
        if and only if there exists $c\geq 0$ such that
        \begin{align}\label{cond}
            b_{kl}=0\quad\text{and}\quad c_{kl}=c\delta_{kl},\quad k,l=1\ \text{or}\ 2.
        \end{align}
    \end{theorem}
    \begin{proof}
        If $\textbf{u}_{0}\in L^{1}(\mathbb{R}^{2})^{2}$, by \cite[Corollary B.4]{TC 02}, it is equivalent to 
        \begin{align}\label{4.43}
            \int_{\mathbb{R}^{2}} w(\xi)\mathrm{d}\xi=0,\quad \int_{\mathbb{R}^{2}}\xi_{1}w(\xi)\mathrm{d}\xi=0,\quad \int_{\mathbb{R}^{2}}\xi_{2}w(\xi)\mathrm{d}\xi=0.
        \end{align}
        Then, we study the solution of the vorticity equation \eqref{nve} in the invariant subspace of $L^2((4-\varepsilon)\alpha)$ defined by \eqref{4.43}. If $w(\xi,\tau)$ is such a solution and if $\textbf{v}(\xi,\tau)$ is the velocity field obtained from $w(\xi,\tau)$ via the Biot-Savart law \eqref{nbs}, then $w$ and $\textbf{v}$ can be decomposed as follows:
        \begin{align}
            w(\xi,\tau)&=\gamma_{1}(\tau)H_{1,\alpha}(\xi)+\gamma_{2}(\tau)H_{2,\alpha}(\xi)+\gamma_{3}(\tau)H_{3,\alpha}(\xi)+R(\xi,\tau),\notag\\
            \textbf{v}(\xi,\tau)&=\gamma_{1}(\tau)\textbf{v}^{H_{1,\alpha}}(\xi)+\gamma_{2}(\tau)\textbf{v}^{H_{2,\alpha}}(\xi)+\gamma_{3}(\tau)\textbf{v}^{H_{3,\alpha}}(\xi)+R(\xi,\tau).
        \end{align}
        Next we compute the coefficients $\gamma_j$ ($j=1,2,3$) via second moments.  For any homogeneous quadratic polynomial $p(\xi)$, set
        \begin{align*}
            M_p(\tau):=\int_{\mathbb R^2} p(\xi)\,w(\xi,\tau)\,\mathrm{d}\xi .
        \end{align*}
        Differentiating and using \eqref{nve} we obtain
        \begin{align*}
            \dot M_p &=\int_{\mathbb R^2} p\,\partial_\tau w\,\mathrm{d}\xi =\underbrace{\int_{\mathbb R^2} p\,\mathcal L_\alpha w\,\mathrm{d}\xi}_{J_1} +\underbrace{\int_{\mathbb R^2} -p\,(\mathbf v\!\cdot\nabla)w\,\mathrm{d}\xi}_{J_2}.
        \end{align*}
        For $J_1$, by duality and integration by parts,
        \begin{align*}
            J_1&=-\int_{\mathbb R^2} p\,(-\Delta)^\alpha w\,\mathrm{d}\xi +\frac1{2\alpha}\int_{\mathbb R^2} p\,(\xi\!\cdot\nabla w)\,\mathrm{d}\xi +\int_{\mathbb R^2} p\,w\,\mathrm{d}\xi \\ &=-\int_{\mathbb R^2} w\,(-\Delta)^\alpha p\,\mathrm{d}\xi +\frac1{2\alpha}\Bigl(-\int_{\mathbb R^2}\nabla\!\cdot(\xi p)\,w\,\mathrm{d}\xi\Bigr) +M_p \\ 
            &=-\int_{\mathbb R^2} w\,(-\Delta)^\alpha p\,\mathrm{d}\xi +\Bigl(1-\frac2\alpha\Bigr)M_p ,
        \end{align*}
        since $\nabla\!\cdot(\xi p)=(2+\deg p)\,p=4p$ for quadratic $p$.  For the three basis polynomials
        \begin{align*}
            p_1(\xi)=|\xi|^2,\qquad p_2(\xi)=\xi_1^2-\xi_2^2,\qquad p_3(\xi)=\xi_1\xi_2,
        \end{align*}
        one has
        \begin{align*}
            (-\Delta)^\alpha p_2\equiv 0,\qquad (-\Delta)^\alpha p_3\equiv 0,\qquad (-\Delta)^\alpha p_1\equiv C_\alpha,
        \end{align*}
        hence (using $\int_{\mathbb R^2}w\,\mathrm{d}\xi=0$ for \eqref{nve})
        \begin{align*}
            \dot M_{p_\ell}=\Bigl(1-\frac2\alpha\Bigr)M_{p_\ell}+J_2(p_\ell),\qquad \ell=1,2,3.
        \end{align*}
        Moreover,
        \begin{align*}
            J_2(p) &=-\int_{\mathbb R^2} p\,(\mathbf v\!\cdot\nabla)w\,\mathrm{d}\xi =\int_{\mathbb R^2}\nabla p\cdot(\mathbf v\,w)\,\mathrm{d}\xi,
        \end{align*}
        so that
        \begin{align*}
            J_2(p_1)=2\int_{\mathbb R^2}(\xi\!\cdot\mathbf v)\,w\,\mathrm{d}\xi,\qquad
            J_2(p_2)=2\int_{\mathbb R^2}(\xi_1 v_1-\xi_2 v_2)\,w\,\mathrm{d}\xi,\qquad
            J_2(p_3)=\int_{\mathbb R^2}(\xi_2 v_1+\xi_1 v_2)\,w\,\mathrm{d}\xi.
        \end{align*}
        Define
        \begin{align*}
            \gamma_1(\tau)=\frac14 M_{p_1},\qquad
            \gamma_2(\tau)=\frac14 M_{p_2},\qquad
            \gamma_3(\tau)=\frac12 M_{p_3}.
        \end{align*}
        Then
        \begin{align*}
            \dot\gamma_1=\Bigl(1-\frac2\alpha\Bigr)\gamma_1+\frac14 J_2(p_1),\qquad
            \dot\gamma_2=\Bigl(1-\frac2\alpha\Bigr)\gamma_2+\frac14 J_2(p_2),\qquad
            \dot\gamma_3=\Bigl(1-\frac2\alpha\Bigr)\gamma_3+\frac12 J_2(p_3).
        \end{align*}

        According to \cite[Lemma 4.9]{TC 02}, we get 
        \begin{gather*}
            J_{2}(|\xi|^{2})=0,\qquad J_{2}(\xi_{1}^{2}-\xi_{2}^{2})=-4\int_{\mathbb{R}^{2}}v_{1}v_{2}\mathrm{d}\xi,\\
            J_{2}(\xi_{1}\xi_{2})=\int_{\mathbb{R}^{2}}(v_{1}^{2}-v_{2}^{2})\mathrm{d}\xi.
        \end{gather*}
        Then
        \begin{align*}
            \dot{\gamma}_{1}=(1-\tfrac{2}{\alpha})\gamma_{1},\
            \dot{\gamma}_{2}=(1-\tfrac{2}{\alpha})\gamma_{2}-\int_{\mathbb{R}^{2}}v_{1}v_{2}\mathrm{d}\xi,\
            \dot{\gamma}_{3}=(1-\tfrac{2}{\alpha})\gamma_{3}+\frac{1}{2}\int_{\mathbb{R}^{2}}v_{1}^{2}-v_{2}^{2}\mathrm{d}\xi.
        \end{align*}
        Now, let $E_{\text{d}}=\text{span}\{H_{1,\alpha},H_{2,\alpha},H_{3,\alpha}\}$, and let $E_{\text{c}}\subset L^{2}(m\alpha)$ be the spectral subspace of $\mathcal{L}_{\alpha}$ corresponding to the other spectrum $\sigma_{c}=\{\lambda\in \mathbb{C}\ |\ \Re \lambda\leq-\left(\frac{m}{2}+\frac{1}{2\alpha}-1\right)\}$. Given $\frac{2}{\alpha}-1 <\mu<1-\frac{\varepsilon}{2}+\frac{1}{2\alpha}$, Theorem \ref{loc-strong} (with $k=2$, $3/4<\alpha<1$ and $m=4-\varepsilon$) shows that, for sufficiently small $r_{1}>0$, the set $W_{\text{c}}^{\text{low}}$ defined by \eqref{local-fast-estimate} is a infinite-dimensional manifold which is tangent to $E_{\text{c}}$ at the origin. Hence, Fix $\frac{2}{\alpha}-1<\mu<\frac{2\alpha+1}{2\alpha}$, by conducting an analysis similar to that in the previous subsection and by changing variables to self-similar coordinates \eqref{v-u}, we can obtain
        \begin{align}\label{4alpha}
            \|w(\cdot,\tau)-(c_{1}H_{1,\alpha}+c_{2}H_{2,\alpha}+c_{3}H_{3,\alpha})e^{\lambda_{2}\tau}\|_{m\alpha}\leq Ce^{-\mu\tau}.
        \end{align}
        here 
        \begin{align*}
            c_{1}&=\gamma_{1}(0),\\
            c_{2}&=\gamma_{2}(0)-\int_{0}^{\infty}e^{(\tfrac{2}{\alpha}-1)\tau}\int_{\mathbb{R}^{2}}v_{1}(\xi,\tau)v_{2}(\xi,\tau)\mathrm{d}\xi\mathrm{d}\tau\equiv \gamma_{2}(0)-c_{12},\\
            c_{3}&=\gamma_{3}(0)+\frac{1}{2}\int_{0}^{\infty}e^{(\tfrac{2}{\alpha}-1)\tau}\int_{\mathbb{R}^{2}}(v_{1}^{2}(\xi,\tau)-v_{2}^{2}(\xi,\tau))\mathrm{d}\xi\mathrm{d}\tau\equiv \gamma_{3}(0)+\frac{1}{2}(c_{11}-c_{22}).
        \end{align*}
        where $c_{kl}$ is defined in \eqref{ckl}. Since the assumption $(1+|x|)\textbf{u}_{0}\in L^{1}(\mathbb{R}^{2})^{2}$, by \cite[Corollary B.4]{TC 02}, we know that $\gamma_{2}(0)=\gamma_{3}(0)=0$. Combined with \eqref{ckl} and \cite[Corollary B.5]{TC 02}, we know that \eqref{cond} is equivalent to 
        \begin{align*}
            c_{1}=c_{2}=c_{3}=0.
        \end{align*}
        Hence, $w_{0}$ lies in the strong stable manifold $W_{\text{c}}^{\text{loc}}$. By Theorem \ref{loc-strong}, we have
        \begin{align*}
            \lim\limits_{\tau\to \infty} e^{(\tfrac{2}{\alpha}-1)\tau}\|w(\cdot,\tau)\|_{m\alpha}=0.
        \end{align*}
        Moreover, we have proved that 
        \begin{align*}
            \lim\limits_{t\to \infty}t^{1/\alpha}|\textbf{u}(\cdot,t)|_{2}=0.
        \end{align*}
        \medskip
        On the other hand,  Assume that
        \begin{equation}\label{decay-ass}
            \lim_{t\to\infty} t^{\frac{1}{\alpha}}|\mathbf u(\cdot,t)|_2 = 0.
        \end{equation}
        We show that \eqref{cond} holds. First, we know that
        \begin{equation}\label{u-v-rel}
            |\mathbf u(\cdot,t)|_2 = (1+t)^{-1+\frac{1}{\alpha}}\,
            |\mathbf v(\cdot,\tau)|_2.
        \end{equation}
        By the spectral decomposition in $L^2(m\alpha)$ and Theorem \ref{loc-strong}, we already know that there exist constants $c_1,c_2,c_3\in\mathbb{R}$ and some $\frac{2}{\alpha}-1<\mu\leq\frac{2\alpha+1}{2\alpha}$ such that
        \begin{equation}\label{w-asym}
            \bigl\|w(\cdot,\tau)-(c_1H_{1,\alpha}+c_2H_{2,\alpha}+c_3H_{3,\alpha})e^{\lambda_2\tau}\bigr\|_{m\alpha}  \le C e^{-\mu\tau},\qquad \tau\ge0,
        \end{equation}
        where $\lambda_2=1-\tfrac{2}{\alpha}$. Since the Biot--Savart operator is bounded from $L^2(m\alpha)$ to $L^2(\mathbb{R}^2)^2$, \eqref{w-asym} implies that
        \begin{equation}\label{v-asym}
            \bigl|\mathbf v(\cdot,\tau)-e^{\lambda_2\tau}\mathbf V\bigr|_2\le C e^{-\mu\tau},\qquad \tau\ge0,
        \end{equation}
        where
        \begin{align*}
            \mathbf V:= c_1\,\mathbf v^{H_{1,\alpha}}+ c_2\,\mathbf v^{H_{2,\alpha}} + c_3\,\mathbf v^{H_{3,\alpha}}\in L^2(\mathbb{R}^2)^2.
        \end{align*}
        Combining \eqref{u-v-rel} and \eqref{v-asym}, and using $1+t=e^{\tau}$, we obtain
        \begin{align*}
            |\mathbf u(\cdot,t)|_2 = (1+t)^{-1+\frac{1}{\alpha}}\Bigl( (1+t)^{\lambda_2}|\mathbf V|_2 + O((1+t)^{-\mu})    \Bigr).
        \end{align*}
        Multiplying by $t^{1/\alpha}\sim s^{1/\alpha}$ as $t\to\infty$, we get
        \begin{align*}
            t^{\frac{1}{\alpha}}|\mathbf u(\cdot,t)|_2= |\mathbf V|_2 + o(1),
        \end{align*}
        because $\lambda_2-1+\tfrac{2}{\alpha}=0$ and $\mu>\tfrac{2}{\alpha}-1$. Hence
        \begin{equation}\label{limit-V}
            \lim_{t\to\infty} t^{\frac{1}{\alpha}}|\mathbf u(\cdot,t)|_2 = |\mathbf V|_2.
        \end{equation}
        By the decay assumption \eqref{decay-ass}, the left-hand side of
        \eqref{limit-V} is zero, so we conclude that $|\mathbf V|_2=0$,
        which implies
        \begin{align*}
            c_1=c_2=c_3=0.
        \end{align*}
        Finally, as explained above (using \cite[Corollary B.4--B.5]{TC 02}), our moment assumption $(1+|x|)\mathbf u_0\in L^1(\mathbb{R}^2)^2$ implies that
        \begin{align*}
             c_1=c_2=c_3=0\quad\Longleftrightarrow\quad b_{kl}=0\ \text{and}\ c_{kl}=c\delta_{kl},\ k,l=1,2,
        \end{align*}
        that is, condition \eqref{cond} holds. This proves the converse implication and completes the proof of the theorem.
    \end{proof}

    \begin{remark}
        The above argument follows the same scheme as in \cite[Theorem 4.17]{TC 02}.  We emphasize that all identities coming from Appendix B of \cite{TC 02} (in particular the formulas for $J_{2}(p_i)$, the relations between $\gamma_j$ and the velocity moments, and the characterization of $b_{kl}$ and $c_{kl}$) are purely elliptic: their derivation uses only the Biot--Savart law, the divergence-free condition and suitable decay assumptions, and does not involve the time evolution or the specific form of the dissipative operator. Since these ingredients are unchanged in our fractional setting, the computations in Appendix B of \cite{TC 02} apply verbatim for $3/4<\alpha<1$, and we use them here without repeating the details.
    \end{remark}

    \begin{appendix}
    \section{Appendix A. The vorticity equation after scaling transformation}

    In order to find suitable scaling transformation, we define 
    \begin{align*}
        \xi=\frac{x}{(1+t)^{\beta}},\quad \tau=\text{log}(1+t).
    \end{align*}
    Moreover, we assume 
    \begin{align}
        \omega(x,t)=\frac{1}{(1+t)^{\gamma}}w(\xi,\tau),\label{A.1}\\
        \textbf{u}(x,t)=\frac{1}{(1+t)^{\delta}}\textbf{v}(\xi,\tau)\label{A.2},
    \end{align}
    where $\beta$, $\gamma$ and $\delta$ are undetermined coefficients. Our goal is to find these coefficients and to calculate the vorticity equation after scaling transformation.\\
    First, using the Biot--Savart law, we have
    \begin{align*}
        \mathbf{u}(x,t)&=\frac{1}{2\pi}\int_{\mathbb{R}^{2}}\frac{(x-y)^{\perp}}{|x-y|^{2}}\omega(y,t)\mathrm{d}y.
    \end{align*}
    By performing the change of variables $y=(1+t)^{\beta}\eta$ and noting that $x=(1+t)^{\beta}\xi$, we substitute these into the integral
    \begin{align*}
        \mathbf{u}(x,t)&=\frac{1}{2\pi}\int_{\mathbb{R}^{2}}\frac{((1+t)^{\beta}(\xi-\eta))^{\perp}}{|(1+t)^{\beta}(\xi-\eta)|^{2}}\frac{1}{(1+t)^{\gamma}}w(\eta,\tau)\cdot (1+t)^{2\beta}\,\mathrm{d}\eta\\
        &=\frac{(1+t)^{\beta}}{(1+t)^{2\beta}}\frac{(1+t)^{2\beta}}{(1+t)^{\gamma}}\frac{1}{2\pi}\int_{\mathbb{R}^{2}}\frac{(\xi-\eta)^{\perp}}{|\xi-\eta|^{2}}w(\eta,\tau)\,\mathrm{d}\eta\\
        &=(1+t)^{\beta-\gamma}\mathbf{v}(\xi,\tau).
    \end{align*}
    We obtain the first relation
\begin{align}\label{A.3}
        \gamma-\beta=\delta.
    \end{align}
    Next, we substitute \eqref{A.1} and \eqref{A.2} into \eqref{ve}, for $\partial_{t}\omega$ term, we have
    \begin{align}
        \partial_{t}\omega&=\partial_{t}\bigg(\frac{1}{(1+t)^{\gamma}}w(\xi,\tau)\bigg)\notag\\
        &=-\frac{\gamma}{(1+t)^{1+\gamma}}w+\frac{1}{(1+t)^{\gamma}}\bigg(\partial_{\tau}w \frac{d \tau}{d t}+\nabla_{\xi}w\cdot \frac{\partial \xi}{\partial t}\bigg)\notag\\
        &=-\frac{\gamma}{(1+t)^{1+\gamma}}+\frac{1}{(1+t)^{1+\gamma}}\partial_{\tau}w-\frac{\beta}{(1+t)^{1+\gamma}}\xi\cdot \nabla_{\xi}w.\label{A.4}
    \end{align}
    For nonlinear term $(\textbf{u}\cdot \nabla_{x})\omega=(\frac{1}{(1+t)^{\delta}}\textbf{v}\cdot\nabla_{x})(\frac{1}{(1+t)^{\gamma}}w)$, we have $\nabla_{x}=\frac{1}{(1+t)^{\beta}}\nabla_{\xi}$ and 
    \begin{align}\label{A.5}
        (\boldsymbol{u}\cdot \nabla_{x})\omega=\frac{1}{(1+t)^{\gamma+\delta+\beta}}(\textbf{v}\cdot\nabla_{\xi})w.
    \end{align}
    For fractional Laplace operator $(-\Delta_{x})^{\alpha}$, we consider it Fourier definition, $(-\nabla_{x})^{\alpha}\omega(x,t)=\mathcal{F}^{-1}[|k|^{\alpha}\mathcal{F}(\omega)(k)](x)$. Substitute \eqref{A.1} into the $\mathcal{F}[\omega](k)$,
    \begin{align*}
        \mathcal{F}[\omega](k)=\frac{1}{(1+t)^{\gamma}}\int_{\mathbb{R}^{2}}e^{-ikx}w\bigg(\frac{x}{(1+t)^{\beta}}\bigg)\mathrm{d}x.
    \end{align*}
    Since $\xi=\frac{x}{(1+t)^{\beta}}$, $x=(1+t)^{\beta}\xi$ and $\mathrm{d}x=(1+t)^{2\beta}\mathrm{d}\xi$, we have
    \begin{align*}
        \mathcal{F}[\omega](k)&=\frac{1}{(1+t)^{\gamma}}(1+t)^{2\beta}\int_{\mathbb{R}^{2}}e^{-ik(1+t)^{\beta}\xi}w(\xi)\mathrm{d}\xi\\
        &=(1+t)^{2\beta-\gamma}\mathcal{F}[w]((1+t)^{\beta}k).
    \end{align*}
    $(-\Delta_{x})^{\alpha}\omega$ can rewrite as
    \begin{align}
        (-\Delta_{x})^{\alpha}(\omega)&=\mathcal{F}^{-1}(|k|^{2\alpha}(1+t)^{2\beta-\gamma}\mathcal{F}[w]((1+t)^{\beta}k))(x)\notag\\
        &=(1+t)^{2\beta-\gamma}\frac{1}{(2\pi)^{2}}\int_{\mathbb{R}^{2}}e^{ikx}|k|^{2\alpha}\mathcal{F}[w]((1+t)^{\beta}k)\mathrm{d}k\notag\\
        &=(1+t)^{2\beta-\gamma}\mathcal{F}^{-1}\bigg(\bigg(\frac{|q|}{(1+t)^{\beta}}\bigg)^{2\alpha}\mathcal{F}[w](q)\bigg)\frac{1}{(1+t)^{2\beta}}\notag\\
        &=\frac{1}{(1+t)^{\gamma+2\alpha\beta}}(-\Delta_{\xi})^{\alpha}w(\xi),\label{A.6}
    \end{align}
    where the third equality uses the following variable substitution $q=(1+t)^{\beta}k$, $k=q/(1+t)^{\beta}$ and $\mathrm{d}k=\frac{\mathrm{d}q}{(1+t)^{2\beta}}$. To eliminate the coefficient of each term, we combine it with\eqref{A.4}, \eqref{A.5} and \eqref{A.6}. The exponents must satisfy the following balance condition:
    \begin{align}
        \gamma+1=\gamma+\delta+\beta=\gamma+2\alpha\beta. \label{A.7}
    \end{align}
    Solving this system yields $\beta=\frac{1}{2\alpha}$, $\delta=1-\frac{1}{2\alpha}$. According to \eqref{A.3}, we know that $\gamma=1$. Moreover, we have $\gamma=1$, $\beta=\frac{1}{2\alpha}$, $\delta=1-\frac{1}{2\alpha}$. So $\xi=\frac{x}{(1+t)^{1/2\alpha}}$, $\tau=\text{log}(1+t)$ and 
    \begin{align*}
        \omega(x,t)=\frac{1}{(1+t)}w(\xi,\tau),\ 
        \textbf{u}(x,t)=\frac{1}{(1+t)^{1-\frac{1}{2\alpha}}}\textbf{v}(\xi,\tau).
    \end{align*}
    Finally, we obtain the vorticity equation 
    \begin{align*}
        \partial_{\tau}w=-(-\Delta)^{\alpha}w+\frac{1}{2\alpha}(\xi\cdot\nabla_{\xi})w+w-(\textbf{v}\cdot\nabla_{\xi})w.
    \end{align*}

\section{Appendix B. Spectrum of the linear operator}
    In this section, we investigate the spectrum of the linear operator $\mathcal{L}_{\alpha}$ defined in \eqref{linear}. This operator is the linear part obtained from the vorticity equation \eqref{ve} through a scaling transformation. As in the integer-order case, we study the operator on $\mathbb{R}^{N}$ rather than just on $\mathbb{R}^{2}$, and so we assume that $N\in \mathbb{N}$, $N\geq 1$. We discuss the linear operator $\mathcal{L}_{\alpha}$ given by   
    \begin{align}\label{Nlin}
        \mathcal{L}_{\alpha}=-(-\Delta_{\xi})^{\alpha}+\frac{1}{2\alpha}(\boldsymbol{\xi}\cdot\nabla_{\xi})+\frac{N}{2},\quad \boldsymbol{\xi}\in\mathbb{R}^{N}.
    \end{align} 
    We investigate in the weighted space $L^{2}(m\alpha)$ defined by 
    \begin{align}
        &L^{2}(m\alpha)=\{f\in L^{2}(\mathbb{R}^{N})| \|f\|_{m,\alpha}<\infty\},\\
        &\|f\|_{m\alpha}=\bigg(\int_{\mathbb{R}^{N}}(1+|
        \xi|^{2})^{m\alpha}|f(\xi)|^{2}\mathrm{d}\xi\bigg)^{\frac{1}{2}}.\notag
    \end{align}
    Our convention for Fourier transform is
    \begin{align}
        &\hat{f}(p)=\mathcal{F}[f](p)=\int_{\mathbb{R}^{N}}f(\xi)\text{exp}(-i\boldsymbol{p}\cdot\boldsymbol{\xi})\mathrm{d}\xi,\\
        &f(\xi)=\mathcal{F}^{-1}[\hat{f}](\xi)=\frac{1}{(2\pi)^{N}}\int_{\mathbb{R}^{N}}\hat{f}(p)\text{exp}(i\boldsymbol{p}\cdot\boldsymbol{\xi})\mathrm{d}p.
    \end{align}
    Since the Fourier transform turns the operator $\mathcal{L}_{\alpha}$ into a first-order differential operator, we study its spectral distribution in the Fourier space. Consider that the Fourier form of $\mathcal{L}_{\alpha}$ is 
    \begin{align}\label{Fourier}
        \widehat{\mathcal{L}_{\alpha}}=-|p|^{2\alpha}-\frac{1}{2\alpha}(\boldsymbol{p}\cdot\nabla_{p})+\frac{N}{2}(1-\frac{1}{\alpha}).
    \end{align}

    The aim of this section is to prove the following result, which underlies our approach for computing the long-time asymptotics of the fractional vorticity equation:
    \begin{theorem}\label{sp}
        Fix $m\geq 0$, and let $\mathcal{L}_{\alpha}$ be the linear operator \eqref{Nlin} in $L^{2}(m\alpha)$, defined on its maximal domain. Then the spectrum of $\mathcal{L}_{\alpha}$ is 
        \begin{align*}
            \sigma(\mathcal{L}_{\alpha})=\bigg\{\lambda\in \mathbb{C}\ \bigg|\  \Re(\lambda)\leq \frac{2\alpha N-N-2m\alpha}{4\alpha}\bigg\}\cup \bigg\{-\frac{k}{2\alpha}+\frac{N}{2}\bigg(1-\frac{1}{\alpha}\bigg)\ \bigg|\ k\in \mathbb{N}\bigg\}.
        \end{align*}
        Moreover, if $(\alpha-\frac{1}{2})N<m\alpha<k+2\alpha+\frac{N}{2}$ and if $k\in \mathbb{N}$ satisfies $k+\frac{N}{2}<m\alpha$, then $\lambda_{k}=-\frac{k}{2\alpha}+\frac{N}{2}(1-\frac{1}{\alpha})$ is an isolated eigenvalue of $\mathcal{L}_{\alpha}$, with multiplicity $\dbinom{N + k - 1}{k}$.
    \end{theorem}
    Before proving Theorem \ref{sp}, we observe that the operator $\mathcal{L}_{\alpha}$ is rotationally invariant. Indeed, both the fractional Laplacian $-(-\Delta_{\xi})^{\alpha}$, defined by the radial symbol $-|p|^{2\alpha}$, and the drift term $\frac{1}{2\alpha}(\boldsymbol{\xi}\cdot\nabla_{\xi})$ commute with the action of the rotation group $SO(N)$. In the Fourier frequency domain, using the relation $\boldsymbol{p}\cdot\nabla_{p}=r\partial_{r}$ with $r=|p|$, the symbol $\widehat{\mathcal{L}_{\alpha}}$ becomes a radial operator:
    \begin{align*}
        \widehat{\mathcal{L}_{\alpha}} = -r^{2\alpha} - \frac{1}{2\alpha}r\partial_r + \frac{N}{2}\left(1-\frac{1}{\alpha}\right).
    \end{align*}
    \begin{remark}
        The rotational invariance implies that the eigenspaces of $\mathcal{L}_{\alpha}$ can be decomposed using spherical harmonics. Consequently, the multiplicity of each eigenvalue is determined by the dimension of the corresponding space of spherical harmonics, exactly as in the classical Laplacian case.
    \end{remark}   
    \begin{proof}[Proof of Theorem \ref{sp}]
        Fix $k\in \mathbb{N}_0$ and let $\beta \in \mathbb{N}_0^N$ with $|\beta|=k$. Consider the function $\hat{\phi}_{\beta}(p)=p^{\beta}e^{-|p|^{2\alpha}}$. Substituting $\hat{\phi}_{\beta}$ into the spectral equation involving \eqref{Fourier}, a direct computation yields:
        \begin{align*}
            \widehat{\mathcal{L}_{\alpha}}\hat{\phi}_{\beta} &= -|p|^{2\alpha}\hat{\phi}_{\beta}-\frac{1}{2\alpha}(\boldsymbol{p}\cdot\nabla_{p})\hat{\phi}_{\beta}+\frac{N}{2}\left(1-\frac{1}{\alpha}\right)\hat{\phi}_{\beta} \\
            &= \left(-\frac{k}{2\alpha}+\frac{N}{2}\left(1-\frac{1}{\alpha}\right)\right)\hat{\phi}_{\beta}.
        \end{align*}
        Thus, $\lambda_{k}=-\frac{k}{2\alpha}+\frac{N}{2}(1-\frac{1}{\alpha})$ are the eigenvalues. To verify that the corresponding eigenfunctions $\phi_{\beta} = \mathcal{F}^{-1}[\hat{\phi}_{\beta}]$ belong to $L^{2}(m\alpha)$, we note that $\phi_{\beta}(\xi) = i^{|\beta|}\partial_{\xi}^{\beta}G_\alpha(\xi)$, where $G_\alpha$ is the fractional heat kernel. Using the asymptotic estimate $G_\alpha(\xi) \sim |\xi|^{-(N+2\alpha)}$ as $|\xi|\to\infty$ (see \cite{E.s 70}), we have $\phi_{\beta}(\xi) \sim |\xi|^{-(N+2\alpha+k)}$. The condition $\phi_{\beta}\in L^{2}(m\alpha)$ requires
        \begin{align*}
            \int_{|\xi|\ge 1} |\xi|^{2m\alpha} |\xi|^{-2(N+2\alpha+k)} |\xi|^{N-1} \mathrm{d}\xi < \infty,
        \end{align*}
        which holds if and only if $m\alpha < k+2\alpha+\frac{N}{2}$.

        To determine the continuous spectrum, we seek radial eigenfunctions $\hat{f}(p)=\hat{f}(r)$ with $r=|p|$. The spectral equation becomes a radial ODE:
        \begin{align*}
            -|p|^{2\alpha}\hat{f}-\frac{1}{2\alpha}r\partial_r \hat{f} +\frac{N}{2}\left(1-\frac{1}{\alpha}\right)\hat{f}=\lambda \hat{f}.
        \end{align*}
        Solving this first-order ODE yields the solution:
        \begin{align*}
            \hat{f}(r) = r^{\gamma}e^{-r^{2\alpha}}, \quad \text{where } \gamma = -2\alpha\left(\lambda - \frac{N}{2} + \frac{N}{2\alpha}\right).
        \end{align*}
        Since $\hat{f}$ is radial, its inverse Fourier transform $f(\xi)$ is given by the Hankel transform (see \cite{GN 44}):
        \begin{align*}
            f(\xi) = (2\pi)^{-\frac{N}{2}} \rho^{-\frac{N-2}{2}} \int_{0}^{\infty} \hat{f}(r) r^{\frac{N}{2}} J_{\frac{N-2}{2}}(r\rho) \,\mathrm{d}r, \quad \rho=|\xi|.
        \end{align*}
        Substituting the asymptotic form $\hat{f}(r)\sim r^{\gamma}$ for small $r$ and using the asymptotic properties of the Hankel transform (see \cite{GN 44}), the decay of $f(\xi)$ as $|\xi|\to\infty$ is determined by the singularity of $\hat{f}$ at the origin. Specifically, we have
        \begin{align*}
            f(\xi) \sim C |\xi|^{-(\gamma+N)}, \quad \text{as } |\xi|\to\infty.
        \end{align*}
        Substituting the expression for $\gamma$, the decay rate is $|\xi|^{2\alpha(\lambda - \frac{N}{2} + \frac{N}{2\alpha}) - N}$. Writing $\lambda = \Re \lambda + i \Im \lambda$, the condition $f \in L^2(m\alpha)$ becomes
        \begin{align*}
            \int_{|\xi|\ge 1} |\xi|^{2m\alpha} \left| |\xi|^{2\alpha(\Re \lambda - \frac{N}{2} + \frac{N}{2\alpha}) - N} \right|^2 |\xi|^{N-1} \,\mathrm{d}\xi < \infty.
        \end{align*}
        This integral converges if and only if
        \begin{align*}
            2m\alpha + 4\alpha\left(\Re \lambda - \frac{N}{2} + \frac{N}{2\alpha}\right) - 2N + N - 1 < -1.
        \end{align*}
        Simplifying this inequality yields $\Re \lambda < \frac{2\alpha N - N - 2m\alpha}{4\alpha}$. Since the spectrum is closed, this confirms the region for the continuous spectrum stated in the theorem.
    \end{proof}
    \begin{remark}
        Unlike the classical case, where the operator can be conjugated by a Gaussian weight to a self-adjoint Ornstein--Uhlenbeck operator, the fractional counterpart $\mathcal L_\alpha$ admits no such straightforward self-adjoint reduction. This necessitates the use of Fourier variables for a direct spectral analysis.
    
        Due to rotational invariance, the spectrum is determined by the radial analysis on spherical-harmonic sectors. In particular, the eigenvalue equation in Fourier space dictates that all discrete eigenfunctions are of the form $\widehat{\phi}(p)=p^{s}e^{-|p|^{2\alpha}}$ with $s\in\mathbb N^{N}$. These polynomial--exponential modes exhaust the discrete spectrum, and no further eigenvalues exist.
    \end{remark}
    



    Next, we derive the explicit representation of the semigroup $e^{\mathcal{L}_{\alpha}\tau}$. Let $w(\xi,\tau)$ be the solution to $\partial_{\tau}w=\mathcal{L}_{\alpha}w$ with initial data $w_{0}$. To solve this equation, we reduce it to the standard fractional heat equation via a time-dependent scaling transformation. We introduce the new variables
    \begin{align*}
        y = \xi e^{\tau/2\alpha}, \quad s(\tau) = e^{\tau}-1,
    \end{align*}
    and define the rescaled function $v(y,s)$ by the relation
    \begin{align}\label{trans-wv}
        w(\xi,\tau) = e^{\frac{N\tau}{2}} v(y, s(\tau)).
    \end{align}
    Substituting \eqref{trans-wv} into the equation $\partial_{\tau}w=\mathcal{L}_{\alpha}w$ and using the chain rule, a direct computation shows that the drift term and the linear growth term cancel out, and $v(y,s)$ satisfies
    \begin{align*}
        \partial_{s}v = -(-\Delta_{y})^{\alpha}v, \quad v(y,0) = w_{0}(y).
    \end{align*}
    The solution in the Fourier domain is given by $\widehat{v}(p,s) = e^{-s|p|^{2\alpha}}\widehat{w}_{0}(p)$. In physical space, this corresponds to the convolution $v(\cdot,s) = P_s * w_0$, where $P_s(x)$ is the fractional heat kernel defined by $\widehat{P}_s(p) = e^{-s|p|^{2\alpha}}$. Utilizing the self-similarity $P_s(x) = s^{-N/2\alpha} \phi_0(x s^{-1/2\alpha})$ with $\phi_0 = P_1$, we have
    \begin{align*}
        v(y,s) = s^{-N/2\alpha} \int_{\mathbb{R}^N} \phi_0\left(\frac{y-z}{s^{1/2\alpha}}\right) w_0(z) \mathrm{d}z.
    \end{align*}
    Reverting to the original variables $\xi$ and $\tau$ using \eqref{trans-wv}, and setting $a(\tau) = 1-e^{-\tau}$, we obtain the Fourier representation of the semigroup
    \begin{align}\label{group-F}
        \widehat{e^{\mathcal{L}_{\alpha}\tau}w_{0}}(k) = e^{\frac{N\tau}{2}(1-\frac{1}{\alpha})} e^{-a(\tau)|k|^{2\alpha}} \widehat{w}_{0}(k e^{-\tau/2\alpha}).
    \end{align}
    Similarly, substituting $y=\xi e^{\tau/2\alpha}$ and $s=e^\tau a(\tau)$ into the integral formula and performing the change of variables $z = \xi' e^{\tau/2\alpha}$, we arrive at the explicit integral representation
    \begin{align}\label{group-P}
        e^{\mathcal{L}_{\alpha}\tau}w_{0}(\xi)=\frac{e^{\frac{N\tau}{2}}}{a(\tau)^{N/2\alpha}} \int_{\mathbb{R}^{N}} \phi_{0}\bigg(\frac{\xi-\xi'}{a(\tau)^{1/2\alpha}}\bigg) w_{0}(e^{\tau/2\alpha}\xi') \mathrm{d}\xi'.
    \end{align}
    Next, we give the following estimate on the semigroup $S^{\alpha}(\tau)=e^{\tau\mathcal{L}_{\alpha}}$
    \begin{proposition}\label{group-est}
        Let $1\leq q\leq p\leq \infty$, $1/q-1/p>\frac{m\alpha-2\alpha-|\beta|}{N}$, $m\geq 0$ and $T>0$. For all $\beta\in \mathbb{N}^{N}$, there exists $C>0$ such that
        \begin{align}\label{de-group-est}
            |b^{m\alpha}\partial^{\beta}S^{\alpha}(\tau)f|_{p}\leq \frac{C}{a(\tau)^{\frac{N}{2\alpha}(\frac{1}{q}-\frac{1}{p})+\frac{|\beta|}{2\alpha}}}|b^{m\alpha}f|_{q},\quad 0<\tau\leq T,
        \end{align}
        where $b(\xi)=(1+|\xi|^{2})^{\frac{1}{2}}$.
    \end{proposition}
    \begin{proof}
        Recalling \eqref{group-P}, we write the derivative of the semigroup as a convolution
        \begin{align*}
            \partial^{\beta}S^{\alpha}(\tau)f = e^{\frac{N\tau}{2}} a(\tau)^{-\frac{|\beta|}{2\alpha}} (\Phi_{a} \ast h),
        \end{align*}
        where $h(\cdot) = f(e^{\tau/2\alpha}\cdot)$ and the rescaled kernel is $\Phi_{a}(x) = a(\tau)^{-\frac{N}{2\alpha}} \phi_{\beta}(a(\tau)^{-1/2\alpha}x)$ with $\phi_{\beta} = \partial^{\beta}\phi_{0}$.
    
        Using Peetre's inequality $b^{m\alpha}(\xi) \leq C b^{m\alpha}(\xi-\xi') b^{m\alpha}(\xi')$, we deduce the pointwise bound $|b^{m\alpha}(\Phi_a \ast h)| \le C (|K_a| \ast |b^{m\alpha}h|)$, where $K_{a} := b^{m\alpha}\Phi_{a}$. Applying Young's inequality with $1+\frac{1}{p} = \frac{1}{q} + \frac{1}{r}$, we obtain
        \begin{align*}
            |b^{m\alpha}\partial^{\beta}S^{\alpha}(\tau)f|_{p} \leq C e^{\frac{N\tau}{2}} a(\tau)^{-\frac{|\beta|}{2\alpha}} |K_{a}|_{r} |b^{m\alpha}h|_{q}.
        \end{align*}
        We explicitly estimate the $L^r$-norm of $K_a$, by the scaling $x = a(\tau)^{1/2\alpha}y$ and noting $a(\tau) \le 1$ for $\tau \le T$,
        \begin{align}\label{est-Ka}
            |K_{a}|_{r} &= a(\tau)^{-\frac{N}{2\alpha}(1-\frac{1}{r})} \left(\int_{\mathbb{R}^{N}} (1+a(\tau)^{1/\alpha}|y|^2)^{\frac{rm\alpha}{2}} |\phi_{\beta}(y)|^r \mathrm{d}y\right)^{\frac{1}{r}} \nonumber \\
            &\leq C a(\tau)^{-\frac{N}{2\alpha}(\frac{1}{q}-\frac{1}{p})}
        \end{align}
        provided the integral converges. Similarly, for the data term, the change of variables $y = e^{\tau/2\alpha}\xi'$ yields
        \begin{align}\label{est-h}
            |b^{m\alpha}h|_{q} \leq e^{-\frac{N\tau}{2\alpha q}} |b^{m\alpha}f|_{q}.
        \end{align}
        Combining estimates \eqref{est-Ka} and \eqref{est-h}, and observing that the exponential factors are bounded on $[0, T]$, we arrive at the desired estimate. 
    
        Finally, the convergence of the integral in \eqref{est-Ka} requires the asymptotic decay $\phi_{\beta}(y) \sim |y|^{-(N+2\alpha+|\beta|)}$. The condition for finiteness is $r(N+2\alpha+|\beta|-m\alpha) > N$, which, upon substituting $\frac{1}{r} = 1 - (\frac{1}{q}-\frac{1}{p})$, is equivalent to $\frac{1}{q}-\frac{1}{p} > \frac{m\alpha-2\alpha-|\beta|}{N}$. This matches the hypothesis.
    \end{proof}
    
    To derive convenient semigroup bounds from the spectral information in Theorem \ref{sp}, we apply a constant spectral shift to the generator. Let $c=\frac{N}{2\alpha}-\frac{N}{2}$ and set $\tilde {\mathcal{L}}_\alpha:=\mathcal{L}_\alpha+cI$. Then $\tilde S_\alpha(\tau)=e^{\tau\tilde {\mathcal{L}}_\alpha}=e^{c\tau}S_\alpha(\tau)$, $\sigma(\tilde {\mathcal{L}}_\alpha)=\sigma(\mathcal{L}_\alpha)+c$, so all semigroup estimates for $S_\alpha$ are equivalent to those for $\tilde S_\alpha$ up to the factor $e^{c\tau}$, and the associated Riesz projections are unchanged (up to translating the contour).
    \begin{proposition}\label{spectral es}
        \begin{enumerate}
            \item Fix $m\geq 0$, and take $n\in \mathbb{Z}$ such that $n+\frac{N}{2}<m\alpha\leq n+1+\frac{N}{2}$. For all $\beta\in \mathbb{N}^{N}$ satisfying $m\alpha<\tfrac{N}{2}+2\alpha+|\beta|$ and all $\varepsilon>0$, there exists $C>0$ such that 
            \begin{align}\label{B.12}
                \|\partial^{\beta}\tilde{S}_{\alpha}(\tau)\tilde{Q}_{n}f\|_{m\alpha}\leq \frac{C}{a(\tau)^{\frac{|\beta|}{2\alpha}}}\mathrm{exp}(\tfrac{\tau}{2\alpha}(\tfrac{N}{2}-m\alpha+\varepsilon))\|f\|_{m\alpha},
            \end{align}
            for all $f\in L^{2}(m\alpha)$ and all $\tau>0$.
            \item Fix $n\in\mathbb{N}\ \cup\ \{-1\}$, and take $m\in \mathbb{R}$ such that $m\alpha>n+1+\frac{N}{2}$. For all $\beta\in \mathbb{N}^{N}$ such that $m\alpha<\tfrac{N}{2}+2\alpha+|\beta|$, there exists $C>0$ such that 
            \begin{align}\label{B.13}
                \|\partial^{\beta}\tilde{S}_{\alpha}(\tau)\tilde{Q}_{n}f\|_{m\alpha}\leq \frac{C}{a(\tau)^{\frac{|\beta|}{2\alpha}}}\mathrm{exp}(-\tfrac{n+1}{2\alpha}\tau)\|f\|_{m\alpha},
            \end{align}
            for all $f\in L^{2}(m\alpha)$ and all $\tau>0$.
        \end{enumerate}
    \end{proposition}
    \begin{proof}[Proof sketch of Proposition B.5]
        Let $\tilde{P}_n$ be the Riesz projection associated with the isolated eigenvalues $\{0,-\frac1{2\alpha},\dots,-\frac n{2\alpha}\}$ of $\tilde{\mathcal{L}}_\alpha$ in $L^2(m\alpha)$ (from the spectral description in Theorem \ref{sp}), and set $\tilde{Q}_{n}:=I-\tilde{P}_{n}$.

        By the Riesz projection formula together with the spectral mapping theorem for $C_0$-semigroups, one obtains
        \begin{align*}
            \|\tilde{S}_\alpha(\tau)\tilde{Q}_n\|_{\mathcal L(L^2(m\alpha))} \le C_\varepsilon \exp\!\Big(\frac{\tau}{2\alpha}\big(\tfrac N2-m\alpha+\varepsilon\big)\Big) \quad\text{if } n+\tfrac N2<m\alpha\le n+1+\tfrac N2,
        \end{align*}
        and
        \begin{align*}
            \|\tilde{S}_\alpha(\tau)\tilde{Q}_n\|_{\mathcal L(L^2(m\alpha))}\le C \exp\!\Big(-\frac{n+1}{2\alpha}\tau\Big) \quad\text{if } m\alpha>n+1+\tfrac N2,
        \end{align*}
        see e.g. \cite[Chapter V]{K.R 00} or the model argument in \cite[Appendix A]{TC 02}, as a standard consequence of the spectral decomposition for quasi-compact semigroups.

        Moreover, the short-time smoothing bound
        \begin{align*}
            \|\partial^\beta \tilde{S}_\alpha(\tau)g\|_{m\alpha}\le C\,a(\tau)^{-\frac{|\beta|}{2\alpha}}\,\|g\|_{m\alpha}, \qquad a(\tau)=1-e^{-\tau},
        \end{align*}
        follows from the kernel representation of $S_\alpha(\tau)$ and the argument of Proposition \ref{group-est} (applied via a local or far-field decomposition, without imposing $p=q=2$).

        Finally, since $Q_n$ commutes with $S_\alpha(\tau)$, the semigroup property yields
        \begin{align*}
            \partial^\beta \tilde{S}_\alpha(\tau)\tilde{Q}_nf =\bigl(\partial^\beta \tilde{S}_\alpha(\tfrac{\tau}{2})\bigr)\,\bigl(\tilde{S}_\alpha(\tfrac{\tau}{2})\tilde{Q}_nf\bigr),
        \end{align*}
        and combining the previous two estimates gives \eqref{B.12}--\eqref{B.13}. The details are routine and therefore omitted.
    \end{proof}

    \section{Appendix C. Some useful result}
    Weighted $L^2\to L^q$ embedding in $\mathbb{R}^N$.

    \begin{proposition}\label{prop:L2mq_to_Lq}
        Let $N\ge1$ and assume $m,\alpha>0$. Fix $q\in(1,2]$. If
        \begin{equation}\label{eq:cond}
            \frac{2m\alpha\,q}{2-q}>N \quad\Longleftrightarrow\quad q>\frac{2N}{2m\alpha+N} \quad\Longleftrightarrow\quad m\alpha>N\!\left(\frac1q-\frac12\right),
        \end{equation}
        then there exists a constant $C=C(N,m,\alpha,q)>0$ such that
        \begin{equation}\label{eq:embedding}
            \|f\|_{L^q(\mathbb{R}^N)} \ \le\ C\,\|f\|_{m\alpha}\qquad \forall\,f\in L^{2}(m\alpha).
        \end{equation}
        In particular, the embedding $L^{2}(m\alpha)\hookrightarrow L^{q}(\mathbb{R}^N)$ is continuous.
    \end{proposition}
    
    \begin{lemma}\label{lem:bracket_integrability}
        For $\gamma>0$ one has
        \begin{align*}
            \int_{\mathbb{R}^N}(1+|x|^{2})^{-\gamma/2}\mathrm{d}x<\infty \quad\Longleftrightarrow\quad \gamma>N.
        \end{align*}
    \end{lemma}
    
    \begin{proof}
        Split $\mathbb{R}^N=B_1(0)\cup B_1(0)^c$. On $B_1(0)$ the integrand is bounded and the
        integral is finite. On $B_1(0)^c$ one has $(1+|x|^{2})^{1/2}\sim |x|$, hence by polar
        coordinates
        \begin{align*}
            \int_{B_1(0)^c}(1+|x|^{2})^{-\gamma/2}\mathrm{d}x\ \sim\ \int_1^\infty r^{N-1-\gamma}\mathrm{d}r,
        \end{align*}
        which converges iff $\gamma>N$.
        \end{proof}
    
        \begin{proof}[Proof of Proposition \ref{prop:L2mq_to_Lq}]
        Fix $q\in(1,2]$ and set
        \begin{align*}
            r=\frac{2}{q}\ \ge 1, \qquad r'=\frac{2}{2-q}\ \in[1,\infty], \qquad \frac1r+\frac1{r'}=1.
        \end{align*}
        Observe the algebraic factorization $|f|^q=\bigl(|f|^2 b^{2m\alpha}\bigr)^{q/2}\,b^{-m\alpha q}$. Applying H\"older's inequality with exponents $(r,r')$ yields
        \begin{align*}
            \int_{\mathbb{R}^N}|f|^q &\le \Bigl(\int_{\mathbb{R}^N}|f|^2 b^{2m\alpha}\Bigr)^{q/2}
            \Bigl(\int_{\mathbb{R}^N} b^{-m\alpha q r'}\Bigr)^{1/r'}\\
            &= \|f\|_{m\alpha}^{\,q}\, \Bigl(\int_{\mathbb{R}^N} (1+|x|^{2})^{-\gamma/2}\mathrm{d}x\Bigr)^{1/r'}, \qquad \gamma=m\alpha q r'=\frac{2m\alpha q}{2-q}.
        \end{align*}
        By Lemma \ref{lem:bracket_integrability}, the latter integral is finite iff $\gamma>N$, which is exactly condition \eqref{eq:cond}. Taking the $q$-th root gives \eqref{eq:embedding} with $C=C(N,m,\alpha,q)=(\int_{\mathbb{R}^N}(1+|x|^{2})^{-\gamma/2}\mathrm{d}x)^{\frac{1}{q}(1-\frac{q}{2})}$. This completes the proof.
    \end{proof}
 
    \begin{remark}
        (i) At the endpoint $q=2$ one trivially has $|f|_{2}\le\|f\|_{m\alpha}$, independently of $m,\alpha$.
        (ii) For $q=1$, condition \eqref{eq:cond} reads $2m\alpha>N$ (e.g. $m\alpha>1$ when $N=2$).
    \end{remark}

    Now, we will establish the estimate \eqref{f-esti}. According to \eqref{F-f}, for any $f,g\in \mathcal{S}'(\mathbb{R}^{N})$, the bilinear form is given by
    \begin{align*}
        \mathcal{E}(f,g) := (f, (-\Delta)^\alpha g) = \frac{C_{N,\alpha}}{2} \iint_{\mathbb{R}^{N}\times\mathbb{R}^{N}} \frac{(f(\xi)-f(\eta))(g(\xi)-g(\eta))}{|\xi-\eta|^{N+2\alpha}} \mathrm{d}\xi\mathrm{d}\eta.
    \end{align*}
    Setting $\psi(\xi)=|\xi|^{2m\alpha}$ and choosing $f=\psi w$ and $g=w$, we utilize the algebraic identity
    \begin{align*}
        \psi(\xi)w(\xi)-\psi(\eta)w(\eta) = \frac{\psi(\xi)+\psi(\eta)}{2}(w(\xi)-w(\eta)) + \frac{\psi(\xi)-\psi(\eta)}{2}(w(\xi)+w(\eta)),
    \end{align*}
    we obtain the decomposition
    \begin{align*}
        \int_{\mathbb{R}^{N}} |\xi|^{2m\alpha} w (-\Delta)^\alpha w \,\mathrm{d}\xi = \mathcal{D} + \mathcal{V},
    \end{align*}
    where
    \begin{align*}
        \mathcal{D} &= \frac{C_{N,\alpha}}{4} \iint_{\mathbb{R}^{N}\times \mathbb{R}^{N}} \frac{(\psi(\xi)+\psi(\eta))(w(\xi)-w(\eta))^2}{|\xi-\eta|^{N+2\alpha}} \mathrm{d}\xi\mathrm{d}\eta, \\
        \mathcal{V} &= \frac{C_{N,\alpha}}{4} \iint_{\mathbb{R}^{N}\times\mathbb{R}^{N}} \frac{(\psi(\xi)-\psi(\eta))(w^2(\xi)-w^2(\eta))}{|\xi-\eta|^{N+2\alpha}} \mathrm{d}\xi\mathrm{d}\eta.
    \end{align*}
    To estimate $\mathcal{D}$, we invoke the elementary inequality 
    \begin{align*}
        (\sqrt{A}a-\sqrt{B}b)^2 \le (A+B)(a-b)^2 + \frac{(A-B)^2}{A+B}(a^2+b^2), \quad \text{for } A,B>0.
    \end{align*}
    Setting $\tilde{w}(\xi) = \sqrt{\psi(\xi)}w(\xi) = |\xi|^{m\alpha}w(\xi)$, we deduce that
    \begin{align*}
        [\tilde{w}]_{W^{\alpha,2}}^2 \le C \mathcal{D} + C \mathcal{J}_1, 
    \end{align*}
    where the error term $\mathcal{J}_1$ is defined by
    \begin{align*}
        \mathcal{J}_1 = \iint_{\mathbb{R}^{N}\times\mathbb{R}^{N}} \frac{|\psi(\xi)-\psi(\eta)|^2}{\psi(\xi)+\psi(\eta)} \frac{|w(\xi)|^2+|w(\eta)|^2}{|\xi-\eta|^{N+2\alpha}} \mathrm{d}\xi\mathrm{d}\eta.
    \end{align*}
    For the term $\mathcal{V}$, applying Young's inequality yields $|\mathcal{V}| \le \varepsilon \mathcal{D} + C_\varepsilon \mathcal{J}_2$, where $\mathcal{J}_2$ is defined similarly to $\mathcal{J}_1$ but involves $|w(\xi)+w(\eta)|^2$. Noting that $\mathcal{J}_2$ shares the same singular structure as $\mathcal{J}_1$, it suffices to estimate $\mathcal{J}_1$.

    We split the domain of integration into the near-field $\mathcal{N} = \{|\xi-\eta|\le 1\}$ and the far-field $\mathcal{F} = \{|\xi-\eta|>1\}$. 

    In the region $\mathcal{N}$, using the gradient estimate $|\nabla \psi(\zeta)| \le C|\zeta|^{2m\alpha-1}$, we have
    \begin{align*}
        \frac{|\psi(\xi)-\psi(\eta)|^2}{\psi(\xi)+\psi(\eta)} \le C |\xi-\eta|^2 (|\xi|+|\eta|)^{2m\alpha-2}.
    \end{align*}
    Thus, the local part $\mathcal{J}_{1,\mathcal{N}}$ is bounded by
    \begin{align*}
        \mathcal{J}_{1,\mathcal{N}} \le C \int_{\mathbb{R}^N} |w(\xi)|^2 \left( \int_{|z|\le 1} \frac{(1+|\xi|)^{2m\alpha-2}}{|z|^{N+2\alpha-2}} \mathrm{d}z \right) \mathrm{d}\xi.
    \end{align*}
    Since $\alpha < 1$, the inner integral converges. By interpolation, for any $\varepsilon > 0$,
    \begin{align*}
        \mathcal{J}_{1,\mathcal{N}} \le C \int_{\mathbb{R}^N} (1+|\xi|^{2m\alpha-2})|w(\xi)|^2 \mathrm{d}\xi \le \varepsilon |\tilde{w}|_{2}^2 + C_\varepsilon |w|_{2}^2.
    \end{align*}
    In the region $\mathcal{F}$, using the trivial bound $|\psi(\xi)-\psi(\eta)| \le \psi(\xi)+\psi(\eta)$, we simply have
    \begin{align*}
        \mathcal{J}_{1,\mathcal{F}} \le C \int_{\mathbb{R}^N} \psi(\xi)|w(\xi)|^2 \left( \int_{|z|>1} \frac{1}{|z|^{N+2\alpha}} \mathrm{d}z \right) \mathrm{d}\xi \le C |\tilde{w}|_{2}^2.
    \end{align*}
    Combining these estimates yields the desired coercivity estimate \eqref{f-esti}.
\end{appendix}

 \section*{Acknowledge}
This work was supported by the National Natural Science Foundation of China Grant  No.12301297.\\
\section*{Conflict of Interest}
The authors declare that they have no conflict of interest.

\section*{Data Availability}
Data sharing not applicable to this article as no datasets were generated or analysed during the current study.

\end{document}